\documentclass{amsart}

\def\multiset#1#2{\ensuremath{\left(\kern-.3em\left(\genfrac{}{}{0pt}{}{#1}{#2}\right)\kern-.3em\right)}}
\theoremstyle{plain}
\newtheorem{theorem}{Theorem}[section]
\newtheorem{con}{Conjecture}[section]

\newtheorem{lemma}[theorem]{Lemma}

\usepackage{thmtools}
\usepackage{thm-restate}

\usepackage{hyperref}

\usepackage{cleveref}

\theoremstyle{definition}
\newtheorem{definition}{Definition}[section]

\theoremstyle{definition}

\usepackage{enumitem}
\begin{document}
\title{Pattern Avoidance in Set Partitions}
\author{Emma Christensen}

\begin{abstract}
A set partition avoids a pattern if no subdivision of that partition standardizes to the pattern. There exists a bijection between set partitions and restricted growth functions (RGFs) on which Wachs and White defined four statistics of interest to this work. We first characterize the restricted growth functions of several avoidance classes based on partitions of size four, enumerate these avoidance classes, and consider the distribution of the Wachs and White statistics across these avoidance classes.  We also investigate the equidistribution of statistics between avoidance classes based on multiple patterns. 
\end{abstract}
\maketitle
\section{Introduction}
Pattern avoidance in set partitions has become an increasingly popular area of research. In 2010 Bruce Sagan enumerated the avoidance classes for patterns of size three and also characterized the restricted growth functions in these classes \cite{Sag09}.
Not only are the sizes of different avoidance classes interesting, statistics on restricted growth functions in avoidance classes also offer a rich area of research. In 1991 Wachs and White defined four statistics of interest to this paper  on the words associated with set partitions \cite{WW91}. In a Research Experience for Undergraduates at Michigan State University in 2014, students investigated the distribution of the Wachs and White statistics over avoidance classes based on patterns of size three \cite{REU1}. This group also investigated multiple pattern avoidance focusing on two patterns of length three, or one pattern of length three and one of length four \cite{REU2}. Vit Jelinek, Toufik Mansour, and Mark Shattuck expanded the idea of pattern avoidance in 2013 by investigating avoidance classes determined by two patterns, especially the enumeration of (3,4)-pairs and the enumeration of (4,4)-pairs \cite{JelMan13}.
\par This paper will begin by considering avoidance classes based on one pattern of size four. We will first characterize the words in each avoidance class. Then using these word characterizations we will enumerate the avoidance classes of interest, develop equations for the Wachs and White statistics of each class, and finally compare the distributions of these statistics on avoidance classes. Following the investigation of the avoidance classes based on patterns of size four, we will consider classes avoiding multiple patterns. This focus on multiple pattern-avoiding classes of set partitions will continue to center around the investigation of the equidistribution of statistics. Finally this paper will introduce topics generated by this research which we will investigate in the future.
\par This paper is the result of the author's undergraduate thesis at the College of Saint Benedict/Saint John's University. The research was completed during the summers of 2016 and 2017. The author thanks Jennifer Galovich for her support in this work.
\section{Definitions and Notation}
A set partition of $S=\{1,2,...,n\}$ is a collection of disjoint subsets (or blocks) $B_1, B_2,...,B_m$ where $B_1\cup B_2\cup...\cup B_m=S$ and $$min(B_1)<min(B_2)<...<min(B_m)$$
A set partition, $\sigma$, of $S$ is denoted:  $$\sigma=B_1/B_2/.../B_m$$ 
The collection of all set partitions of $[n]$ is denoted:
$$\Pi_n=\{\sigma:\sigma\vdash [n]\}$$
\par Before defining avoidance and containment, one necessary structure to develop will be a subpartition of a set partition $\sigma$, denoted $\sigma'$. 
\definition When $S=\{1,...,n\}$ and $A\subseteq S$, then a subpartition $\sigma'$ of a partition $\sigma$ is defined as $\sigma'=A\cap\sigma$. 
\par For example, if $\sigma=1347/25/68\vdash [8]$ and $A=\{1,2,4,7\}$, then
$\sigma'=\{1,2,4,7\}\cap{1347/25/68}=147/2$. Our subpartition $\sigma'$ does not have consecutive elements in this example and $\sigma'$ does not appear to be an element of $\Pi_n$. To deal with a subpartition with nonconsecutive elements we standardize $\sigma'$. To do this we assign the smallest element to $1$, the next smallest to $2$, and so on. So $147/2$ standardizes to $134/2$. \par Standardizing subpartitions becomes useful when comparing a set partition to a pattern, where a pattern is another set partition under consideration. If a partition has at least one subpartition that standardizes to a given pattern $\pi$, then the partition is said to \textit{contain} $\pi$. If no subpartition standardizes to $\pi$, then the partition \textit{avoids} $\pi$.
\par Examples are helpful in understanding the basic concepts of pattern avoidance and containment. Let $\pi=134/2$ and $\sigma=1347/25/68$. Now consider $A=\{1,2,4,7\}$ such that $\sigma'=A\cap\sigma$ which standardizes to $\pi=134/2$. We can conclude that $\sigma$ contains $\pi$. Another helpful example is determining if $\sigma=12/3/469/578$ contains $\pi=13/2/4$. To decide if $\sigma$ avoids or contains $\pi$ we must consider every possible $\sigma'$ until either we find some $\sigma'$ that standardizes to $\pi$ or we show that no such $\sigma'$ exists. This is a long and time consuming process; however there is a much more elegant approach. Instead of considering every possible $\sigma'$, if we first find a basic form of partitions that always avoid $\pi$, this process will become much more efficient. For this example, $\sigma=12/3/469/578$ avoids $\pi=13/2/4$ because for $i,j,k\in\mathbb{N}$ there are no $a,b\in B_i$ where there exist $c\in B_j$ and $d\in B_k$ such that $a<c<b<d$.
\definition The collection of the set partitions that avoids a specific pattern, $\pi$, is called an avoidance class and is denoted:
$$\Pi_n(\pi)=\{\sigma\in\Pi_n:\sigma\mbox{ avoids }\pi\}$$
\par As stated in the introduction, we are not only interested in the types of elements in, and size of the avoidance class, we also want to determine the relationship between avoidance classes and Wachs and White statistics. However, before we can find Wachs and White statistics for partitions in avoidance classes, these set partitions must be converted into restricted growth functions. 
\definition[Restricted Growth Function]
If $w=a_1 a_2...a_n$ where $a_1=1$ and when $i\geq 2$, then $a_i\leq 1+\mbox{max}\{a_1,...,a_{i-1}\}$, then $w$ is a restricted growth function (RGF).

\par There is a well known correspondence between set partitions of $n$ and restricted growth functions. The letter assigned to each position is determined as follows: if $i$ is an element in $\sigma=B_1/B_2/.../B_m$ and $i\in B_j$, then $w_i=j$. The following example is a set partition labeled with the block number for each element and the corresponding word labeled with the position of each letter.
$$\sigma=\underset{1}{1}/\underset{2}{2}\underset{2}{3}\underset{2}{5}/\underset{3}{4}\underset{3}{6}$$
$$w(\sigma)=\underset{1}{1}\underset{2}{2}\underset{3}{2}\underset{4}{3}\underset{5}{2}\underset{6}{3}$$
The collection of RGFs of length $n$ is denoted:
$$R_n=\{w:w\mbox{ is an RGF of length }n\}$$
The collection of RGFs associated with the avoidance class for a pattern $\pi$ is denoted:
$$R_n(\pi)=\{w(\sigma):\sigma\in\Pi_n(\pi)\}$$.
\par It will be useful to note that we say that $w=1^{a_1}2^{a_2}...m^{a_m}$ when $w$ is weakly increasing. When $w$ of type $1^{a_1}2^{a_2}...m^{a_m}$ this simply requires $w$ to have $a_i$ $i$s for every $1\leq i\leq m$. 
\par Because our statistics of interest are defined on the RGFs, and not on set partitions, one of the goals of this study is to characterize the words in $R_n(\pi)$ for $|\pi|=4$ and for multiple patterns. Some work in this area has already been completed and will be useful in characterizing our avoidance classes.
\begin{lemma}[Word Characterizations based on patterns of size $3$]\cite{Sag09}
\begin{itemize}
	\item[1.]$R_n(1/2/3)=\{w\in R_n: w$ consists of only $1$s and $2$s$\}$
    \item[2.]$R_n(1/23)=\{w\in R_n: w$ is obtained by inserting a single $1$ into a word of the form $1^l23...m$ for some $l\geq0$ and $m\geq1\}$
    \item[3.]$R_n(13/2)=\{w\in R_n: w$ is layered$\}$
    \item[4.]$R_n(12/3)=\{w\in R_n: w$ has initial run $1...m$ and $a_{m+1}=...=a_n\leq m\}$
    \item[5.]$R_n(123)=\{w\in R_n: w$ has no element repeated more than twice$\}$
    
\end{itemize}
\end{lemma}
Another goal will be to enumerate the avoidance classes. The following theorem is a helpful result:
\begin{lemma}\cite{Sag09} If $w\in R_n$ is weakly increasing with $m$ distinct letters, then there is a total of $\sum_m {{n-1}\choose{n-m}}$ possibilities for $w$.
\end{lemma}
\begin{proof} If $w$ is weakly increasing, we are not concerned with order when enumerating. Then there are $m$ choices for each letter and $n-m$ letters not assigned. So for each $m$ there are ${{m+(n-m)-1}\choose{n-m}}={{n-1}\choose{n-m}}$ weakly increasing words of length $n$ with $m$ distinct letters.
\end{proof}
There are several statistics of interest to this paper based on the word given by the restricted growth function. Once we have a characterization for the words in an avoidance class we can determine the distribution of the statistics across avoidance classes. This paper will focus on the four Wachs and White statistics, where the letters $l,r,s,$ and $b$ indicate left, right, smaller, and bigger respectively \cite{WW91}. Before defining the Wachs and White statistics across entire words, however, we will begin by defining the contributions to the statistics for individual letters:

\definition For a position $i\in [n]$, the associated letter $w_i$ will contribute the following to each statistic:
$$rs_i(w)=|\{w_j<w_i:j>i\}|$$
$$ls_i(w)=|\{w_j<w_i:j<i\}|$$
$$lb_i(w)=|\{w_j>w_i:j<i\}|$$
$$rb_i(w)=|\{w_j>w_i:j>i\}|$$

\par For example, consider $w=1123124255$, then $ls_4(w)=2$ because we are only counting distinct terms to the left and smaller than $w_4=3$. Now we will consider the statistics across an entire word.
\definition[Wachs and White Statistics]\cite{WW91} If $w\in R_n$, then
$$rs(w)=\sum_{i=1}^nrs_i(w)$$
$$rb(w)=\sum_{i=1}^nrb_i(w)$$
$$lb(w)=\sum_{i=1}^nlb_i(w)$$
$$rs(w)=\sum_{i=1}^nrs_i(w)$$

\par So the right and smaller statistic counts the number of distinct letters to the right and smaller of each letter in $w$. To see this more clearly, consider the word, $w$, labeled with $rs_i(w)$ for all $i\in[n]$:
$$w\underset{rs_i(w):}{=}\underset{0}{1}\underset{1}{2}\underset{2}{3}\underset{1}{2}\underset{3}{4}\underset{2}{3}\underset{0}{1}\underset{0}{2}\underset{0}{5}$$ From the definition of the $rs$ statistic we know $rs(w)$ is simply the sum of these individual contributions $rs_i(w)$ for every $i\in n$, so
$$rs(w)=0+1+2+1+3+2+0+0+0=9$$
\par Next consider the $lb$ statistic for $w$, we first label each letter below with $lb_i(w)$ for all $i\in [n]$ and then find $lb(w)$.
$$w\underset{lb_i(w):}{=}\underset{0}{1}\underset{0}{2}\underset{0}{3}\underset{1}{2}\underset{0}{4}\underset{1}{3}\underset{3}{1}\underset{2}{2}\underset{0}{5}$$
$$lb(w)=0+0+0+1+0+1+3+2+0=7$$
The two other Wachs and White statistics for $w$ are:
$$rb(w)=4+3+2+3+1+1+2+1+0=17$$
$$ls(w)=0+1+2+1+3+2+0+1+4=14$$
\par One of the intentions of this research is to find equations for the four Wachs and White statistics on specific avoidance classes. Before focusing on specific classes however, it will be useful in this work to recognize that there is a simple formula for the left and smaller statistic of any word.
\begin{lemma} If $w\in R_n$ such that $w$ is of type $1^{a_1}2^{a_2}...m^{a_m}$, then
$$ls(w)=\sum\limits_{i=1}^m ia_i-n$$
\end{lemma}
\begin{proof}
We claim that for each letter, $i\in w$, there will be $i-1$ distinct elements that are smaller and occur before $i$. Suppose that $w$ is a word of type $1^{a_1}2^{a_2}...m^{a_m}$. 
\par Now consider some $j\in w$ where $w$ does not have all of the integers that occur between $1$ and $j$ in some order in the word before $a_j$. Then for some $k$ where $k<j$ and $k$ first occurs after $j$, $k+1\leq j$ then $j\not\leq1+max\{1,...,j-1\}$. Therefore no element can appear without all integers before that element also appearing in the word by our definition of RGF. Because $w$ will have $a_i$ occurrences of $i$, for every $i$, there will be $a_i$ letters, each contributing $i-1$ to the $ls$ statistic. Then by the definition of $ls$,
$$ls(w)=\sum\limits_{i=1}^m a_i(i-1)$$
$$    \hfill=\sum\limits_{i=1}^m ia_i-n$$
Therefore when $w\in R_n$ and has $m$ distinct letters, $ls(w)=\sum\limits_{i=1}^m ia_i-n$.
\end{proof}
There is a simple parallel equation for the $rb$ statistic of a word, however this formula only applies to weakly increasing words.
\begin{lemma} If $w\in R_n$  where $w$ is weakly increasing such that $w=1^{a_1}2^{a_2}...m^{a_m}$, then
$$rb(w)=\sum_{i=1}^m (m-i)a_i$$
\end{lemma}
\begin{proof} Let $w$ be a weakly increasing word where $w\in R_n$ such that $w=1^{a_1}2^{a_2}...m^{a_m}$, and let $b$ be a position in $w$ where $w_b=c$ with $a_c$ occurrences of $c$ in $w$. Then $rb_b(w)=(m-c)$. There will be $a_c$ of these letters contributing the same value for $rb_i(w)$ to the $rb$ statistic, therefore $rb(w)=\sum_{i=1}^m (m-i)a_i$.
\end{proof}
As we investigate these statistics across avoidance classes, we first identify avoidance classes of the same size. To begin, we use word characterizations to help us enumerate the avoidance classes. The complement of a set partition, defined below, is helpful in enumerating avoidance classes. 
\definition We define the \emph{complement} of $\sigma$ to be the set partition $\sigma^c$ where every element $l\in\sigma$ becomes $(n-l+1)\in \sigma^c$.
\par For example, let $\sigma=1246/37/5$. Then $\sigma^c=15/2467/3$.
\par From \cite{Sag09}, we have the following theorem:
\begin{restatable}[Cardinality of the Avoidance Class of the Complement of $\pi$]{thm}{Comp} 
\label{thm:comp}
If $\pi\in\Pi_k$ and $\pi^c$ is the complement of $\pi$, then $\#\Pi_n(\pi)=\#\Pi_n(\pi^c)$.
\end{restatable}
\definition If $\pi_1,\pi_2\in\Pi_n$ such that $\#\Pi_n(\pi_1)=\#\Pi_n(\pi_2)$, then $\pi_1$ is \emph{Wilf-equivalent} to $\pi_2$. We write $\pi_1\equiv\pi_2$. 
\par  Before considering whether any statistic is equidistributed on two avoidance classes, we first determine if the two classes are Wilf-equivalent. By Theorem 2.2, for all $\sigma$ and $\sigma^c$, $\sigma\equiv\sigma^c$. This reduces the number of cases to consider when determining the equidistribution of statistics between avoidance classes. 
\par Here is another important lemma to note that will help us in our word characterizations for the avoidance classes.
\begin{lemma} If $w(\sigma)$ is weakly increasing, then $w(\sigma^c)$ is also weakly increasing.
\end{lemma}
\begin{proof}
 Let $\sigma\in\Pi_n$ where $\sigma$ has $m$ blocks and $w(\sigma)$ is weakly increasing. Then $w(\sigma)=1^{a_1}2^{a_2}...m^{a_m}$ for some $a_1,a_2,...,a_m$ such that $\sum_{i=1}^ma_i=n$. By the definition of RGF, we know that $\sigma=1...a_1/.../(\sum_{i=1}^{m-1}a_{i}+1)...n$, where each block contains only consecutive terms. Then $\sigma^c=(n+1-1)...(n+1-a_1)/.../(n+1-\sum_{i=1}^{m-1}a_{i}-1)...1$, which we can simplify to $\sigma^c=1...a_m/.../(\sum_{i=2}^ma_i+1)...n$ where each block contains consecutive terms. Therefore by the definition of RGF we know that $w(\sigma^c)=1^{a_m}2^{a_{m-1}}...m^{a_1}$. So $w(\sigma^c)$ is weakly increasing when $w(\sigma)$ is weakly increasing.        
\end{proof}
\section{Word Characterizations and Cardinalities of Single Pattern Avoidance Classes}
\par There are $15$ set partitions of size four. Sagan enumerated and found the word characterization for the avoidance classes based on $1/2/3/4$ and $1234$ \cite{Sag09}. Following this the Research Experience for Undergraduates at Michigan State University found the cardinality and word characterizations of avoidance classes based on $14/2/3$, $13/24$, and $14/23$ \cite{REU2}. This group also investigated the Wachs and White statistics for these avoidance classes. 
\par In this paper we will enumerate, characterize and investigate the equidistribution of statistics for the avoidance classes based on $12/3/4$, $1/2/34$, $1/234$, $123/4$, $134/2$, $124/3$, $13/2/4$, and $1/24/3$. We chose these avoidance classes because they have not yet been investigated. We did not get the chance to investigate the avoidance classes based on the patterns $12/34$ and $1/23/4$. After finding the word characterizations and the cardinality for each of these eight avoidance classes we find which classes are Wilf-Equivalent. Recognizing these Wilf-Equivalent classes will be important in the study of the distribution of the Wachs and White Statistics. Before explaining our word characterization and cardinality results, we will begin with two avoidance classes with some known results \cite{Sag09}.

	\begin{restatable}{thm}{SagOne}
    \label{thm:SagOne}
    \cite{Sag09} The word characterization and cardinality of the avoidance class based on the pattern $1/2/3/4$ are as follows:
    \begin{itemize}
    	\item[(i)]$R_n(1/2/3/4)=\{w\in R_n$: $w$ consists of only 1's, 2's, and 3's$\}$       
        \item[(ii)]$\#\Pi_n (1/2/3/4) = 2^{n-1}+S(n,3)$   
        \end{itemize}
    \end{restatable}
    \begin{proof} 
    \begin{itemize}
    	\item[(i)] The word characterization is due to Sagan \cite{Sag09}.
        \item[(ii)] Let $\sigma\in\Pi(1/2/3/4)$. We know that if $\sigma$ avoids $1/2/3/4$, $\sigma$ can have at most three blocks. By the definition of the Stirling numbers of the second kind there are $S(n,1)=1$ ways for $\sigma$ to have one block, $S(n,2)=2^{n-1}-1$ ways for $\sigma$ to have two blocks, and $S(n,3)$ ways for $\sigma$ to have three blocks. Therefore $\#\Pi_n (1/2/3/4) = 2^{n-1}+S(n,3)$.
    \end{itemize} 
    \end{proof}
	\begin{restatable}{thm}{SagTwo}
    \label{SagTwo}
    \cite{Sag09} The word characterization and cardinality of the avoidance class based on the pattern $1234$ are as follows:
		\begin{itemize}
        \item[(i)]$R_n(1234)=\{w\in R_n$: $w$ has no letter repeated more than three times$\}$
     
        	\item[(ii)]When $n\leq 12$, $\#\Pi_n (1234)=1+{n\choose 2}+S(n,n-2)+\sum\limits_{i=3}^n S(n,n-i)-\sum\limits_{j=4}^{i+1} {n\choose j}S(n-j,n-i-1)-\sum\limits^{n-j-1}_{k=4}{n-j \choose k}S(n-j-k,n-i-2)$ 
            \end{itemize}
    \end{restatable}

\begin{proof} 
	\begin{itemize} 
    	\item[(i)] The word characterization is due to Sagan \cite{Sag09}.
        \item[(ii)] From the result for $R_n(1234)$ we know that if $\sigma$ avoids $1234$, $\sigma$ cannot have more than three elements in the same block. If there are $n$ blocks, there can only be one element in each block and there are $S(n,n)=1$ ways to have $n$ blocks. If there are $n-1$ blocks, there are $n-2$ blocks that have have one element and one block with two elements and there are ${n\choose 2}$ ways to do this. When there are $n-2$ blocks, there are at most three items in one block and there are $S(n,n-2)$ ways to do this. When there are fewer than $n-2$ blocks, there is the potential for more than three elements in one block, which would allow any $\sigma$ of this type to contain $1234$. The total number of ways to partition the set $\{1,...,n\}$ into $n-i$ blocks where $i$ can be anything larger than two will be $S(n, n-i)$. Now we must remove the partitions which allow more than three elements in the same block. Whenever a block has four or more elements it will standardize to $1234$. To make a block of four or more we must choose four or more elements. When $j$ is the number of elements in the block, there are ${n\choose j}$ ways to choose the elements in this block. The other elements are not in this block, so there are $(m-1)$ blocks to choose from where $m=n-i$ and $n-j$ elements left to assign to these $n-i-1$ blocks. Therefore there are ${n\choose j}S(n-j,n-i-1)$ ways for the set partition to have at least one block of four or more. However this overcounts the number of possible partitions with four or more elements in one block because it also will count when more than one block has four or more elements, so to account for those set partitions that were removed more than once we must add back the cases with two different blocks each with four or more elements. There are $(n-j)$ elements to choose from and $k$ elements to choose. Then there are $n-j-k$ other elements to place into $n-i-2$ blocks because the first block and the second block that we found both had more than four elements so there are $S(n-j-k,n-i-2)$ ways of assigning the other elements to blocks. This formula will work up to $n=12$ because there can be at most three blocks with four elements and once the first two blocks of four are found, there will only be one option for the third, after this, as $n$ becomes larger we will have to consider when three or more blocks are larger than four where the last block is not necessarily determined by the other two blocks already found.
    \end{itemize} 
    \end{proof}
Next we consider avoidance classes that have not been investigated. The proofs of these theorems will be very similar. Therefore, we will present the proof of Theorem 3.3 in the body of this paper and refer the reader to the appendix for complete proofs of Theorems 3.4 through 3.10.
	\par We first consider avoidance classes based on the pattern $12/3/4$. It is easier to begin by considering an example of words found in $R_n(12/3/4)$. It is simple to see that the set partition $168/2/37/4/5$ avoids $12/3/4$ because any block that contains two or more elements has at most one other block with a larger element. Therefore $12345131\in R_n(12/3/4)$. We can describe these words found in $R_n(12/3/4)$ more generally in the following theorem.
	\begin{restatable}{thm}{EmOne}
    \label{thm:EmOne}
    The word characterization and cardinality of the avoidance class based on the pattern $12/3/4$ are given by:
    \begin{itemize}
        \item[i.]$R_n(12/3/4)=\{w\in R_n:$ \begin{enumerate}[label=(\Alph*)]
        	\item $m\leq 2$, or
            
            \item $m\geq 3$ and $w=uv$ such that $u$ is strictly increasing, and $v$ begins at the first repeated letter. The suffix $v$ contains at most two distinct letters ($v$ can be empty).$\}$
        \end{enumerate}
        \item[ii.]$\#\Pi_n (12/3/4)=2^{n-1}+n(n-2)-\frac{(n-1)(n-2)}{2}+\sum\limits^{n-2}_{|v|=2}\sum\limits_{j=1}^{i-1}{n-|v|\choose 2}{|v|\choose j}+\sum\limits^{n-2}_{|v|=2}\sum\limits_{j=1}^{i-1}(n-i){i-1\choose j}$
        \end{itemize}
	\end{restatable}

    \begin{proof} \begin{itemize} \item[(i)] Word Characterization:
    \begin{enumerate}[label=(Case \Alph*)]
    			\item Let $w(\sigma)$ have no letter larger than $2$. Then $\sigma$ can only have one or two blocks. In general if $\sigma$ has $k$ blocks it must avoid any pattern with more than $k$ blocks. Therefore $\sigma\in\Pi_n(12/3/4)$ and $w(\sigma)\in R_n(12/3/4)$. 
            	\item Let $w(\sigma)$ be of the form $uv$ such that $u$ is strictly increasing, and $v$ begins at the first repeated letter. When $u$ is strictly increasing, then $u$ has no repeated letter. Then because $v$ can have at most two distinct letters, there are at most two blocks that can have two or more elements in $\sigma$, by the definition of an RGF. The repeated elements in $w$ are only in $v$ and there can be at most two distinct letters in $v$. So when there is a block with two elements (a and b), then $\sigma$ can have at most one distinct block that contains elements that are larger than a and b. Therefore no subpartition of $\sigma$ will standardize to $12/3/4$ because there must be at least two larger elements in different blocks after a block of size two or larger. Therefore $\sigma$ will avoid the pattern.
                      \end{enumerate}  
                \par Next suppose $w(\sigma)$ is not of form $(A)$ or $(B)$. Then $m\geq 3$ and $v$ must contain more than two distinct letters. There is no way to violate the strictly increasing property of $u$ because $v$ must begin at the first repeated term. 
                \par Then using the definition of the restricted growth function, consider the subpartition constructed by taking the element $a$ where $a$ is the first letter in $v$, $b$ the other element that occurs in the same block as $a$, along with $c$ the element associated with the second distinct letter in $v$, and $d$ the element associated with the third distinct letter in $v$. This subpartition will have three blocks: there will be a block of the two smaller elements ($a$ and $b$), followed by a block containing the second largest element, and then a block with the largest element. So the subpartition is $ba/c/d$ where $b<a<c<d$ which therefore standardizes to $12/3/4$. For example if $v(\sigma)=12123$ then $\sigma=13/24/5$ and one subpartition created through this process would be $\sigma'=13/4/5$. It is clear that in this example $\sigma'$ standardizes to $12/3/4$. 
                
                \par Therefore when $\sigma$ avoids $12/3/4$ then $w(\sigma)$ has either form (A) or (B).
                    
            \item[(ii)] Next we consider the cardinality of the avoidance class. By the word characterization there are two main cases to consider:
            \begin{itemize} 
            	\item[(Case A)] Let $m\leq 2$. There are $S(n,2)=1+(2^{n-1}-1)=2^{n-1}$ words in (Case A) by the definition of the Stirling numbers of the second kind. 
                \item[(Case B)]Let $w=uv$ as defined in the proof of (Case B) of the word characterization. Now we have to consider three possible forms for $v$.
                	\begin{enumerate} 
                    \item Let $v$ have at most one  letter. Then the length of $v$, denoted as $|v|$, can be from $0$ to $n-3$ because $m\geq 3$. When $|v|=0$ there is one form for $w$; $w$ is strictly increasing. When $|v|>0$, there are $m$ choices for the letter in $v$, And $m=n-|v|$ because $m$ is found in $u$ in this case. So there are $1+\sum\limits^{n-3}_{|v|=1}(n-|v|)$ possible $w$ of this form.
                \item Let $v$ be of type $p^{j}q^{i}$ where $p>q$ and $p,q\in u$. There are ${n-|v|\choose 2}$ ways to choose $p,q$. There are $|v|$ positions in $v$ to choose from and $j$ $p$'s in $v$, so there are ${|v|\choose j}$ ways to arrange $v$ where $1\leq j\leq |v-1|$ because there must be two letters present in $v$. Therefore there are $\sum\limits^{n-3}_{|v|=2}\sum\limits_{j=1}^{|v|-1}{n-|v|\choose 2}{|v|\choose j}$ possible $w$ of this form.
                \item  Let $v$ be of the form $p^{j}m^{i}$ where $p$ also occurs in $u$. The first letter in $v$ must be $p$ because $v$ begins with the repeated letter. Because there are $n-|v|$ letters in $u$, there are $n-|v|$ options for this first letter in $v$. We know already that the other letter in $v$ must be $m$, and that this $m$ occurs $j$ times. There are $|v|-1$ options for where these $j$ $m's$ can occur; because $v$ must begin with the first repetition, it must begin with the smaller letter in $v$, then any other position in $v$ is available. So there are  ${|v|-1\choose j}$ ways of arranging $v$ in this case. Therefore there are $\sum\limits^{n-2}_{|v|=2}\sum\limits_{j=1}^{|v|-1}(n-|v|){|v|-1\choose j}$ possible $w$ of this form.
                \end{enumerate}
            \end{itemize}
            Therefore $\#\Pi_n (12/3/4)=1+2^{n-1}+\sum\limits^{n-3}_{|v|=1}(n-|v|)+\sum\limits^{n-3}_{|v|=2}\sum\limits_{j=1}^{|v|-1}{n-|v|\choose 2}{|v|\choose j}+\sum\limits^{n-2}_{|v|=2}\sum\limits_{j=1}^{|v|-1}(n-|v|){|v|-1\choose j}$.
            \end{itemize}
            \end{proof}

        \begin{restatable}{thm}{EmTwo}
        \label{thm:EmTwo}
        The word characterization and cardinality of the avoidance class based on the pattern $1/2/34$ are given by:
        \begin{itemize}
        \item[(i)]$R_n(1/2/34)=\{w\in R_n:$ 
        	\begin{enumerate}[label=(\Alph*)] 
            \item $m\leq 2$
            \item $w$ is a word of the form $u3v$ where u is a prefix of type $1^{a_1}2^{a_2}$, and $v$ contains $b_1$ $1$s, $b_2$ $2$s, $b_3$ $4$s, and so on such that $b_i\leq 1\}$
            \end{enumerate}
      \item[(ii)]$\#\Pi_n (12/3/4)=2^{n-1}+n(n-2)-\frac{(n-1)(n-2)}{2}+\sum\limits^{n-2}_{|v|=2}\sum\limits_{j=1}^{i-1}{n-|v|\choose 2}{|v|\choose j}+\sum\limits^{n-2}_{|v|=2}\sum\limits_{j=1}^{i-1}(n-i){i-1\choose j}$
      \end{itemize}
      \end{restatable}
	\begin{restatable}{thm}{EmThree}
    \label{thm:EmThree}
    The word characterization and cardinality of the avoidance class based on the pattern $1/234$ are given by:
    \begin{itemize}
        \item[(i)]$R_n(1/234)=\{w\in R_n: w$ is of the form $uv$ such that $u$ is of the form $1^{a_1}2$ and $v$ has $b_1$ $1$s, $b_2$ $2$s, and so on, such that $b_1, b_3,...,b_m\leq 2$ and $b_2\leq 1\}$
        \item[(ii)]$\#\Pi_n (1/234)=1+\sum\limits_{r=1}^{\lfloor\frac{n}{2}\rfloor}\prod\limits_{j=0}^{r-1}{n-2j\choose 2}+ \sum\limits_{i=1}^{n-2}\sum\limits_{r=1}^{\lfloor\frac{n-i}{2}\rfloor}\prod\limits_{j=0}^{r-1}r{n-2j-i\choose 2}$
        \end{itemize}
	\end{restatable}
	\begin{restatable}{thm}{EmFour}
    \label{thm:EmFour}
    The word characterization and cardinality of the avoidance class based on the pattern $123/4$ are given by:
    \begin{itemize}
        \item[(i)]$R_n(123/4)=\{w\in R_n$: $w$ is of the form $uv$ where $u$ has $a_1$ $1$s, $a_2$ $2$s, ..., $a_m$ $m$s such that $a_1,...,a_{m}\leq 2$ and $v$ can be any size of one letter repeated$\}$
        \item[(ii)]$\#\Pi_n (123/4)=1+\sum\limits_{r=1}^{\lfloor\frac{n}{2}\rfloor}\prod\limits_{j=0}^{r-1}{n-2j\choose 2}+ \sum\limits_{i=1}^{n-2}\sum\limits_{r=1}^{\lfloor\frac{n-i}{2}\rfloor}\prod\limits_{j=0}^{r-1}r{n-2j-i\choose 2}$
        \end{itemize}
    \end{restatable}

Before characterizing and enumerating the following two avoidance classes it will be useful to define three special types of letters that will aid in describing the word characterizations.
\definition[Block Letter] We define a block letter as a letter $a\in w$ such that the only occurrence of $a$ is in a block of $a$s where this block contains one or more $a$s.
\definition[Block-Singleton Letter] We define a block-singleton letter as a letter $a\in w$ where the first occurrence of $a\in w$ is the beginning of a block of $a$s, and after some distinct letter there is exactly one more occurrence of $a\in w$.

\definition[Singleton-Block Letter] We define a singleton-block letter as a letter $b\in w$ where the first occurrence of $b\in w$ is followed by a distinct letter, and the only other occurrences of $b\in w$ are found in a block after this.
\par To clarify these definitions we include the following examples:
\begin{itemize}
\item[a.]The block letters of the word $w_1=1112234334354$ are $1$, $2$, and $5$.
 \item[b.]The block-singleton letters of $w_2=1233322443124$ are $1$, $3$ and $4$
    \item[c.] The singleton-block letters of $w_3=122311123$ are $1$ and $3$.
\end{itemize}
\begin{lemma}[Sinlgeton-Block/Block-Singleton Avoidance] If the RGF for $\sigma$ contains a singleton-block or a block-singleton letter and $w$ is weakly increasing, then $w\in R_n(\sigma)$.
\end{lemma}
	\begin{restatable}{thm}{EmFive}
    \label{thm:EmFive}
    The word characterization  of the avoidance class based on the pattern $134/2$ is:
    \begin{itemize}
		\item $R_n(134/2)=\{w\in R_n:$ every letter in $w$ is either a block letter or a block-singleton letter$\}$
	\end{itemize}
	\end{restatable}

\begin{restatable}{thm}{EmSix}
\label{thm:EmSix}
The word characterization of the avoidance class based on the pattern $124/3$ is:
\begin{itemize}
\item $R_n(124/3)=\{w\in R_n:$ every letter in $w$ is either a block letter or a singleton-block letter$\}$
\end{itemize}
\end{restatable}
	\begin{restatable}{thm}{EmSeven}
    \label{thm:EmSeven}
    The word characterization and cardinality of the avoidance class based on the pattern $13/2/4$ are given by:
    \begin{itemize}
    	\item[(i)]$R_n(13/2/4)=\{w\in R_n$: 
        \begin{enumerate}[label=(\Alph*)] 
        \item $m\leq 2$, or
        \item $m>3$ and $w=uv$ where $u$ is weakly increasing and $v$ begins at the first $m$ and contains $m$ and $(m-1)$ in any order, or
        \item $m>3$ and $w=uv$ where $u$ is weakly increasing and $v$ begins at the first $m$ and contains a block of $m$s followed by a block of $z$s where $z<(m-1)\}$
        \end{enumerate}
        \item[(ii)]$\#\Pi_n (13/2/4)=2^{n-1}+\sum\limits^n_{m=3}{n-1 \choose n-m}$$ $$+\sum\limits^{n-1}_{m=3}\sum\limits^{n-2}_{|u|=m-1}{|u|-1 \choose |u|-m+1}[(2^{(n-|u|-1)}-1)+(m-2)(n-|u|-1)]$
    \end{itemize}
    \end{restatable}

	\begin{restatable}{thm}{EmEight}
    \label{thm:EmEight}
    The word characterization and cardinality of the avoidance class based on the pattern $1/24/3$ are given by:
    \begin{itemize}
    	\item[(i)]$R_n(1/24/3)=\{w\in R_n:$
        \begin{enumerate}[label=(\Alph*)]
        \item $m\leq 2$, or
        \item $m>3$ and $w$ is a weakly increasing word with a block of $1$s of any size inserted between two distinct terms, or
        \item $m>3$ and $w$ is a word beginning with $1$s and $2$s in any order followed with a weakly increasing suffix that begins with the first occurrence of $3\}$
        \end{enumerate}
        \item[(ii)]$\#\Pi_n (1/24/3)=2^{n-1}+\sum\limits^n_{m=3}{n-1 \choose n-m}$$ $$+\sum\limits^{n-1}_{m=3}\sum\limits^{n-2}_{|u|=m-1}{|u|-1 \choose |u|-m+1}[(2^{(n-|u|-1)}-1)+(m-2)(n-|u|-1)]$
        \end{itemize}
   \end{restatable}
This completes our study of the word characterization and cardinality for the avoidance classes: $R_n(1/2/3/4)$, $R_n(1234)$, $R_n(12/3/4)$, $R_n(1/2/34)$, $R_n(1/234)$, $R_n(123/4)$, $R_n(13/2/4)$, $R_n(1/24/3)$, $R_n(134/2)$, and $R_n(124/3)$. Next we will investigate the Wachs and White statistics. We found especially promising results for the distribution of Wachs and White statistics for $R_n(\sigma)$ and $R_n(\sigma^c)$.
\section{Characterization of Statistics across Single Pattern Avoidance Classes}
Before considering the distribution of statistics across avoidance classes it will be helpful to obtain formulas for the Wachs and White statistics for individual words based on the avoidance class where each word can be found. Theorems 4.1 through 4.7 describe the statistics for words in certain avoidance classes that showed promising equidistribution results with other avoidance classes. The proofs for these equations for statistics are dependent on the word characterizations and can be found in the appendix. We begin with the results of the Michigan State REU for statistics when $m\leq 2$.
\begin{restatable}{thm}{wwREU}
\label{thm:wwREU}
\cite{REU1} If $m\leq 2$ and $l$ is the number of $1$s present in $w$, then we have the following equations for the Wachs and White statistics:

\[
 lb(w) =
  \begin{cases} 
       0 & \text{$w$ weakly increasing}\\
      l-i & \text{$w=1^iw'1^j2^k$}
  \end{cases}
  \]

\hspace{2.25cm} $ls(w) = n-l$
\[
 rb(w) =
  \begin{cases} 
      l   & \text{$w$ weakly increasing} \\
       l-\delta_{k,0}j & \text{$w=1^iw'1^j2^k$}
  \end{cases}
\]
\[
 rs(w) =
  \begin{cases} 
      0  & \text{$w$ weakly increasing} \\
      n-l-k & \text{$w=1^iw'1^j2^k$}
  \end{cases}
  \]

\end{restatable}
Now we will use previous theorems to give more useful characterizations of the words to find equations for the Wachs and White statistics. For the $ls$ and $lb$ statistics for $w\in R_n(1/2/3/4)$ it is helpful to restate the word characterization. By Theorem 3.1 we know that for all $w\in R_n(1/2/3/4)$, $m\leq 3$, so we can restate this as, 
\\$R_n(1/2/3/4)=\{w\in R_n:$\begin{itemize}
		\item[(A) ] $m\leq 3$ and $w=1^{h_1}2^{h_2}3^{h_3}$ (meaning $w$ is weakly increasing), or
       \item[(B) ] $1<m\leq 3$ and $w=1^{v}{w_a}1^{x}{w_b}1^{y}{w_c}1^{z}$ where $w_a$ begins and ends with a 2 and has $a_1$ $1$s and $a_2$ $2$s, $w_b$ begins with the first $3$ and ends with the last $2$, and $w_c$ begins with a $3$ and ends with the last $3$, or
        \item[(C) ] $1<m\leq 3$ and $w=1^{w}{w_a}1^{x}{w_b}1^{y}{w_c}1^{z}$ where $w_1$ begins and ends with a 2 and has $a_1$ $1$s and $a_2$ $2$s, $w_b$ begins with the first $3$ and ends with the last $3$, and $w_c$ begins with a $2$ and ends with the last $2\}$
		\end{itemize}
\begin{restatable}[Left Statistics for $R_n(1/2/3/4)$]{thm}{lsww}  
\label{thm:lsww}
If $w\in R_n(1/2/3/4)$, then

\[
 ls(w) =
  \begin{cases} 
      \hfill 2h_3+h_2    \hfill & \text{$w\in A$} \\
      \hfill a_2+b_2+2b_3+c_2+2c_3 \hfill & \text{$w\in B$ or $C$} \\
  \end{cases}
\]
\[
 lb(w) =
  \begin{cases} 
      \hfill 0    \hfill & \text{$w\in A$} \\
      \hfill x+2y+2z+a_1+2b_1+2c_1 \hfill & \text{$w\in B$ or $C$} \\
  \end{cases}
  \]
\end{restatable}
\par The $rb$ and $rs$ statistics for $w\in R_n(1/2/3/4)$ are more easily described when we restate the word characterization for this avoidance class as follows:
\\$R_n(1/2/3/4)=\{w\in R_n:$
	\begin{enumerate}[label=(\Alph*)]
    \item $m\leq 3$ and $w=1^{h_1}2^{h_2}3^{h_3}$
    \item $w=w_aw_b3^c$ where $w_a$ ends with the last $1$; $w_b$ ends with the last $2$
    \item $w=w_aw_b3^c$ where $w_a$ ends with the last $2$, $w_b$ ends with the last $1$
    
    \item $w=w_aw_b2^c$ where $w_a$ ends with the last $3$, $w_b$ ends with the last $1$
    \item $w=w_aw_b2^c$ where $w_a$ ends with the last $1$, $w_b$ ends with the last $3$
    \item $w=w_aw_b1^c$ where $w_a$ ends with the last $2$, $w_b$ ends with the last $3$
    \item $w=w_aw_b1^c$ where $w_a$ ends with the last $3$, $w_b$ ends with the last $2$

    \item $m=2$ and $w=w_a2^b$ where $w_a$ ends with the last $1$
    \item $m=2$ and $w=w_a1^b$ where $w_a$ ends with the last $2\}$
    \end{enumerate}
    where $w_a$ contains $a_1$ $1$s, $a_2$ $2$s, and $a_3$ $3$s; $w_b$ contains $b_1$ $1$s, $b_2$ $2$s and $b_3$ $3$s 
\begin{restatable}[Right Statistics for $R_n(1/2/3/4)$]{thm}{rww} 
\label{thm:rww}
If $w\in R_n(1/2/3/4)$, then
\[
 rb(w) =
  \begin{cases} 
      \hfill (m-1)h_1+(m-2)h_2   \hfill & \text{$w\in A$} \\
      \hfill 2a_1+a_2+b_2 \hfill & \text{$w\in B$ or $E$} \\
      \hfill 2a_1+a_2+b_1 \hfill & \text{$w\in C,D,F$ or $G$} \\ 
      \hfill a_1 \hfill & \text{$w\in H$ or $I$} \\

  \end{cases}
\]

  \[
 rs(w) =
  \begin{cases} 
      \hfill 0    \hfill & \text{$w\in A$} \\
      \hfill a_2+2a_3+b_3 \hfill & \text{$w\in B,C,E$ or $F$} \\
       \hfill a_2+2a_3+b_2   \hfill & \text{ $w\in D$ or $F$} \\
      \hfill a_2 \hfill & \text{$w\in H$ or $I$} \\
  \end{cases}
\]
\end{restatable}
	\par Next we include the statistics for words in $R_n(1/2/34)$. By Theorem 3.4, if $w\in R_n(1/2/34)$, then $w$ is in one of the following subsets of $R_n(1/2/34)$:
    \begin{enumerate}[label=(\Alph*)]
	\item $m\leq 2$
     \item $w=u3v$ and $b_1=0=b_2$
    \item $w=u3v$ and $b_1=0$, $b_2=1$
    \item $w=u3v$ and $b_1=1$, $b_2=0$ 
    \item $w=u3v$ and $b_1=1=b_2$ such that $x<y$ where $x$ is the position of the $1$ in $v$, and $y$ is the position of the $2$ in $v$
    \item $w=u3v$ and $b_1=1=b_2$ such that $x>y$ 
\end{enumerate}
Also let $l$ be the position of the last $1\in u$, $z$ be the position of the first $2\in u$, and $h$ be the position of the last $2\in u$.
\begin{restatable}[Statistics for $R_n(1/2/34)$]{thm}{wwTwo}
\label{thm:wwTwo}
If $w\in R_n(1/2/34)$, then for the statistics on $w$ when $w\in A$ are given by Theorem 4.1, and for the other cases:

$$ls=a_2+b_2+2+\sum\limits^{m-1}_{i=3}i$$

 \[
  rb=\begin{cases}
               \hfill (m-1)a_1+(m-2)a_2+|v|-y+\sum\limits^m_{i=3}(m-i)\hfill & \text{$w\in C$}\\
               \hfill(h-a_2)(m-1)+(m-2)(|u|-h)+(m-2)a_2+|v|-x+\sum\limits^m_{i=3}(m-i)\hfill & \text{$w\in B$ or $D$}\\
               \hfill(m-1)a_1+(m-2)a_2+2|v|-x-y+\sum\limits^m_{i=3}(m-i)\hfill & \text{$w\in E$}\\
               \hfill(m-1)a_1+(m-2)a_2+2|v|-x-y-1+\sum\limits^m_{i=3}(m-i)\hfill & \text{$w\in F$}\\
            \end{cases}
\]
$$lb=a_1-z+x+y+1$$
\[
 rs=\begin{cases}
 			\hfill y-1\hfill & \text{$w\in D$}\\
 			\hfill a_2+x+y-1\hfill & \text{$w\in E$}\\
            \hfill a_2+x+y\hfill & \text{$w\in F$}\\
            \hfill (l-a_1)+y\hfill & \text{$w\in B$ or $C$}\\
         \end{cases}
\]
\end{restatable}
\par Similarly when we state the statistics for $w\in R_n(12/3/4)$ we identify the subset of the avoidance class in which $w$ resides. By Theorem 3.3 the possible subsets include:
\begin{enumerate}[label=(\Alph*)]
	\item $m\leq 2$
	\item $w=uv$ where $m\in u$ and $v$ contains two distinct letters
    \item $w=uv$ where $m\in u$ and $v$ is empty or contains one distinct letter
    \item $w=uv$ where $m\in v$ and $v$ contains two distinct letters, where $v$ is of type $c^{b_c}d^{b_d}$ and $x$ is the position of the first $d\in v$ where $d>c$.
\end{enumerate}

\begin{restatable}[Statistics for $R_n(12/3/4)$]{thm}{wwOne}
\label{thm:wwOne}
If $w\in R_n(12/3/4)$, then when $m\leq 2$ refer to Theorem 4.1 for the Wachs and White statistics for $w$; otherwise 
$$ls=\sum\limits^{|u|-1}_{i=1}i+(c-1)b_1+(d-1)b_2$$
$$rb=(y-b_2)+\sum\limits^{|u|}_{i=1}(m-i)$$
 \[
  lb=\begin{cases}
               \hfill (m-c)b_c+(m-d)b_d\hfill & \text{$w\in B$ or $C$}\\
              \hfill (m-c-1)b_c+b_d-x+1\hfill & \text{$w\in D$}\\
            \end{cases}
\]
\[
  rs=\begin{cases}
              \hfill |u|-c+|u|-d+b_c-x\hfill & \text{$w\in B$ or $D$}\\
              \hfill |u|-d+b_c-x\hfill & \text{$w\in C$}\\
              
            \end{cases}
\] 
\end{restatable}

            \begin{proof} There are four results to prove.
            \begin{enumerate} \item Let $w(\sigma)$ be of the form $uv$. Because every element in $u$ is only found once, the largest element in $u$ will be $|u|$, so the sum from one to $|u|-1$ is equivalent to $\sum\limits^{|u|}_{i=2}(i-1)a_i$. Now consider $v$ to have $b_1$ c letters, and $b_2$ d letters. We know that $c$ must have $(c-1)$ unique letters that are smaller and come before $c$, and therefore the contribution from the $c's$ will be $(c-1)b_1$. We also know that $d$ must have $(d-1)$ unique elements that are smaller and come before $d$, and therefore all of the $d's$ will contribute $(d-1)b_2$ to the left and smaller statistic. Therefore the left and smaller statistic is $\sum\limits^{|u|-1}_{i=1}i+(c-1)b_1+(d-1)b_2$.
            \item Let $y$ be the last occurrence of $d$ where $d>c$ in $v$, then $(y-b_2)$ will tell us how many $c's$ come before the last $d$. Now consider the sum, whether $m$ is in $u$ or $v$ it must come after every element in $u$ so the sum from one to $|u|$ of $(m-i)$ will give us the number of larger elements that must occur after each element in $u$. 
            \item When the largest element is found in $u$
            $$lb=(m-c)b_1+(m-d)b_2$$
            Let $m$ be in $u$. Then there are $(m-c)$ distinct letters that are larger and occur before each $c$ and there are $b_1$ $c's$, therefore the contribution to the whole statistic based on $c$ is $(m-c)b_1$. Now consider $d$, there are $m-d$ unique elements that are larger and occur before each $d$ and there are $b_2$ $d's$, therefore the contribution to the whole statistic based on $d$ is $(m-d)b_2$. These are the only things that are not weakly increasing, therefore there will be nothing that contributes to the left and bigger statistic from $u$ because $u$ is weakly increasing. Therefore $lb=(m-c)b_1+(m-d)b_2$.
           
            \par When the largest element is found in $v$ the formula for $lb$ is:
            $$lb=(m-c-1)b_1+b_1-x+1$$
            Where $x$ is the position of the first largest element in $v$.
            
           \par Let the largest element be found in $v$. Then the largest element in $u$ is $(m-1)$ and there are $(m-1)-c$ unique elements in $u$ that come before the first $c$. There are also $b_1$ $c's$ so this contributes $(m-c-1)b_1$ to the left and bigger statistic. Now because $d=m$ there is no contribution from any $d$ to the left and bigger statistic, however we must consider the $c's$ that occur after the first $d$. There will be $|v|-x$ elements that come after the first $m$. Now of these elements we only want the $c's$ so subtract $b_2-1$ because when we subtracted $x$ earlier we accounted for the first $m$. So
            $$|v|-x-(b_2-1)=b_1+b_2-x-b_2+1$$
            $$|v|-x-(b_2-1)=b_1-x+1$$
            Therefore $lb=(m-c-1)b_1+b_1-x+1$.
           \item  When $b_1$ and $b_2$ are greater than zero. 
            $$rs=|u|-c+|u|-d+b_1-x$$
            \par Let $b_1$ and $b_2$ be greater than zero. Then there are $|u|-c$ unique elements that come before $c$ that are larger than $c$, each of these will contribute one to the right and smaller statistic. There are $|u|-d$ unique elements that are larger and come before the first $d$, these will each also contribute one to the right and smaller statistic. Finally let $y$ be the last occurrence of $d$, there will be $(y-b_2+1)$ $c's$ that come before the last $d$ that will each contribute one to the statistic.
            
            \par When $b_1$ is zero remove $|u|-c$ from equation. 
            $$rs=|u|-d+b_1-x$$
            \par When $b_2$ is zero, remove $|u|-d$ from the equation.
            $$rs=|u|-c+b_1-x$$
            \end{enumerate}
            \end{proof}
	The next avoidance class we will consider is $R_n(13/2/4)$. By Theorem 3.9, when $w\in R_n(13/2/4)$, $w$ will be one of the following forms:
    \begin{enumerate}[label=(\Alph*)]
    	\item $m\leq 2$
        \item $w$ is weakly increasing and $m\geq 3$
        \item $w=uv$ where $v=m^{b_m}z^{b_z}$ such that $z\leq (m-1)$
        \item $w=uv$ where $v$ is not weakly increasing and $v$ is of type $m^{b_m}z^{b_z}$ such that $z=(m-1)$ where $x$ is the position of the last $m\in v$, and $y$ is the position of the last $z\in v$ ($v$ begins with the first $m$ in $w$)
    \end{enumerate}
\begin{restatable}[Statistics for $R_n(13/2/4)$]{thm}{wwSeven}
\label{thm:wwSeven}
If $w\in R_n(13/2/4)$, then when $m\leq 2$ refer to Theorem 4.1 for the Wachs and White statistics for $w$; otherwise 
 \[
 lb=\begin{cases}
 			\hfill 0\hfill & \text{$w\in B$}\\
            \hfill (m-z)b_{z}\hfill & \text{$w\in C$ or $D$}\\
         \end{cases}
\]
 \[
 rs=\begin{cases}
 			\hfill 0\hfill & \text{$w\in B$}\\      
            \hfill \sum\limits^m_{i=z+1}a_i\hfill & \text{$w\in C$}\\
            \hfill (y-b_z)\hfill & \text{$w\in D$}\\
         \end{cases}
\]
 \[
 rb=\begin{cases}
 			\hfill l\hfill & \text{$w\in B$}\\
 			\hfill \sum\limits^{m-1}_{i=1}(m-i)a_i\hfill & \text{$w\in C$}\\
            \hfill (x-a_m)+\sum\limits^{m-1}_{i=1}(m-i)a_i\hfill & \text{$w\in D$}\\
         \end{cases}
\]
 
 $$ls=(z-1)b_z+\sum\limits^m_{i=2}(i-1)a_i$$
\end{restatable}
\par The last avoidance class on which we investigated the Wachs and White statistics was $R_n(1/24/3)$. By Theorem 3.10, when $w\in R_n(1/24/3)$ this word could take the following forms:
\begin{enumerate}[label=(\Alph*)]
	\item $m\leq 2$
    \item $w$ is weakly increasing
    \item $w$ is a weakly increasing word, except for a block of $b_1$ $1$s inserted between two distinct letters where $x$ is the distinct letter immediately before this block of $1$s
    \item $w=uv$ where $u$ is of type $1^{a_1}2^{a_2}$ and $v=3^{b_3}4^{b-4}...m^{b_m}$ where $c$ is the position of the first $2$ in $u$ and $d$ is the position of the last $2$ in $u$
\end{enumerate}
\begin{restatable}[Statistics for $R_n(1/24/3)$]{thm}{wwEight}
\label{thm:ww1/24/3}
If $w\in R_n(1/24/3)$, then when $m\leq 2$ refer to Theorem 4.1 for the Wachs and White statistics for $w$, otherwise:
\[
 lb=\begin{cases}
 			\hfill 0\hfill & \text{$w\in B$}\\
            \hfill (x-1)b_1\hfill & \text{$w\in C$}\\
            \hfill a_1-c+1\hfill & \text{$w\in D$}\\
         \end{cases}
\]
 \[
 rs=\begin{cases}
 			\hfill 0\hfill & \text{$w\in B$}\\           
            \hfill \sum\limits^x_{i=2}a_i\hfill & \text{$w\in C$}\\
            \hfill f-a_1 \hfill & \text{$w\in D$}\\
         \end{cases}
\]
 \[
 rb=\begin{cases}
 			\hfill l \hfill & \text{$w\in B$}\\
            \hfill \sum\limits^m_{i=1}(m-i)a_i+(m-x)b_1\hfill & \text{$w\in C$}\\
 			\hfill (d-a_2)(m-1)+(|u|-d)(m-2)+\sum\limits^{m-1}_{i=2}(m-i)a_i\hfill & \text{$w\in D$}\\
            
         \end{cases}
\]

$$ls= \sum\limits^m_{i=1}(i-1)a_i$$

\end{restatable}
\section{Equidistribution of Wachs and White Statistics}
The goal of this work has been to identify distributions of Wachs and White Statistics that are the same for different avoidance classes. Up to this point we have found the cardinalities of the avoidance classes, (because no two avoidance classes can have a statistic equidsitributed unless they are of the same size). We have also characterized the words in each avoidance class we investigated, and found formulas for the Wachs and White statistics for these word. These equations will become very useful in proving our equidistribution results.
\subsection{The Patterns $1/24/3$ and $13/2/4$} First, notice that the complement of $1/24/3$ is $13/2/4$. Therefore $1/24/3\equiv 13/2/4$. We found four equidistribution results on the avoidance classes based on the patterns $1/24/3$ and $13/2/4$. The proof of the first result will be included here, we refer the reader to the appendix for the full proof. 
\begin{restatable}[Equidistribution of Statistics I]{thm}{eqOne}
\label{thm:eqOne}
Consider the avoidance classes based on $1/24/3$ and $13/2/4$, then
\begin{itemize}
\item[(i)]$LS(13/2/4)\sim RB(1/24/3)$
\item[(ii)]$RB(13/2/4)\sim LS(1/24/3)$
\item[(iii)]$LB(13/2/4)\sim LB(1/24/3)$
\item[(iv)]$RS(13/2/4)\sim RS(1/24/3)$
\end{itemize}
\end{restatable}
\begin{proof}Suppose $\sigma_1$ avoids $13/2/4$ and $\sigma_2$ avoids $1/24/3$. Also let $w$ denote the RGF for $\sigma_1$ and let $y$ denote the RGF for $\sigma_2$. Let $A_1, A_2, A_3\subset R_n(13/2/4)$ such that $w\in A_1$ where $w$ has $m\leq 2$ or $w$ is weakly increasing, $w\in A_2$ where $w=uv$ where $v$ is a block of $z\leq (m-1)$, and $w\in A_3$ where $w=uv$ where $v$ contains $m$ and $(m-1)$ in any order after the first $m$. Also let $B_1, B_2, B_3\subset R_n(1/24/3)$ where $y\in B_1$ where $y$ has $m\leq 2$ or $y$ weakly increasing, $y\in B_2$ where $y$ has a block of $1$s inserted between two distinct letters in a weakly increasing word, and finally $y\in B_3$ when $y=uv$ where $u$ contains $1$s and $2$s in any order allowable by the RGF.
\par To show the equidistribution of the Wachs and White statistics we will identify statistic preserving functions from $A_i$ to $B_i$. Following this we will show that these functions are bijections.
\begin{itemize}
	\item[a.] For the proof that $ls(w)=rb(y)$ when $w\in A_1$ and $y\in B_1$ refer to \cite{REU1}.
    \item[b.] Let $w\in A_2$, $y\in B_2$. Then the following are equations for the statistics of individual words:
                    $$ls(w)=(z-1)b_z+\sum\limits^m_{i=2}(i-1)a_i$$
                    $$rb(y)=(m-x)b_1+\sum\limits^{m-1}_{i=1}(m-i)a_i$$
                    Now let $f:A_2\rightarrow B_2$ where when $w=uv$, $f(w)=y$ such $b_z=b_1$ (where $b_z$ is the number of $z\in v$ and $b_1$ is the size of the block of inserted $1$s in $y$), the distinct letter $x\in y$ before this block of $1$s is $x=m-z+1$ where $z$ is the last letter in $w$, and the weakly increasing base of $y$ is formed by taking the RGF of the complement of the partition in a bijection with $u$. An example would be $f(1234422)=1123114$. Then $x=4-2+1=3$ so we insert the block of $1$s after the last $3$ and there will be $b_2=2$ $1$s in this block. Also we take the complement of $12344$ to get the base of the word to be $11234$.
                    \par We will prove that $f$ is injective. Consider $w_1, w_2\in A_2$ where $w_1\neq w_2$. Then either $u_1\neq u_2$ or $v_1\neq v_2$. When $u_1\neq u_2$ then $u_1^c\neq u_2^c$, so the weakly increasing word in $f(w_1)$ differs from the weakly increasing word in $f(w_2)$ so $f(w_1)\neq f(w_2)$. If $v_1\neq v_2$ then there will be a different number of $z$s in $w_1$ and $w_2$, so the block of $1$s inserted in $y$ will be of different sizes so $f(w_1)\neq f(w_2)$. Therefore when $f(w_1)=f(w_2)$, $w_1=w_2$, meaning that $f$ is injective.
                    \par Let $f^{-1}$ be the inverse of $f$, where $f^{-1}(y)=w$ is formed by taking the RGF of the complement of the partition corresponding to the weakly increasing base of $y$. This will be weakly increasing by Lemma 2.6. Then $z\in w$ is determined by $z=m-x+1$ and $b_1=b_z$. 
                    \par Because $f$ is injective and $f^{-1}$ exists, we conclude that $f$ is bijective. 
                   \par Now consider (notice that $c_i$ denotes the number of the letter $i$ in the weakly increasing base of $f(w)$)
                    $$rb[f(w)]=(m-x)b_1+\sum\limits^{m-1}_{i=1}(m-i)c_i$$
                    $$=[m-(m-z+1)]b_z+\sum\limits^{m-1}_{i=1}(m-i)a_{m-i+1}$$
                    $$=(z-1)b_z+\sum\limits^{m-1}_{i=1}[(m+1-i)-1]a_{m-i+1}$$
                    $$=(z-1)b_z+\sum\limits^m_{i=2}(i-1)a_i=ls(w)$$
                    Therefore because $rb(f(w))=ls(w)$ and $f$ is bijective, $LS(A_2)\sim RB(B_2)$. 
                    
          \par For the rest of the proof see appendix.
           \end{itemize}
\end{proof}
\subsection{Other Equidistribution Results for Single Pattern Avoidance Classes}
The majority of equidistribution results based on single pattern avoidance classes we found focused on the two patterns $13/2/4$ and $1/24/3$. However we have another equidistribution result based on the patterns $124/3$ and $134/2$. From Theorem 2.5 we have $124/3\equiv 134/2$.
\begin{restatable}[Equidistribution of Statistics II]{thm}{eqTwo}
\label{thm:eqTwo}
Consider the avoidance classes based on the patterns $124/3$ and $134/2$, then
$$LS(124/3)\sim LS(134/2)$$
\end{restatable}
\begin{proof} Let $\phi:R_n(134/2)\rightarrow R_n(124/3)$ such that 
\[
 \phi(w)=\begin{cases}
 			\hfill w\hfill & \text{if $w$ is weakly increasing}\\
            \hfill w*\hfill & \text{if $w$ is not weakly increasing}\\

         \end{cases}
\]
where $w*$ is defined to be the word that is formed from flipping each block-singleton letter to a singleton-block letter. We can see that this function is well-defined because there is only one way to flip the block-singleton letter to a singleton-block letter. So $\phi(122233213)=123222133$ and $\phi(11222133243)=12113222433$. Now let $w_1,w_2\in R_n(134/2)$ such that $\phi(w_1)=\phi(w_2)$. There are two options based on what form $w_1$ takes:
\begin{itemize}
\item[i)] Suppose $w_1$ is weakly increasing. Then $w_1\in R_n(124/3)$ by Lemma 3.1. Because $w_1$ is weakly increasing $\phi$ is the identity function which is clearly bijective. 
\item[ii)] Suppose $w_1$ is not weakly increasing. Then $w_1$ is weakly increasing except for some number of block-singleton letters. There is only one way to flip a block-singleton letter to a singleton block, therefore  when $\phi(w_1)=\phi(w_2)$, $w_1=w_2$. So when $w_1$ is not weakly increasing $\phi$ is injective.
\end{itemize}
Next consider the function $\phi^{-1}:R_n(124/3)\rightarrow R_n(134/2)$, where $y\in R_n(124/3)$ and 
\[
 \phi^{-1}(y)=\begin{cases}
 			\hfill y\hfill & \text{if $y$ is weakly increasing}\\
            \hfill y*\hfill & \text{if $y$ is not weakly increasing}\\

         \end{cases}
\]
where we have $y*$ defined as the word that is formed by changing every singleton-block letter to a block-singleton. So $\phi^{-1}(123222133)=122233213$ and $\phi^{-1}(12113222433)=11222133243$. So it is clear that $\phi^{-1}$ reverses $\phi$. Therefore because $\phi$ is injective and there exists $\phi^{-1}$ which reverses $\phi$, $\phi$ is a bijection between $R_n(124/3)$ and $R_n(134/2)$. 
\par An important thing to note about this bijection is that it preserves the occurrences of each letter that occurs in the input word. Therefore if there are the same occurrences of each letter in $w$ and $\phi(w)$, then $ls(w)=ls(\phi(w))$ by Lemma 2.3.
Therefore because there is a bijection between $R_n(134/2)$ and $R_n(124/3)$ that preserves the left and smaller statistic for every word, $LS(134/2)\sim LS(124/3)$.
\end{proof}
\section{Avoidance Classes based on Multiple Patterns}
The idea of pattern avoidance can be extended to avoidance classes based on multiple patterns. Here in this paper we will investigate both (3,4)-pair avoidance classes and (4,4)-pair avoidance classes. Specifically we will focus on avoidance classes based on at least one pattern of interest from the single pattern results above.
\par The investigation of these multiple pattern avoidance classes will follow the same structure as the previous results: specifically we will characterize the avoidance classes, find the cardinality of the avoidance classes, and finally investigate the Wachs and White statistics for these avoidance classes. 

\subsection{3,4-Pairs}
\par As we begin to consider avoidance classes determined by pairs of patterns, it is helpful to realize that when we have $R_n(\sigma_1,\sigma_2)$ it will be more interesting when $w(\sigma_2)\in R_n(\sigma_1)$, (where $\sigma_1$ is the pattern of length three). Otherwise $R_n(\sigma_1,\sigma_2)$ is already determined by the single pattern avoidance class $R(\sigma_1)$. For example, consider $R_n(1/2/3,1/2/34)$. Then $w(1/234)=1222\in R_n(1/2/3,1/2/3/4)$ because $w\in R_n(1/2/3)$. And $w(12/3/45)=11233\not\in R_n(1/2/3,1/2/34)$ because $w\not\in R_n(1/2/3)$. So we have the following property:
\begin{restatable}{lemma}{DifSiz} Let $\sigma_1\in\Pi_k$, $\sigma_2\in\Pi_l$ where $l<k$, such that $\sigma_2$ contains $\sigma_1$. Then $R_n(\sigma_1,\sigma_2)=R_n(\sigma_1)$.
\end{restatable}
\begin{restatable}{thm}{ThreeEight}
\label{thm:ThreeEight}
The word characterization and cardinality of the avoidance class based on the patterns $12/3$ and $1/24/3$ are given by:
\begin{itemize}
\item[(i)]$R_n(12/3,1/24/3)=\{w\in R_n:w=ab$ where $a$ is strictly increasing prefix ending with $m$, and \begin{enumerate}[label=(\Alph*)] \item $b$ is a suffix of $m$'s 
	\item $b$ is a suffix of $1$s where $|b|\geq 1\}$
	\end{enumerate}
\item[(ii)]$\#\Pi_n (12/3,1/24/3)=2(n-1)$
\end{itemize}
\end{restatable}
\begin{proof}Let $w\in R_n(12/3,1/24/3)$. Then, by Lemma 2.1, because $w\in R_n(12/3)$, the only element repeated is the last element, $x$, in the word such that $x$ occurs once after $x-1$ and then in a block at the end of the word. We also know that $w\in R(1/24/3)$. So by Theorem 3.9 $w$ could take three different forms. 
	\begin{itemize}
    \item[a. ]First let $m\leq 2$. Then if $w\in R_n(12/3)$, $m$ could be $1$, or $m=2$. If $m=2$ then $w=12b$ where $b$ is either a block of $1$s or a block of $2$s by Lemma 2.1.
    \item[b. ]Let $w$ be weakly increasing word with a block of ones of any size inserted between two distinct terms. Then $w$ will contain $12/3$ if we consider the subpartition associated with the subword containing the first $1\in w$, the first $1$ inserted between two distinct letters, and the distinct letter after the block of inserted ones. This subpartition will standardize to $12/3$. So when $w\in R_n(12/3)$ it cannot take this form unless the block of $1$s is inserted after the last $m\in w$.
    \item[c. ] Let $w$ begin with $1$s and $2$s in any order followed by a weakly increasing suffix that begins with the first occurrence of 3. Then if $w$ also avoids $12/3$ $w$ must begin with $12$ and the weakly increasing suffix must actually be strictly increasing, except for $m$ which can be a block of any size one or larger by Lemma 2.1.
    \end{itemize}
Therefore $w=ab$ where $a$ is strictly increasing prefix ending with $m-1$, and $b$ is a suffix of $m$'s or a suffix of $1$s where $|b|\geq 1$.
\par Next consider $\#\Pi(12/3,1/24/3)$. When $m=1$ there are $S(n,1)=1$ forms $w$ could take. When $m=n$ there are $S(n,n)=1$ forms $w$ could take. Now consider $w$ where $1<m<n$. Then there will be one form of $a$ because $a$ is strictly increasing, and two possibilities for the form of $b$: either $b$ is a block of $1$s or a block of $m$s. Therefore $\#\Pi_n(12/3,1/24/3)=2+2(n-2)=2(n-1)$.
\end{proof}
While we continue to investigate the word characterization and cardinality for multiple pattern avoidance classes, the proofs for the following results can be found in the appendix.
\begin{restatable}{thm}{TwoSeven}
\label{thm:TwoSeven}
The word characterization and cardinality of the avoidance class based on the patterns $1/23$ and $13/2/4$ are given by:
\begin{itemize}
\item[(i)]$R_n(1/23,13/2/4)=\{w\in R_n:$ 
	\begin{enumerate}[label=(\Alph*)]
    \item $w=ab$ where $a$ is a prefix of ones of any length and $b$ is strictly increasing beginning with the first occurrence of $2$ or $b$ can be empty
    \item $w=abc$ where $a$ is a prefix of $1$s of any length, $b$ is strictly increasing beginning with the first occurrence of $2$ and ending with $m$, and $c$ is a suffix of length one containing a $1\}$
    \end{enumerate}
\item[(ii)]$\#\Pi_n (1/23,13/2/4)=2(n-1)$
\end{itemize}
\end{restatable}

Therefore $\#\Pi_n(12/3,1/24/3)=\#\Pi_n(1/23,13/2/4)$. So these pairs of patterns are Wilf-equivalent which we will write as $(12/3,1/24/3)\equiv(1/23,13/2/4)$. Now we will continue to characterize and enumerate $(3,4)$-pair avoidance classes.
\begin{restatable}{thm}{FourFour}
\label{thm:FourFour}
The word characterization and cardinality of the avoidance class based on the patterns $13/2$ and $123/4$ are given by:
\begin{itemize}
\item[(i)]$R_n(13/2,123/4)=\{w\in R_n: w$ is weakly increasing and no element except $m$ can be repeated more than twice$\}$
\item[(ii)]$\#\Pi_n (13/2,123/4)=1+\sum^{n-1}_{j=0}\sum^{\lfloor\frac{n-j}{2}\rfloor}_{i=0}{{n-j-i}\choose i}$
\end{itemize}
\end{restatable}

\begin{restatable}{thm}{FourThree}
\label{thm:FourThree}
The word characterization and cardinality of the avoidance class based on the patterns $13/2$ and $1/234$ are given by:
\begin{itemize}
\item[(i)]$R_n(13/2,1/234)=\{w\in R_n:w=ab$ where $a$ is a block of ones, $b$ is weakly increasing beginning with $2$ with no element repeated more than twice$\}$
\item[(ii)]$\#\Pi_n (13/2,1/234)=1+\sum^{n-1}_{j=0}\sum^{\lfloor\frac{n-j}{2}\rfloor}_{i=0}{{n-j-i}\choose i}$
\end{itemize}
\end{restatable}

Remember that the goal of this section is to find $(3,4)$-pair avoidance classes that are Wilf-equivalent. From Theorems 6.3 and 6.4 we conclude that $(13/2,123/4)\equiv(13/2,1/234)$. So these two avoidance classes are candidates for the equidistribution of Wachs and White statistics. Before continuing on to consider these distributions first let us expand our characterization and enumeration of avoidance classes determined by $(4,4)$-pairs.
\subsection{4,4-Pairs} For $(4,4)$-pair avoidance classes, we are no longer concerned about the containment or avoidance of patterns in the pair. Two patterns of the same size cannot contain one another if they are distinct.

\begin{restatable}{thm}{SevenSix}
\label{thm:SevenSix}
The word characterization and cardinality of the avoidance class based on the patterns $13/2/4$ and $124/3$ are given by:
\begin{itemize}
\item[(i)]$R_n(13/2/4,124/3)=\{w\in R_n:$
\begin{enumerate}[label=(\Alph*)]
	\item $w$ is weakly increasing
    \item $m\leq 2$ and $1$ and $2$ are singleton-block letters
    \item $m\leq 2$ and $1$ is a singleton-block letter and $2$ is a block letter
    \item $w=ab$ where $a$ is weakly increasing and $b$ is a block of a letter $z$ such that $z$ is a singleton-block letter in $w$ and $z\leq(m-1)$
    \item $w=ab$ where $a$ weakly increasing from $1$ to $m-2$ and $b$ contains the singleton-block letters $(m-1)$ and $m$ such that $b=(m-1)m(m-1)^lm^nl\}$
    \end{enumerate}
\item[(ii)]$\#\Pi_n(13/2/4,124/3)=2n-5+\sum\limits^n_{m=1}{{n-1}\choose{n-m}}+\sum\limits_{|a|=3}^{n-1}\sum\limits_{m=3}^{|a|}{{|a|-1}\choose{|a|-m}}(m-1)+ \sum\limits_{|a|=2}^{n-4}\sum\limits_{m=3}^{|a|+2}{{|a|-1}\choose{|a|-(m-2)}}(n-|a|-2)$
\end{itemize}
\end{restatable}
\begin{proof} Let $w\in R_n(13/2/4,124/3)$. Then $w\in R_n(13/2/4)$ and so we know that $w$ could be of three forms:
\begin{itemize}
	\item[a. ]First let $m\leq 2$. By Theorem 3.9, $w\in R_n(13/2/4)$. However we also want $w\in R_n(124/3)$ so there are some restrictions. By Theorem 3.8 if $w$ is weakly increasing, where $w=1^{a_1}2^{a_2}$, or if $w$ consists of only block letters and singleton-block letters, then $w\in R_n(13/2/4,124/3)$. Therefore $1$ can be a singleton-block, or $1$ and $2$ are singleton-blocks.
    \item[b. ]Next let $w=ab$ where $a$ is weakly increasing and $b$ is a suffix beginning with $m$ and containing $m$ and $m-1$ in any order after that. By Theorem 3.9, $w\in R_n(13/2/4)$. However by Theorem 3.8, if $w\in R_n(124/3)$ then $a$ must contain only one occurrence of $m-1$, and $b=m(m-1)^lm^n$ or $b=m^{a_m}(m-1)^l$. So either $(m-1)$ and $m$ are singleton-blocks, or only $(m-1)$ is a singleton-block.
    \item[c. ]Finally let $w=ab$ where $a$ is weakly increasing from $1$ to $m$, and $b$ contains a block of some $z<m-1$. Then by Theorem 3.9, $w\in R_n(13/2/4)$. By Theorem 3.8, if $w\in R_n(124/3)$, then $z$ can only be a letter that occurs exactly once in $a$. Therefore $z$ is a singleton-block in $w$.
\end{itemize}
Therefore the characterization given by Theorem 6.5 describes every possible $w\in R_n(13/2/4,124/3)$.
\par Next consider the enumeration of the avoidance class. 
\begin{itemize}
	\item[(i)] Let $w$ be weakly increasing and $1\leq m\leq n$. Then by Lemma 2.2 there are $\sum\limits^n_{m=1}{{n-1}\choose{n-m}}$ possibilities for $w$.
    \item[(ii)] Let $m=2$ and only $1$ be a singleton-block letter. Then there are three assigned positions. We know that the first letter is $1$ and the second letter is $2$ and the last letter is also $1$. Therefore the length of the block of $1$s can vary from one to $n-2$. So there are $n-2$ possible words for $w$.
    \item[(ii)] Let $m=2$ and $1$ and $2$ be singleton-block letters. Then $w=121^l2^k$ where $l,k>0$. So $l$ can vary from one to $(n-3)$ and as $l$ varies, so too will $k$. So there are $n-3$ possible words of this form.
    \item[(iii)] Let $w=ab$ where $m\geq 3$, $b$ is a block of $z$ where $z$ is a singleton-block, and $a$ is weakly increasing. Then $3\leq m\leq |a|$. Also note that $1\leq z\leq (m-1)$. There will be $\sum\limits_{|a|=3}^{n-1}\sum\limits_{m=3}^{|a|}{{|a|-1}\choose{|a|-m}}(m-1)$ forms of $w$ by Lemma 2.2.
    \item[(iv)] Let $w=ab$ where $m\geq 3$, $a$ is weakly increasing, and $b=(m-1)m(m-1)^lm^n$. So $2\leq |a|\leq (n-4)$, $3\leq m\leq |a|-2$ and the last two elements $(m-1), m\not\in a$, and there are $n-|a|-2$ possible forms of $b$. Therefore there are $\sum\limits_{|a|=2}^{n-4}\sum\limits_{m=3}^{|a|+2}{{|a|-1}\choose{|a|-(m-2)}}(n-|a|-2)$ ways to form $w$.
\end{itemize}
Therefore $\#\Pi_n(13/2/4,124/3)=2n-5+\sum\limits^n_{m=1}{{n-1}\choose{n-m}}+\sum\limits_{|a|=3}^{n-1}\sum\limits_{m=3}^{|a|}{{|a|-1}\choose{|a|-m}}(m-1)+ \sum\limits_{|a|=1}^{n-4}\sum\limits_{m=3}^{|a|+2}{{|a|-1}\choose{|a|-(m-2)}}(n-|a|-2)$.
\end{proof}
The proofs for the following results can be found in the appendix.
\begin{restatable}{thm}{SevenFive}
\label{thm:SevenFive}
The word characterization and cardinality of the avoidance class based on the patterns $13/2/4$ and $134/2$ are given by:
\begin{itemize}
\item[(i)]$R_n(13/2/4,134/2)=\{w\in R_n:$
\begin{enumerate}[label=(\Alph*)]
	\item $w$ is weakly increasing
    \item $m=2$, $1$ a block-singleton
    \item $m=2$, $1$ and $2$ block-singletons
    \item $w=ab$ where $a$ weakly increasing and $b$ is either the last occurrence of a block-singleton in $w$, or $(m-1),m\in b$ such that both are block-singleton letters, $(b=(m-1)^lm^k(m-1)m)\}$
    \end{enumerate}
\item[(ii)]$\#\Pi_n(13/2/4,134/3)=2n-5+\sum\limits^n_{m=1}{{n-1}\choose{n-m}}+\sum\limits_{|a|=3}^{n-1}\sum\limits_{m=3}^{|a|}{{|a|-1}\choose{|a|-m}}(m-1)+ \sum\limits_{|a|=1}^{n-4}\sum\limits_{m=3}^{|a|+2}{{|a|-1}\choose{|a|-(m-2)}}(n-|a|-2)$.
\end{itemize}
\end{restatable}

\begin{restatable}{thm}{NewEight}
\label{thm:NewEight}
The word characterization and cardinality of the avoidance class based on the patterns $14/2/3$ and $1/24/3$ are given by:
\begin{itemize}
\item[(i)]$R_n(14/2/3,1/24/3)=\{w\in R_n:$
\begin{enumerate}[label=(\Alph*)]
	\item $w$ is weakly increasing
    \item $m\leq 2$
    \item $w=ab$ where $a$ is a prefix of $1$s and $2$s beginning with a $1$, and $b$ is a weakly increasing word beginning with the first occurrence of $3\}$
    \end{enumerate}
\item[(ii)]$\#\Pi_n (14/2/3,1/24/3)=2^{n-1}+\sum\limits_{|b|=1}^{n-2}\sum\limits^{|b|}_{i=1}{{|b|-1}\choose{|b|-i}}(2^{n-|b|}-1)$
\end{itemize}
\end{restatable}

\begin{restatable}{thm}{NewSeven}
\label{thm:NewSeven}
The word characterization and cardinality of the avoidance class based on the patterns $14/2/3$ and $13/2/4$ are given by:
\begin{itemize}
\item[(i)]$R_n(14/2/3,13/2/4)=\{w\in R_n:$
\begin{enumerate}[label=(\Alph*)]
	\item $w$ is weakly increasing
    \item $m\leq 2$
    \item $w=ab$ where $a$ is weakly increasing and $b$ begins with the first occurrence of $(m-1)$ and then contains $m$ and $(m-1)$ in any order$\}$
    \end{enumerate}
\item[(ii)]$\#\Pi_n (14/2/3,1/24/3)=2^{n-1}+\sum\limits_{|a|=1}^{n-2}\sum\limits^{|a|}_{i=1}{{|a|-1}\choose{|a|-i}}(2^{n-|a|}-1)$
\end{itemize}
\end{restatable}

\begin{restatable}{thm}{EightFive}
\label{thm:EightFive}
The word characterization and cardinality of the avoidance class based on the patterns $1/24/3$ and $134/2$ are given by:
\begin{itemize}
\item[(i)]$R_n(1/24/3,134/2)=\{w\in R_n:$
\begin{enumerate}[label=(\Alph*)]
	\item $w$ is weakly increasing
    \item $w$ is a weakly increasing except for the block-singleton $1\}$
    \end{enumerate}
\item[(ii)]$\#\Pi_n(1/24/3,134/2)=2+\sum_{m=2}^{n-1}({n-1\choose n-m}+{n-2\choose n-m-1}(m-1))$
\end{itemize}
\end{restatable}

\begin{restatable}{thm}{EightSix}
\label{thm:EightSix}
The word characterization and cardinality of the avoidance class based on the patterns $1/24/3$ and $124/3$ are given by:
\begin{itemize}
\item[(i)]$R_n(1/24/3,124/3)=\{w\in R_n:$
\begin{enumerate}[label=(\Alph*)]
	\item $w$ is weakly increasing
    \item $w=ab$ where $a$ is weakly increasing with only one occurrence of $z$ followed by the  suffix, $b$, where $|b|\geq 1$ and contains $z\leq(m-1)$ such that $z$ is a singleton-block in $w\}$
    \end{enumerate}
\item[(ii)]$\#\Pi_n(1/24/3,124/3)=2+\sum_{m=2}^{n-1}({n-1\choose n-m}+{n-2\choose n-m-1}(m-1))$
\end{itemize}
\end{restatable}

\begin{restatable}{thm}{EightSag}
\label{thm:EightSag}
The word characterization and cardinality of the avoidance class based on the patterns $1/24/3$ and $1/23/4$ are given by:
\begin{itemize}
\item[(i)]$R_n(1/24/3,1/23/4)=\{w\in R_n:$
\begin{enumerate}[label=(\Alph*)]
	\item $m\leq 2$
    \item $w=abc$ where $a$ is a prefix of $1$s, $b$ is strictly increasing beginning with $2$ and ending with $(m-1)$ and $c$ is a suffix of $m$s or $1$s
    \item $w$ is a strictly increasing word with a prefix, $a$, of $1$s where $|a|\geq 1$ and a $1$ inserted between two terms$\}$
    \end{enumerate}
\item[(ii)]$\#\Pi_n (1/24/3,1/23/4)=1+2^{n-1}+\sum\limits_{m=3}^{n-1}(nm-m^2)$
\end{itemize}
\end{restatable}

\begin{restatable}{thm}{SevenSag}
\label{thm:EightSag}
The word characterization and cardinality of the avoidance class based on the patterns $13/2/4$ and $1/23/4$ are given by:
\begin{itemize}
\item[(i)]$R_n(13/2/4,1/23/4)=\{w\in R_n:$
\begin{enumerate}[label=(\Alph*)]
	\item $m\leq 2$
    \item $w=abc$ where $a$ is a prefix of $1$s, $b$ is strictly increasing beginning with $2$ and ending with $(m-1)$ and $c$ is a block of $z\leq m\}$
    \end{enumerate}
\item[(ii)]$\#\Pi_n (13/2/3,1/23/4)=1+2^{n-1}+\sum\limits_{m=3}^{n-1}(nm-m^2)$
\end{itemize}
\end{restatable}

After this study of $(3,4)$-pair avoidance classes and $(4,4)$-pair avoidance classes one finding supported by the cardinalities of these avoidance classes is the following conjecture:
\begin{con} Let $\pi_1\in\Pi_k$ and $\pi_2\in\Pi_l$ such that $k\leq l$.
	\begin{itemize}
    	\item[(i)] If $\pi_2,\pi_2^c\in\Pi_l(\pi_1)$, then $(\pi_1,\pi_2)$ is Wilf equivalent to $(\pi_1,\pi_2^c)$.
        \item[(ii)] If $\pi_2\in\Pi_l(\pi_1)$ and $\pi_2^c\not\in\Pi_l(\pi_1)$, then $(\pi_1,\pi_2)$ is Wilf equivalent to $(\pi_1^c,\pi_2^c)$.
    \end{itemize}
\end{con}

\section{Statistics across Multiple Pattern Avoidance Classes}

The following theorems give formulas for statistics of individual words of some of the avoidance classes based on multiple patterns. 

\begin{restatable}{thm}{multi} The Patterns 12/3 and 1/24/3

        	 $$ls(w)=\sum\limits_{i=1}^m a_i-m$$
            $$rb(w)=\sum\limits^{m-1}_{i=1}(m-i)a_i$$
             \[
 lb=\begin{cases}
 			\hfill 0\hfill & \text{if $w=ab$ where $b$ is a block of $m$'s}\\
			\hfill (m-1)|b|\hfill & \text{if $w=ab$ where $b$ is a block of $1$'s}\\            
         \end{cases}
\]
\[
 rs=\begin{cases}
 			\hfill 0\hfill & \text{if $w=ab$ where $b$ is a block of $m$'s}\\
			\hfill |a|-1 \hfill & \text{if $w=ab$ where $b$ is a block of $1$'s}\\            
         \end{cases}
\]   
\end{restatable}
\begin{proof} The equation for $ls(w)$ is given by Lemma 2.3. The equation for $rb(w)$ is given by Lemma 2.4. 
\par Now consider when $w(\sigma)\in R_n(12/3, 1/24/3)$. Then there are two possibilities for the form of $w(\sigma)$ . The first possibility is when $w(\sigma)=ab$ where $a$ is strictly increasing and $b$ is a block of $m$'s of any length by Theorem 6.2. Then $lb(w)=0$ and $rs(w)=0$ because $w$ is weakly increasing. 
\par Now consider the other form that $w$ could take, which is $w=ab$ where $a$ is strictly increasing and $b$ is a block of ones of any size. Then $a$ will have no contribution to the left and bigger statistic and every one in $b$ will have $(m-1)$ elements to the left and bigger and there are a total of $b$ ones. Therefore $lb(w)=|b|(m-1)$. Now let $x\in a$ where $x>0$ then $x$ will be larger than the ones in $b$. So every element in $a$ other than the first element, the only $1$ in $a$, will have a contribution of one to $rs(w)$. Therefore $rs(w)=|a|-1$.
\end{proof}

\begin{restatable}{thm}{multistat} The patterns 1/23,13/2/4
            	 $$ls(w)=\sum\limits_{i=1}^m a_i-m$$
                $$rb(w)=\sum\limits^{m-1}_{i=1}(m-i)a_i$$
                \[
 lb=\begin{cases}
 			\hfill 0\hfill & \text{if the word does not end with a one}\\
			\hfill m-1\hfill & \text{if the last letter is a one}\\            
         \end{cases}
\] 
                \[
 rs=\begin{cases}
 			\hfill 0\hfill & \text{if the word does not end with a one}\\
			\hfill \sum\limits_{i=2}^m a_i\hfill & \text{if the last letter is a one}\\            
         \end{cases}
\]   
\end{restatable}
\section{Multiple Pattern Equidistribution of Statistics}
Like the single pattern avoidance classes, the purpose of characterizing and enumerating avoidance classes based on pairs of patterns has been to consider the distribution of Wachs and White statistics across these avoidance classes. Because of the large number of Wilf-equivalent pairs of patterns considered we found a wealth of equidistribution results for avoidance classes based on $(3,4)$-pairs as well as $(4,4)$-pairs.
\par We have already seen that $(12/3,1/24/3)\equiv(1/23,13/2/4)$. Therefore we investigated the distribution of the Wachs and White statistics on these two avoidance classes.
\begin{restatable}{thm}{eqTwoThree}
\label{thm:eqTwoThree} We have the following equidistribution results for the avoidance classes $R_n(12/3,1/24/3)$ and $R_n(1/23,13/2/4)$:
$$RB(12/3,1/24/3)\sim LS(1/23,13/2/4)$$
$$RS(12/3,1/24/3)\sim LB(1/23,13/2/4)$$
\end{restatable}
\begin{proof} Let $\phi:R_n(12/3,1/24/3)\rightarrow R_n(1/23,13/2/4)$ such that for every $w\in R_n(12/3,1/24/3)$, where $w=ab$ $\phi(w)=b'a'$. Here $a$ is strictly increasing and $a'$ is formed by adding one to every letter in $a$. Also $b$ (the block of $m$s in $w$), becomes $b'$ by $\phi$ such that $|b'|=|b|$ and $b'$ is a block of $1$s. An example is $\phi(1234555)=1112345$ because $a=1234$ so $a'=2345$ and $b=555$ so $b'=111$. When $w=abc$ where $a$ is the first $1$, $b$ is a strictly increasing subword from $2$ to $m$, and $c$ is a block of $1$s, then $\phi(w)=cba$. Then $\phi(1234111)=1112341$ because $a=1$, $b=234$, and $c=111$.
\par Now using proof by contradiction we will show that $\phi$ is injective. Let $w_1,w_2\in R_n(12/3,1/24/3)$ such that $w_1\neq w_2$. Then if $w_1=a_1b_1$ and $w_2=a_2b_2$ we have $b'_1a'_1=\phi(w_1)=\phi(w_2)=b'_2a'_2$. However $w_1\neq w_2$; because $a$ is strictly increasing, the only way for this to happen is for $b_1\neq b_2$. Then $b'_1\neq b'_2$. 
\par Next consider when $w_1=a_1b_1c_1$ and $w_2=a_2b_2c_2$. Then if $c_1b_1a_1=\phi(w_1)=\phi(w_2)=c_2b_2a_2$, the only way for this to occur is when $c_1=c_2, b_1=b_2$ and $a_1=a_2$. Therefore $w_1=w_2$. Therefore $\phi$ is injective.
\par Next consider the inverse of $\phi$, $\phi^{-1}:R_n(1/23,13/2/4)\rightarrow R_n(12/3,1/24/3)$ such that 
\[
\phi^{-1}(y) =
  \begin{cases} 
      \hfill ab   \hfill & \text{$y=b'a'$} \\
      \hfill abc \hfill & \text{$y=cba$} 
  \end{cases}
\]
\par Then because $\phi$ is injective, $\phi^{-1}$ exists, and $\phi^{-1}$ reverses $\phi$, $\phi$ a bijection. 
\par Now consider $ls(\phi(w))$, when $w=ab$ then because $a'$ is strictly increasing from $2$ to $m$ and $b'$ is a block of $1$s, and because $a$ is strictly increasing from $1$ to $(m-1)$ and $b$ is a block of $m$s, 
$$ls(\phi(w))=ls(a')+ls(b')=\sum\limits_{i=2}^{m}a'_i-m+0=\sum\limits_{i=1}^{m-1}a_i-m=rb(a)=rb(a)+rb(b)=rb(w)$$
When $w=abc$ then $a$ contributes $m$ to $rb(w)$, and when $y=cba$, $a$ contributes $m$ to $ls(w)$, $rb(c)=0=ls(c)$ and $rb(b)=ls(b)$ because $b$ is strictly increasing.
\par Therefore $RB(12/3,1/24/3)\sim LS(1/23,13/2/4)$.
\par Next consider $rs(w)$ when $w=ab$; then by 7.1 and 7.2 we know that $rs(w)=0=lb(\phi(w))$. Also consider when $w=abc$, then $lb(\phi(w))=m-1=|a|+|b|-1=rs(w)$. Therefore $RS(12/3,1/24/3)\sim LB(1/23,13/2/4)$.
\end{proof}
\begin{restatable}{thm}{eqEight}
\label{thm:eqEight}
The equidistribution of statistics on avoidance classes $R_n(1/24/3,134/2)$ and $R_n(1/24/3,124/3)$ are
$$LS(1/24/3,134/2)\sim LS(1/24/3,124/3)$$
$$RS(1/24/3,134/2)\sim RS(1/24/3,124/3)$$
\end{restatable}
\begin{proof} Let $f:R_n(1/24/3,134/2)\rightarrow R_n(1/24/3,124/3)$ such that 
\[
 f(w) =
  \begin{cases} 
      \hfill w   \hfill & \text{$w$ weakly increasing} \\
      \hfill 12^k1^l \hfill & \text{$w=1^l2^k1$} \\
      \hfill 121^l2^k \hfill & \text{$w=1^l2^k12$}\\
      \hfill w' \hfill & \text{$w$ weakly increasing except block-singleton $1$}
  \end{cases}
\]
where $w'$ is formed by switching the block-singleton to a singleton-block.
We know that $f(w_1)=f(w_2)$ because,
\begin{itemize} 
	\item[i)] When $w_1$ is weakly increasing, $w_1=f(w_1)=f(w_2)=w_2$ because we then know that $f(w_2)$ is a weakly increasing word.
    \item[ii)] When $w_1$ contains one or more block-singletons, then $f(w_1)$ is formed by switching the block-singleton to a singleton block without changing anything else in $w_1$, so the weakly increasing base of $w_1$ must stay the same, and the length of the block of the block-singleton letters and position of the singleton will become the first position for the block. Therefore if $f(w_1)=f(w_2)$, $w_1=w_2$.
\end{itemize}
Therefore $f$ is injective.
Next consider $f^{-1}$ where $f^{-1}(y)=y$ when $y$ is strictly increasing or $f^{-1}(y)=y'$ where $y'$ switches every singleton-block to a block-singleton. By the same reasoning as above, $f^{-1}$ is injective. Therefore $f$ is a bijection. 
\par Because $f$ preserves the number of each letter $l$ that occurs in $w$, by Lemma 2.3, $ls(w)=ls(f(w))$. Therefore $LS(1/24/3,134/2)\sim LS(1/24/3,124/3)$. Also by Lemma 2.4, and the fact that $1$ contributes nothing to the $rs$ statistic for $w$, $rs(w)=rs(f(w))$. Therefore $RS(1/24/3,134/2)\sim LS(1/24/3,124/3)$.
\end{proof}
\begin{restatable}{thm}{eqSeven}
\label{thm:eqSeven}
The equidistribution of statistics on avoidance classes $R_n(13/2/4,134/2)$ and $R_n(13/2/4,124/3)$ are:
$$LS(13/2/4,134/2)\sim LS(13/2/4,124/3)$$
$$RS(13/2/4,134/2)\sim RS(13/2/4,124/3)$$
\end{restatable}
\begin{proof} Let $\phi:R_n(13/2/4,134/2)\rightarrow R_n(13/2/4,124/3)$ such that 
\[
 \phi(w) =
  \begin{cases} 
      \hfill w  \hfill & \text{$w$ weakly increasing} \\
      \hfill w' \hfill & \text{$w$ contains at least one block-singleton} \\
  \end{cases}
\]
where $w'$ is formed by switching each block-singleton to a singleton-block.
\par Then let $\phi(w_1)=\phi(w_2)$ such that $w_1,w_2\in R_n(13/2/4,134/2)$. Consider the following possibilities:
\begin{itemize}
	\item[i)] If $w_1$ is weakly increasing, then $\phi$ is simply the identity. Therefore $w_1=w_2$.
    \item[ii)] If $w_1$ contain block-singletons, then when we "switch" a block-singleton to a singleton-block. This is done by moving the singleton to the position directly after the distinct element originally before the block, and moving the block directly after the distinct element originally before the singleton. There is only one way to do this and it preserves our original weakly increasing base of the word $w_1$. So if $w'_1=\phi(w_1)=\phi(w_2)=w'_2$ we know that the only way for $w'_1=w'_2$ is for $w_1=w_2$. Therefore $\phi$ is injective.
\end{itemize}
By the same reasoning we can show that $\phi^{-1}$ is injective. Therefore $\phi$ is a bijection. 
\par Now consider $ls(\phi(w))$. Because $\phi$ does not change the number of occurrences of each letter, we know that $ls(\phi(w))=ls(w)$ by Lemma 2.3. Therefore $LS(13/2/4,134/2)\sim LS(13/2/4,124/3)$.
\par Next consider $rb(\phi(w))$. While $\phi(w)$ may or may not be weakly increasing we know that $\phi(w)$ has the same number of descents as $w$, because we only switched the block-singletons to singleton-blocks and we did not change anything else in the word, therefore $rb(\phi(w))=rb(w)$. So $RS(13/2/4,134/2)\sim RS(13/2/4,124/3)$.
\end{proof}

\begin{restatable}{thm}{eqSag}
\label{thm:eqSag}
The equidistribution of statistics on avoidance classes $R_n(14/2/3,1/24/3)$ and $R_n(14/2/3,13/2/4)$ are:
\begin{enumerate}
	\item $LB(14/2/3,1/24/3)\sim RS(14/2/3,1/24/3)\sim  LB(14/2/3,13/2/4)\sim RS(14/2/3,13/2/4)$
	\item $LS(14/2/3,1/24/3)\sim RB(14/2/3,13/2/4)$
	\item $RB(14/2/3,1/24/3)\sim LS(14/2/3,13/2/4)$
\end{enumerate}
\end{restatable}
\begin{proof} We refer the reader to the REU paper \cite{REU1} for the proof of (1). By the proof of Theorem 5.1 we also have proven results (2) and (3).
\end{proof}
\section{Open Problems and Future Work}
After investigating specific avoidance classes for the equidistribution of Wachs and White statistics we have many new ideas for research. 
\begin{itemize}
	\item We would like to generalize our results for the equidistribution of statistics across single pattern avoidance classes to explain why specific avoidance classes have fail to have any equidistribution results of Wachs and White statistics.
    \item While investigating avoidance classes determined by pairs of patterns we discovered an interesting property of Wilf-equivalence for some of the avoidance classes, Conjecture 6.1, and would like to formally prove this conjecture.
    \item We found many more equidistribution results for the multiple pattern avoidance classes versus our single pattern avoidance classes in this paper and would like to generalize a theorem about this equidistribution of Wachs and White statistics across avoidance classes based on pairs of patterns.
\end{itemize}
\newpage
\section{Appendix}

\EmTwo*

        	\begin{itemize}
            	\item{Claim:} If $w(\sigma)$ has at most two blocks, or $w$ is a word of the form $u3v$ where u is a prefix with $a_1$ one's and $a_2$ two's in any order allowed by the restricted growth function, and $v$ contains $b_1$ one's, $b_2$ two's, $b_3$ four's, and so on such that $b_i\leq 1$, then $\sigma$ avoids $1/2/34$.
                	\begin{proof} 
                    	\begin{itemize}
                        	\item{i)} Let $w(\sigma)$ have $m\leq 2$. Then by the definition of restricted growth function, $\sigma$ can have no more than two blocks. Then there is no subpartition of $\sigma$ that will standardize to three blocks, because any subpartition of $\sigma$ must also only have one or two blocks. Therefore because $1/2/34$ is a pattern with three blocks, there is no way that $\sigma$ can standardize to the pattern when there are at most two blocks in $\sigma$.
                            \item{ii)} Let $w(\sigma)$ have the form $u3v$ where u is a prefix with $a_1$ one's and $a_2$ two's in any order allowed by the restricted growth function, and $v$ contains $b_1$ one's, $b_2$ two's, $b_3$ four's, and so on such that $b_i\leq 1$. Then by the definition of restricted growth function, every element larger than $a_1+a_2+1$, meaning every element in $\sigma$ associated with $v$ in the word must be either in a block of its own, or is the only element, after the lone element in the third block, to be in the first or second block because one letter after $v$ can be a one and there can also be a two which would signify that the element relating to the position of the one is in the first block, and similarly with the two. So there are no blocks with two or more elements where there are two smaller elements in $\sigma$ that are found in two distinct blocks. Therefore no larger element can be repeated twice after two other elements in a subpartition where those two other elements are in separate blocks, so no subpartition of $\sigma$ will standardize to the pattern. So $\sigma$ must avoid the pattern $1/2/34$. 
                        \end{itemize}
                    \end{proof}
                \item{Claim: }If $w(\sigma)$ has more than 2 blocks and is not a word of the form $u3v$ where u is a prefix with $a_1$ one's and $a_2$ two's in any order allowed by the restricted growth function, and $v$ contains $b_1$ one's, $b_2$ two's, $b_3$ four's, and so on such that $b_i\leq 1$, then $\sigma$ contains the pattern $1/2/34$.
                	\begin{proof} Let $w(\sigma)$ have more than 2 blocks and not be a word of the form $u3v$ where u is a prefix with $a_1$ one's and $a_2$ two's in any order allowed by the restricted growth function, and $v$ contains $b_1$ one's, $b_2$ two's, $b_3$ four's, and so on such that $b_i\leq 1$. That means some element in $v$ is repeated, or there is more than one $3$ because there must be a one and a two found in the word before any larger letters by the definition of the restricted growth function so the $u$ prefix will still be valid no matter the word. Now consider if it is any letter equal to or larger than three that is repeated. If we take the elements in the partition relating to the two repeated letters in $w$, and the first element in the partition, relating to the first one in the word, along with the smallest element in the second block, relating to the first two in $w$, then this will give us a subpartition that standardizes to $1/2/34$. If a one is found more than once in $v$, then if we take the smallest element in the second block, which is associated with the first two in the word, the smallest element in the third block, which is associated with the first three in $w$, and the elements in $\sigma$ associated with the first two occurrences of one in $v$, this subpartition will standardize to $1/2/34$. Finally if a two is found more than once in $v$, then if we take a subpartition associated with the elements linked to the first two occurrences of two in $v$, the smallest element, which will be the first one in the word, and the smallest element in the third block, which is associated with the first three in the word, this subpartiton will also standardize to $1/2/34$. So when $w(\sigma)$ has more than two blocks and is not of the form $u3v$, $\sigma$ will contain $1/2/34$. Therefore for $\sigma$ to avoid $1/2/34$ it must either have no more than two blocks or be of the form $u3v$.
                
                    \end{proof}
            \end{itemize}
\EmThree*
        	\begin{itemize}
            	\item{Claim:} If $w(\sigma)$ is of the form $uv$ such that $u$ is of the form $1^{a_1}2$ and $v$ has $b_1$ ones, $b_2$ twos, and so on, such that $b_1, b_3,...,b_m\leq 2$ and $b_2\leq 1$, then $\sigma$ avoids $1/234$.
            		\begin{proof} Let $w(\sigma)$ be of the form $uv$ such that $u$ is of the form $1^{a_1}2$ and $v$ has $b_1$ ones, $b_2$ twos, and so on, such that $b_1, b_3,...,b_m\leq 2$ and $b_2\leq 1$. Then by the definition of restricted growth function, no block after the first is larger than size three, and there are no more than two elements in the first block that are larger than any element in a different block. So let us consider taking three elements from one block, this must be from the first block because all other blocks cannot have more than two elements. To get a subpartition to standardize we must have some element smaller than these three elements and in a different block. There is no way to do this because while there are at most two elements in our block of there elements with at least one element in a different block, the other element must only have elements from the first block come before it. Therefore $\sigma$ must avoid $1/234$ because there is no block of three larger elements when $w(\sigma)$ is of this form.            
            		\end{proof}
           
            	\item{Claim:} If $w(\sigma)$ is not of the form $uv$ such that $u$ is of the form $1^{a_1}2$ and $v$ has $b_1$ ones, $b_2$ twos, and so on, such that $b_1, b_3,...,b_m\leq 2$ and $b_2\leq 1$, then $\sigma$ contains $1/234$.
           			\begin{proof} Let $w(\sigma)$ not be of the form $uv$ such that $u$ is of the form $1^{a_1}2$ and $v$ has $b_1$ ones, $b_2$ twos, and so on, such that $b_1, b_3,...,b_m\leq 2$ and $b_2\leq 1$. Then some $b_i>2$ or $b_2>1$. So if we take the first element in the partition along with three elements from the block associated with $b_2$ and standardize this subpartition we will get $1/234$ using the definition of the restricted growth function. Also if we take the first element in the second block, associated with the first two in the word, and then take the first three elements in $v$ associated with the $b_i>2$, this will standardize to $1/234$. Therefore $\sigma$ contains $1/234$. So if $\sigma$ avoids $1/234$ then $w(\sigma)$ is of the form $uv$ such that $u$ is of the form $1^{a_1}2$ and $v$ has $b_1$ ones, $b_2$ twos, and so on, such that $b_1, b_3,...,b_m\leq 2$ and $b_2\leq 1$.        
                	\end{proof}
             \end{itemize}
\EmFour*        
        	\begin{itemize}
            	\item{Claim:} If $w$ is of the form $uv$ where $u$ has $a_1$ ones, $a_2$ twos, ..., $a_m$ m's such that $a_1,...,a_{m}\leq 2$ and $v$ can be any size of one letter repeated, then $\sigma$ avoids $123/4$.
                	\begin{proof} Let $w$ is of the form $uv$ where $u$ has $a_1$ ones, $a_2$ twos, ..., $a_m$ m's such that $a_1,...,a_{m}\leq 2$ and $v$ can be any size of one letter repeated. Then by the definition of restricted growth function there are no blocks of size three or larger unless it is the block with the largest element in $\sigma$. Therefore if you try to take three elements in one block together then there is no element that is larger because the block of three or more that we took these three elements from has the largest element in it so no subpartition will standardize to $123/4$. Therefore when $w$ is of the form $uv$, $\sigma$ avoids $123/4$.
                    \end{proof}
                \item{Claim:} If $w$ is not of the form $uv$ where $u$ has $a_1$ ones, $a_2$ twos, ..., $a_m$ m's such that $a_1,...,a_{m}\leq 2$ and $v$ can be any size of one letter repeated, then $\sigma$ contains $123/4$.
                	\begin{proof} Let $w$ not be of the form $uv$ where $u$ has $a_1$ ones, $a_2$ twos, ..., $a_m$ m's such that $a_1,...,a_{m}\leq 2$ and $v$ can be any size of one letter repeated. Then $v$ must contain more than one distinct letter because $v$ begins at the third occurrence of any letter in the word. This means that there is one block of size three followed by at least one more larger element that must be in a different block, by the definition of restricted growth function. Therefore if those three elements from the same block and the larger element found in a separate block are in a subpartition, they will standardize to $123/4$. Therefore when $w(\sigma)$ does not follow this given form, $\sigma$ contains $123/4$. So when $\sigma$ avoids $123/4$, $w(\sigma)$ must be of the form $uv$.
                    \end{proof}
                    \item{Cardinality}
                    \begin{proof} From Sagan we know that the cardinality of the avoidance class of a pattern is equal to the cardinality of the avoidance class of the complement of that pattern. Therefore $\#\Pi_n(123/4)=\#\Pi_n(1/234)$. Now consider $\pi=123/4$. Then any word that avoids this will be of the form $uv$ where $u$ contains $a_i$ $i$'s in any order allowed by the restricted growth function, such that $a_1,...,a_{m}\leq 2$ and $v$ can be any size of one element repeated. Then there is one way to have every $a_i=1$, this is when $u$ is strictly increasing and $|u|=n$. Now when $v$ is empty we can have letters repeated at most once. This means that we can have between one and $\lfloor\frac{n}{2}\rfloor$ pairs of letters which also means by the definition of the restricted growth function that we can have at most two elements in every block of the set partition, so we cannot have more than half the number of elements as the number of pairs. Now there are ${n\choose 2}$ ways to choose the two elements that will appear in the same block together when there is only one pair. Next we must pick the next two elements to appear in the same block together, however two elements are already assigned so we will have $n-2$ options and $2$ elements must be chosen we will multiply our first result with this result to get the total number of ways to have two pairs with all other elements in blocks of size one in the partition. This will continue for every possible number of pairs, so we will get the product of these combinations where we stop the product at each possible number of pairs and add all of these products together. Next consider when $v$ is nonempty. There are one to $\lfloor\frac{n}{2}\rfloor$ possible values for the letter in $v$. Because $v$ begins at the third repeat of any one letter it must be one of the letters that is repeated in $u$, $v$ can be any length from one to $n-3$ because the smallest $u$ possible is $122$ where $v$ contains twos. So we will have the summation of the product again, this time however let $r$ be the letter in $v$, we will multiply the combination by $r$ to get the number of ways to have a specific number of pairs with every possible value of the letter in $v$.

            \end{proof}
            \end{itemize}
\EmSeven*
        	\begin{itemize}
            	\item{Claim: } $w$ has some letters $a$, $b$, and $c$ such that these letters appear in the order $abac$ where other letters may come between if and only if $w$ contains $13/2/4$.
                \begin{proof} 
                \begin{itemize} 
                \item Let $w$ have some letters $a$, $b$, and $c$ such that these letters appear in the order $abac$ where other letters may come between. Then by the restricted growth function if we take the elements associated with these letters as a subpartition, we will have the smallest element along with the third smallest element in one block, the second smallest element in its own block, and the largest element in a separate block. This means that this subpartition will standardize to $13/2/4$.
                \item Let $w$ contain $13/2/4$. Then a subpartition of $\sigma$ must standardize to $13/2/4$. Which would mean that there must be two elements in a block together where at least one element is greater than one of these two elements and smaller than the other and is in a different block, and there is an element larger than all other elements so far that is in another different block. Let the first block of this subpartition be the $a^{th}$ block, the next block be the $b^{th}$ block, and the block containing the largest element of the subpartition will be the $c^{th}$ block. Then by the definition of the restricted growth function $abac$ must occur in the word $w$ but other letters can come between them. Therefore when $w$ contains $13/2/4$, $w$ must contains $abac$.
                \end{itemize}
                \end{proof}
            	\item{Claim:} If $m\leq 2$, $w$ is weakly increasing, $w$ takes on the form $uv$ where $u$ is weakly increasing and $v$ begins at the first $m$ and contains $m$ and $(m-1)$ in any order, or $v$ begins at the first $m$ and contains a block of $m's$ followed by a block of $z's$ where $z\leq (m-1)$, then $\sigma$ avoids $13/2/4$.
                	\begin{proof}
                    
                    	\begin{itemize}
                        	\item{i)} Let $m\leq 2$. Then this means that $w(\sigma)$ can have no letter larger than two, and $\sigma$ can have at most two blocks by the definition of the restricted growth function. Then because the pattern $13/2/4$ has three blocks, there will be no subpartition of $\sigma$ that standardizes to the pattern. Therefore when $w(\sigma)$ has no letter larger than two, $\sigma$ avoids $13/2/4$ because $13/2/4$ has three blocks.
                            \item{ii)} Let $w(\sigma)$ be of the form $uv$ where $u$ is weakly increasing and $v$ begins with the first $m$ and contains a block of $a_m$ $m's$ followed by a block of $b_z$ $z's$, where $z\leq (m-1)$. Then while the word allows for one repeated letter to have a different letter come in between them, there will be no distinct letter that occurs afterward, because no different unique letter comes after the first $z$ in $v$. Therefore $w$ contains no letters which follow $abac$. Therefore when $w(\sigma)$ is of this form, $\sigma$ avoids $13/2/4$.
                            \item{iii)} Let $w(\sigma)$ be of the form $uv$ where $u$ is weakly increasing and $v$ begins with the first $m$, then contains $m$ and $(m-1)$ in any order. Then while there may be two elements with the same letter with a unique letter coming between the two, if an element comes after these letters that is distinct from the repeated letter, it cannot be distinct from the letter that came between these repeated letters. Therefore no part of $w$ contains $abac$, so no subpartition of $\sigma$ will standardize to $13/2/4$. So when $w$ is of this form, $\sigma$ avoids $13/2/4$.
                        \end{itemize}
                    \end{proof}
                \item{Claim:} If $w(\sigma)$ has elements larger than two and does not follow the form $uv$ where $v$ begins at the first $m$ and contains a block of $m$ followed by a block of $z$ where $z<(m-1)$, or $uv$ where $v$ begins at the first $m$ and contains $m$ and $(m-1)$ in any order, then $\sigma$ contains $13/2/4$. 
                	\begin{proof} 
                    	\begin{itemize}
                        	\item{i)}Let $w$ have more than two letters and be of the form $uv$ where $v$ begins at the first $m$ and contains $m$ and $z$ in any order after this where at least one $m$ occurs after $z$, where $z\leq (m-1)$. Then if we take the first occurrence of whatever $z$ is in $u$, the next distinct letter in $u$ which will be have the value $z+1$ cannot be $m$ because $z<(m-1)$, and the first occurrence of $m$ after the first $z$ in $v$ we can see that there are three distinct letters forming $z(z+1)zm$ which follows $abac$ where other letters can come between these letters but they must occur in this order. Therefore this means that $\sigma$ must contain $13/2/4$.
                            \item{ii)}Let $v$ contain more than two distinct letters from $u$. Now consider taking the first occurrence in $u$ of the second distinct letter in $v$, the first $m$ in $v$, the first occurrence of the second unique letter in $v$, and the first occurrence of the third unique letter in $v$. This will be four letters, three of which are distinct, that follow $abac$, meaning that $\sigma$ must contain $13/2/4$. 
                            \item{iii)} Let $u$ not be weakly increasing. This means that there must be at least one descent in $w(\sigma)$. So if we take the first letter occurring out of order, the first occurrence of this letter in $u$, a distinct letter that comes between these two, and another distinct letter that follows after such as $m$, this will follow $abac$ and therefore $\sigma$ must contain $13/2/4$.
                            \end{itemize}
                    \end{proof}
                    \item{Cardinality} 
                    \begin{proof} 
        \begin{itemize}
        \item Let $m\leq 2$. Then by the Stirling numbers of the second kind, there are $S(n,1)+S(n,2)=1+2^{n-1}-1=2^{n-1}$ different set partitions.
        \item Let $w(\sigma)$ be of the form $uv$ where $v$ is empty, then $w(\sigma)$ is weakly increasing and there must be at least one element in each block of the set partition where $3\leq m\leq n$. Therefore $m$ elements are taken and there are $n-m$ elements left to assign a block. Because we know that $w$ is weakly increasing, the order these elements are assigned does not matter, for we are only interested in how large each block is. Therefore we will use the idea that there are $n-m$ elements and $m$ blocks to choose from where order does not matter and repetition is allowed, so there will be ${n-m+m-1\choose n-m}={n-1\choose n-m}$ different set partitions associated with weakly increasing words.
        \item Let $w(\sigma)$ be of the form $uv$ where there are two distinct letters in $v$. Then let us begin by focusing on $u$ which can be between $m-1$ and $n-2$ where $m$ is between $3$ and $n-1$ because at least one letter in the word is repeated. Then we know that $u$ must contain at least one letter between one and $m-1$. So there will be $|u|-m+1$ other letters in $u$ to assign a name. Because $u$ must be weakly increasing we again do not care which order these letters are assigned names in, all we care is learning how many times each letter is repeated. There are $|u|-m+1$ letters to assign names, $m-1$ possible names to assign, order does not matter and repetition is allowed, therefore ${|u|-m+1+m-1-1\choose |u|-m+1}={|u|-1\choose |u|-m+1}$ different $u's$ possible for that specific length of $u$.
        	\begin{itemize}
            \item Let $v$ begin with the first $m$ and contain a block of $m's$ followed by a block of $z's$ where $z$ is any letter less than $m-1$. Then there are $n-|u|$ letters in $v$ and there are $m-2$ options for $z$. Because $v$ has two blocks, we are also interested in the size of the blocks. For each size of $v$ there can be a block of $z's$ of size one to $n-|u|-1$. So there are $(m-2)(n-|u|-1)$ different $v's$ of this type possible for each length of $v$.
            \item Let $v$ begin with the first $m$ and contain $m$ and $(m-1)$ in any order after this. Then there are $n-|u|$ letters which can either be $m$ or $(m-1)$ where we know the first is $m$. So by the Stirling numbers of the second kind there are $S(n-|u|,2)=2^{n-|u|-1}-1$ possible $v's$.
            \end{itemize}
        \end{itemize}
       \end{proof}
            \end{itemize}
\EmEight*    	
        	\begin{itemize}
            	\item{Claim: } If $m\leq 2$, a weakly increasing word with a block of ones of any size inserted between two distinct terms, or a word beginning with one's and two's in any order followed with a weakly increasing suffix that begins with the first occurrence of three, then $\sigma$ avoids $1/24/3$.
                	\begin{proof}
                	\begin{itemize}
                    	\item{i)} Let $m\leq 2$. Then there are no more than two blocks in $\sigma$ by the definition of the restricted growth function, and no partition with only two blocks can have a subpartition with more than two blocks, meaning that $\sigma$ must avoid any pattern that has three or more blocks. Therefore $\sigma$ avoids $1/24/3$ when $w(\sigma)$ has no letter larger than two.
                        \item{ii)} Let $w(\sigma)$ be weakly increasing. Then there are no descents in $w(\sigma)$ meaning that every block in $\sigma$ only contains letters which are consecutive in each respective block. Therefore $\sigma$ cannot contain $1/24/3$ because there is no possible way to have two elements in one block with another element that smaller than one of these and larger than the other in a separate block.
                        \item{iii)} Let $w(\sigma)$ be weakly increasing with a block of ones inserted between two distinct letters. Then the only block that can contain nonconsecutive letters in $\sigma$ is the first block by the restricted growth function. However, because no element is in a block before the first block, there can be no subparition that standardizes to $1/24/3$. To standardize to this pattern there would need to be a distinct letter that occurred before the block with nonconsecutive letters. Therefore $\sigma$ avoids $1/24/3$
                        \item{iv)} Let $w(\sigma)$ contain one's and two's followed by a weakly increasing suffix which begins with the first three. Then because the only descents can occur with the one's and two's, while an element may be smaller than either elements in a different block, and these elements in this block may not be consecutive, the element that comes between the two in this block is found in the other. Therefore, $\sigma$ must avoid $1/24/3$
                    \end{itemize}
                    \end{proof}
                \item{Claim: }If $w(\sigma)$ has more than two distinct letters and is not weakly increasing with only one block of ones inserted between two distinct terms and does not contain one's and two's followed by a weakly increasing suffix that begins with the first three, then $\sigma$ contains $1/24/3$.
                	\begin{proof}
                    \begin{itemize}
                    	\item{i)} Let $w(\sigma)$ be weakly increasing with more than one block of ones inserted between distinct letters. Then by the definition of the restricted growth function, if an element from the second block is taken which will be associated with a two in the word, one element from the first block in $\sigma$ which is associated with the first group of inserted ones so the element must be larger than that associated with the two in the word, one element associated with a letter that comes between the two separate groups of inserted ones, which will be larger than the element from the first block, and finally one element from the first block of the partition that is associated with a letter from the last group of inserted ones which will be the largest element in the subpartition. Then if we standardize this subpariton we will have a small element in its own block, the second smallest and the largest elements in a different block, and the third smallest element in its own block. This standardizes to $1/24/3$. Therefore $\sigma$ contains $1/24/3$. 
                        \item{ii)} Let $w(\sigma)$ be weakly increasing with a block of ones inserted between two of the same letter where this letter is larger than two. Then using the definition of the restricted growth function, if one element associated with the distinct letter that occurs directly before this group of ones, the element that is associated with the letter that occurs directly after this group of ones, the element associated with the first two in the word, and the element associated with the first of the inserted ones in the word are taken as a subpartition of $\sigma$ this will standardize to $1/24/3$. This is because the smallest element is in its own block in the subpartition, this is the element in the second block of $\sigma$, then the second smallest element and the largest elements are in a block together, and finally the third smallest element is in a block by itself.
                        \item{iii)} Let $w(\sigma)$ have one's and two's followed by a suffix that begins with the first three and is not weakly increasing. Then by the definition of the restricted growth function, consider the subpartition based on the first element that is associated with the letter that occurs out of order in the suffix of the word where this letter is greater than one, the element associated with the first occurrence of this letter in the word, the element associated with the letter that comes directly before the first letter that occurs out of order, and an element associated with the first one in the word. Then by the definition of standardization this subpartition becomes $1/24/3$. Now if the letter that occurs out of order is a one in the suffix, then take the element associated with the first one in the suffix, take the element associated with the first one after a two in the prefix of the word, also take the element associated with the first two in the word, and the element associated with the first three in the word. This will also standardize to $1/24/3$. Therefore $\sigma$ contains $1/24/3$.
                    \end{itemize}
                    \end{proof}
            \end{itemize}
\lsww*
    	\par There are two clear cases, either the word is weakly increasing, otherwise there is one or more descents, which will cause the word to be of the form $1^{a_1}{w_1}1^{a_2}{w_2}1^{a_3}{w_3}1^{a_4}$ where $w_1$ has $b_1$ one's and $b_2$ two's where $w_1$ must begin and end with a two, $w_2$ has $c_1$ one's, $c_2$ two's, $c_3$ three's such that $w_2$ begins with the first three and ends with the last two or with the last three, whichever comes first, and finally $w_3$ begins with a three (two) and ends with a three (two) such that $w_3$ has $d_1$ one's, ($d_2$ two's), and $d_3$ three's.
    	\begin{itemize}
        	\item{Left and Smaller}
            	\[
  ls=\begin{cases}
               2h_3+h_2\\
               b_2+c_2+2c_3+d_2+2d_3
               
            \end{cases}
\]
            \begin{proof} Let $m\leq 3$ where $w$ is weakly increasing. Then there must be two elements to the left and smaller than each $3$ found in $w(\sigma)$ because a one and a two must occur at least once before any three appears. There must be one elements to the left and smaller than each $2$, because a one must appear once before any $2$ appears by the definition of the restricted growth function. Therefore the left and smaller statistic will be $2h_3+h_2$ where $h_3$ is the number of threes and $h_2$ is the number of ones.
            \end{proof}
            \begin{proof} Let $w(\sigma)$ be of the form $1^{a_1}{w_1}1^{a_2}{w_2}1^{a_3}{w_3}1^{a_4}$. There are $b_2$ two's in $w_1$, $c_2$ two's in $w_2$, and $d_2$ two's in $w_3$. There must be at least a one that occurs before any two appears, therefore there will be $b_2+c_2+d_2$ elements with a smaller element that occurs before it. There are $c_3$ three's in $w_2$ and $d_3$ three's in $w_2$. There must be two smaller elements, a one and two must occur before any three appears in the word, therefore the three's contribute $2c_3+2d_3$. Therefore the left and smaller statistic will be $b_2+c_2+d_2+2c_3+2d_3$.
            \end{proof}
            \item {Left and Bigger}
            	\[
  lb=\begin{cases}
               0\\
              a_2+2a_3+2a_4+b_1+2c_1+2d_1
            \end{cases}
\]
			\begin{proof} 
            \begin{itemize}
            \item Let $w(\sigma)$ be weakly increasing. Then there are no descents in the word, so no letter in the word will have another letter occur to the left and be larger. Therefore $lb=0$.
            Let $w(\sigma)$ be of the form $1^{a_1}{w_1}1^{a_2}{w_2}1^{a_3}{w_3}1^{a_4}$, then for the left and bigger statistic, every one that occurs after the first two contributes to the statistic. There are $a_2+a_3+a_4+b_1+c_1+d_1$ ones that occur after the first two. Every one that occurs after the first three also contributes one to the statistic. There are $a_3+a_4+c_1+d_1$ ones that occur after the first three. Every two that occurs after the first three also contributes one to the statistic. There are $c_2+d_2$ twos that occur after the first three. Therefore the right and bigger statistic is $a_2+2a_3+2a_4+b_1+2c_1+c_2+2d_1+d_2$.
            \end{itemize}
            \end{proof}
            \rww*
            \item{Right and Bigger}
            
			\begin{proof} Let $w(\sigma)$ have no more than three blocks that is weakly increasing. Then there are $(m-1)$ distinct letters larger than each one and there are $h_1$ ones. Then there are $(m-2)$ distinct letters larger than each two and there are $h_2$ twos. Therefore the right and bigger statistic is $(m-1)h_1+(m-2)h_2$.
            \end{proof}
            \begin{proof} Let the last two be found in $w_1$ and the last three be found in $w_2$. Then there are $a_1+b_1$ ones that come before the last two, $a_1+a_2+b_1+c_1$ ones that come before the last three, and $b_2+c_2$ twos that come before the last three. Therefore the right and bigger statistic will be $2a_1+a_2+2b_1b_2+c_1+c_2$.
            \end{proof}
            \begin{proof} Let the last two be found in $w_2$ and the last three be found in $w_3$. Then there are $a_1+a_2+b_1+c_1$ ones that occur before the last two, $a_1+a_2+a_3+b_1+c_1+d_1$ ones that occur before the last three, and $b_2+c_2$ twos that occur before the last three. Therefore the right and bigger statistic will be $2a_1+2a_2+a_3+2b_1+b_2+2c_1+c_2+d_1$.
            \end{proof}
            \begin{proof}Let the last two be found in $w_3$ and the last three be found in $w_2$. Then there are $a_1+a_2+a_3+b_1+c_1+d_1$ ones that occur before the last two, $a_1+a_2+b_1+c_1$ ones that occur before the last three, and $b_2+c_2$ ones that occur before the last three. Therefore the right and bigger statistic will be $2a_1+2a_2+a_3+2b_1+b_2+2c_1+c_2+d_1$.
            \end{proof}
            
            \item{Right and Smaller}
          
  			\begin{itemize}
            \item Contribution based on the last one:
            	\begin{itemize}
                	\item If the last one occurs before $w_1$, then there will be no contribution to the $rs$ statistic.
                    	\begin{proof} Let the last one be before $w_1$. Then there are no ones that occur after any two, and there are no ones that occur after any three, therefore there will be no ones to the left and smaller than any letter in the word. So there will be no contribution to the $rs$ statistic based on the last one.
                        \end{proof}
                    \item If the last one occurs in $w_1$, then this will cause a contribution of $p_1-b_1$ where $p_1$ is the position of the last one in $w_1$.
                    \begin{proof}
                    Let the last one occur in $w_1$. When the last one is in $w_1$ there will be the position that last one occurred in relation to $w_1$, denoted by $p_1$, minus the number of ones that occurred in $w_1$, denoted by $b_1$, $p_1-b_1$ will give us the number of two's that occurred before the last one, each of these two's will contribute one to the right and smaller statistic. Therefore when the last one is in $w_1$ there are $p_1-b_1$ contributions to the right and smaller statistic.
                    \end{proof}
                    \item If the last one occurs between $w_1$ and $w_2$, then this will cause a contribution of $b_1$ to the $rs$ statistic for the word.
                    \begin{proof} Let the last one occur between $w_1$ and $w_2$. There will be $b_2$ two's that occur before this last one, so there are $b_2$ objects with something smaller that occurs after them. Therefore each of these two's will contribute one to the right and smaller statistic, so when the last one occurs between $w_1$ and $w_2$ this causes $b_2$ to be contributed to the $rs$ statistic.
                    \end{proof}
                    \item If the last one occurs in $w_2$, then this will cause $b_2+p_2-c_1$ to be contributed to the $rs$ statistic. 
                    \begin{proof} Let the last one occur in $w_2$, then there will be $b_2$ two's in $w_1$ that occur this last one and will each contribute one to the $rs$ statistic. There are also $p_2-c_1$ two's and three's in $w_2$ that occur before the last one because there are a total of $c_1$ ones in $w_2$ and the last one is in position $p_2$ so there will be $p_2-c_1$ other letters that are larger than one and come before the last one in the word, each of these will also contribute one to the right and smaller statistic. Therefore this contributes $b_2+p_2-c_1$.
                    \end{proof}
                    \item If the last one occurs after $w_2$ but before $w_3$, then this will contribute $b_2+c_2+c_3$ to the $rs$ statistic.
                    \begin{proof} Let the last one occur between $w_2$ and $w_3$. Then there are $b_2$ two's that occur in $w_1$ and each of these occur before the last one. There are also $c_2$ two's that occur in $w_2$ that also occur before the last one, and each of these will contribute one to the right and smaller statistic. Finally there are $c_3$ three's that occur before the last one and each of these will contribute one to the right and smaller statistic.
                    \end{proof}
                    \item If the last one occurs in $w_3$, then this will contribute $b_2+c_2+c_3+p_3-d_1$.
                    \begin{proof} Let the last one occur in $w_3$. Then there are $b_2+c_2$ two's that occur in $w_1$ and $w_2$ that must occur before the last one, each will contribute one to the right and smaller statistic. There are $c_3$ three's that occur in $w_2$ which all occur before the last one, so each of these contribute one to the right and smaller statistic. There are also $p_3-d_1$ two's and three's that occur before the last one in $w_3$, where $p_3$ is the position of the last one in $w_3$ and there are $d_1$ ones in $w_3$ so there are $p_3-d_1$ two's and three's that occur before the last one.
                    \end{proof}
                    \item If the last one occurs after $w_3$, then this causes a contribution of $b_2+c_2+c_3+d_2+d_3$ to the right and smaller statistic.
                    \begin{proof} Let the last one occur after $w_3$. Then there are $b_2+c_2+d_2$ two's that occur before the last one, so each of these will contribute one to the right and smaller statistic. There are also $c_3+d_3$ three's that occur before the last one and each of these will contribute one to the right and smaller statistic.
                    \end{proof}
                \end{itemize}
                \item{Contribution based on the last two:}
                \begin{itemize} 
                	\item If the last two is in $w_1$, then this will cause a contribution of $0$ to the right and smaller statistic.
                    \begin{proof} Let the last two occur in $w_1$. There are no three's that occur before or in $w_1$ so there are no letters that are larger than two to contribute to the right and smaller statistic.
                    \end{proof}
                    \item If the last two is in $w_2$, then this will cause a contribution of $c_3$ to the right and smaller statistic.
                    \begin{proof} Let the last two be in $w_2$, then the only larger letter could be a three. We know that $w_2$ will end with the last two by how we defined the word. So there are a total of $c_3$ three's in $w_2$ that must occur before the last two. Each of these will contribute one to the right and smaller statistic. Therefore when the last two is in $w_2$ this contributes $c_3$ to the $rs$ statistic.
                    \end{proof}
                    \item If the last two is in $w_3$, then this will cause a contribution of $c_3+d_3$ to the right and smaller statistic.
                    \begin{proof} Let the last two be in $w_3$. Then there are $c_3$ three's in $w_2$ which must occur before it, because two is present in $w_3$ we know that $d_3$ must be zero. Therefore there are $c_3+d_3$ elements that will contribute one each to the right and smaller statistic.
                    \end{proof}
                \end{itemize}
            
            \end{itemize}
        \end{itemize}

        \wwTwo*
    	\begin{itemize}
        	\item{Left and Smaller: } 
            $$ls=a_2+b_2+2+\sum\limits^{m-1}_{i=3}i$$
            \begin{proof} Let $w(\sigma)$ be of the form $u3v$. Then there will be $a_2+b_2$ two's after the first one. There is also a three and a two and a one must come before that three, so this will contribute 2 to the left and smaller statistic. For every letter larger than three there will be one less than that element total distinct letters smaller. Therefore $ls=a_2+b_2+2+\sum\limits^{m-1}_{i=3}i$. 
            \end{proof}
            \item{Right and Bigger}
            \[
  rb=\begin{cases}
               (m-1)a_1+(m-2)a_2+|v|-y+\sum\limits^m_{i=3}(m-i)\\
               (h-a_2)(m-1)+(m-2)(|u|-h)+(m-2)a_2+|v|-x+\sum\limits^m_{i=3}(m-i)\\
               (m-1)a_1+(m-2)a_2+2|v|-x-y+\sum\limits^m_{i=3}(m-i)\\
               (m-1)a_1+(m-2)a_2+2|v|-x-y-1+\sum\limits^m_{i=3}(m-i)\\
            \end{cases}
\]
			\begin{proof} Let $w$ be of the form $u3v$. 
            \begin{itemize}
            \item Let $b_1=0$ and $b_2=1$. Then there is a two present in $v$ but no one in $v$. Then there will be $(m-1)$ letters to the right and bigger than each one present in the word, because there is a two present in $v$ we know that a two will be present after every one in the word so the last two in $u$ does not impact this contribution to our statistic, there are a total of $a_1$ ones. There are also $(m-2)$ letters to the right and bigger than each two present in $u$ and $a_2$ two's in $u$. Because there is a two in $v$ let $y$ denote the placement of this $2$. Then there are $|v|-y$ distinct letters that come after the two in $v$ that must be larger than two. Also every letter three or larger has $m-i$ letters larger than it where $i$ is the letter value. Therefore $lb=(m-1)a_1+(m-2)a_2+|v|-y+\sum\limits^m_{i=3}(m-i)$. 
            \item Let $b_2=0$ and let $h$ be the position of the last $2$ in $u$. Then there are $h-a_2$ ones that occur before the last two in $u$. Each of these one's will have $(m-1)$ letters to the right and bigger than $1$ in the word. There will be $(|u|-h)$ ones that occur in $u$ after the last two, these ones will have the letters three to $m$ larger than them and occurring to the right, so each one has $(m-2)$ letters to the right and bigger. Now there will be $a_2$ two's in the word and there are a total of $m-2$ distinct letters to the right and bigger. Next there are a total of $|v|-x$ letters that occur after the one in $v$, all of these must be distinct and none can be one, therefore each is to the right and bigger than this one. Finally all other letters between three and $m$ will have $m-i$ letters larger than them where $i$ is the value of the letter and each letter between three and $m$ can only occur once in the word, therefore this contributes $\sum\limits^m_{i=3}(m-i)$ to the right and bigger statistic. So $rb=(h-a_2)(m-1)+(m-2)(|u|-h)+(m-2)a_2+|v|-x+\sum\limits^m_{i=3}(m-i)$.
            \item Let $x<y$ where $b_1=b_2=1$. Then there are $(m-1)$ letters that are larger and come after every one in $u$ because there is a two in $v$ every one in $u$ will have the letters two through $m$ after it. There are $a_1$ ones in $u$. Now there are $(m-2)$ letters in $w$ that are distinct and are larger than two that come after every two in $u$, specifically the letters three through $m$. There are a total of $a_2$ two's in $u$. When $x<y$ that means that the one in $v$ comes before the $2$ in $v$. Then every letter that comes after the one in $v$ is larger than one, so there are $|v|-x$ letters that are distinct and to the right of this one. Every letter that comes after the two in $v$ will larger than two because the one is found earlier in $v$. Therefore there will be $|v|-y$ letters to the right and bigger than the two in $v$. Therefore $rb=(m-1)a_1+(m-2)a_2+2|v|-x-y$.
            \item Let $x>y$ where $b_1=b_2=1$. Then there are $(m-1)$ letters that are larger and come after every one in $u$ because there is a two in $v$ every one in $u$ will have the letters two through $m$ after it. There are $a_1$ ones in $u$. Now there are $(m-2)$ letters in $w$ that are distinct and are larger than two that come after every two in $u$, specifically the letters three through $m$. There are a total of $a_2$ two's in $u$. When $x>y$ that means that the one in $v$ occurs after the two in $v$. So there will be $|v|-x$ letters to the right and bigger than the one in $v$. There will be $|v|-y-1$ letters to the right and bigger than the two in $v$ because there are $|v|-1$ letters which come after, one of which is a one which is less than two. Therefore $rb=(m-1)a_1+(m-2)a_2+2|v|-x-y-1$.
            
            \end{itemize}
            \end{proof}
            \item{Left and Bigger}
            $$lb=a_1-z+x+y+1$$
            \begin{proof} Let $z$ be the position of the first $2$. Then there are $|u|-z$ letters in $u$ that come after the first two. Of these letters we must remove all other two's to get the number of ones that come after the fist two, so this would be $|u|-z-(a_2-1)=a_1+a_2-z-a_2+1=a_1-z+1$. So $a_1-z+1$ ones come after the first two and each of these will have one distinct letter to the left and bigger than them. The next thing to consider is the one in $v$. Again let $x$ be the position of the one in $v$. Then there will be $x+1$ letters to the left and bigger, the two's present in $u$, the three and all of the letters that occur before one in $v$. And let $y$ be the position of the two in $v$. This two will have three to the left and bigger, it will also have every letter from the beginning of $v$ until the occurrence of the letter larger than it if $x>y$, this would give us $v$ letters to the left and bigger than the two present in $v$. Now consider if $x>y$. Then that two in $v$ is not distinct from the two in $u$ for the $lb$ contribution of the one in $v$, every element in $3v$ that comes before $1$ is larger than it, so this will have an $lb$ contribution of $x$. Now when considering the two in $v$, this comes before the one, so every element in $v$ that comes before this one is also to the left and bigger, and each of these letters in $v$ must be distinct. Therefore $lb=a_1-z+x+y+1$
            \end{proof}
            \item{Right and Smaller}
            \[
 rs=\begin{cases}
 			a_2+x+y-1\\
            a_2+x+y\\
            (l-a_1)+y
         \end{cases}
\]
\begin{proof} 
\begin{itemize}
	\item Let one and two both be present in $v$, and let the position of the one, also referred to as $x$, be less than the position of the two, referred to as $y$. Then all of the two's found in $u$ will have a one to the right of them. There are $a_2$ two's of this type and each will contribute one to the right and smaller statistic. The next possible contribution to the right and smaller statistic is based on the placement of the one in $v$. When one is in the $x^{th}$ position of $v$ then there are $x-1$ letters in $v$ that come before the one and are larger than one. Also there is the three which comes before $v$ and is larger than one. Each of these letters will then contribute one to the right and smaller statistic, so this will be $x-1+1=x$. Finally consider the contribution based on the two in $v$. When $x<y$ there are $y-1$ letters that come before two in $v$, however one of these is a one, which is less than two, so there are $y-2$ letters larger than two that come before. There is also a three which must come before this two. Because each of these cases will contribute one to the $rs$ statistic, the whole statistic becomes $a_2+x+y-1$. No other contributions are possible because these are the last two possible descents in the word, and you must have a descent present after a letter for that letter to possibly contribute to the $rs$ statistic.
    \item Let one and two both be present in $v$ where $y<x$. Then because there is a one in $v$, each two in $u$ has one distinct letter to the right and smaller, contributing $a_2$ to the $rs$ statistic. Next because the two comes before the one in $v$ there will be $y-1$ letters to the left of two that are larger from $v$, and there is also the three, causing a contribution of $y$ to the $rs$ statistic. Finally consider the one, there are $x-1$ letters that come before one and are larger, because two is also larger, and there is the one three found before $v$, so this will contribute $x$ to the $rs$ statistic. Therefore $rs=a_2+x+y$.
    \item Let $b_1=0$ and let $l$ denote the last one present in $u$. There will be $l-a_1$ two's that appear before the last one, each of these must contribute one to the $rs$ statistic. Then when there is a two present in $v$ there are $y-1$ letters that come before it and are larger in $v$, and there is one three that occurs before $v$. Therefore $rs=(l-a_1)+y$. If there is no two present in $v$, then $y=0$ and the $rs$ statistic is still $(l-a_1)+y$.
\end{itemize}
\end{proof}
        \end{itemize}
 
    \section{The Pattern 1/2/34}
    	\begin{itemize}
        	\item{Left and Smaller: } 
            $$ls=a_2+b_2+2+\sum\limits^{m-1}_{i=3}i$$
            \begin{proof} Let $w(\sigma)$ be of the form $u3v$. Then there will be $a_2+b_2$ two's after the first one. There is also a three and a two and a one must come before that three, so this will contribute 2 to the left and smaller statistic. For every letter larger than three there will be one less than that element total distinct letters smaller. Therefore $ls=a_2+b_2+2+\sum\limits^{m-1}_{i=3}i$. 
            \end{proof}
            \item{Right and Bigger}
            \[
  rb=\begin{cases}
               (m-1)a_1+(m-2)a_2+|v|-y+\sum\limits^m_{i=3}(m-i)\\
               (h-a_2)(m-1)+(m-2)(|u|-h)+(m-2)a_2+|v|-x+\sum\limits^m_{i=3}(m-i)\\
               (m-1)a_1+(m-2)a_2+2|v|-x-y+\sum\limits^m_{i=3}(m-i)\\
               (m-1)a_1+(m-2)a_2+2|v|-x-y-1+\sum\limits^m_{i=3}(m-i)\\
            \end{cases}
\]
			\begin{proof} Let $w$ be of the form $u3v$. 
            \begin{itemize}
            \item Let $b_1=0$ and $b_2=1$. Then there is a two present in $v$ but no one in $v$. Then there will be $(m-1)$ letters to the right and bigger than each one present in the word, because there is a two present in $v$ we know that a two will be present after every one in the word so the last two in $u$ does not impact this contribution to our statistic, there are a total of $a_1$ ones. There are also $(m-2)$ letters to the right and bigger than each two present in $u$ and $a_2$ two's in $u$. Because there is a two in $v$ let $y$ denote the placement of this $2$. Then there are $|v|-y$ distinct letters that come after the two in $v$ that must be larger than two. Also every letter three or larger has $m-i$ letters larger than it where $i$ is the letter value. Therefore $lb=(m-1)a_1+(m-2)a_2+|v|-y+\sum\limits^m_{i=3}(m-i)$. 
            \item Let $b_2=0$ and let $h$ be the position of the last $2$ in $u$. Then there are $h-a_2$ ones that occur before the last two in $u$. Each of these one's will have $(m-1)$ letters to the right and bigger than $1$ in the word. There will be $(|u|-h)$ ones that occur in $u$ after the last two, these ones will have the letters three to $m$ larger than them and occurring to the right, so each one has $(m-2)$ letters to the right and bigger. Now there will be $a_2$ two's in the word and there are a total of $m-2$ distinct letters to the right and bigger. Next there are a total of $|v|-x$ letters that occur after the one in $v$, all of these must be distinct and none can be one, therefore each is to the right and bigger than this one. Finally all other letters between three and $m$ will have $m-i$ letters larger than them where $i$ is the value of the letter and each letter between three and $m$ can only occur once in the word, therefore this contributes $\sum\limits^m_{i=3}(m-i)$ to the right and bigger statistic. So $rb=(h-a_2)(m-1)+(m-2)(|u|-h)+(m-2)a_2+|v|-x+\sum\limits^m_{i=3}(m-i)$.
            \item Let $x<y$ where $b_1=b_2=1$. Then there are $(m-1)$ letters that are larger and come after every one in $u$ because there is a two in $v$ every one in $u$ will have the letters two through $m$ after it. There are $a_1$ ones in $u$. Now there are $(m-2)$ letters in $w$ that are distinct and are larger than two that come after every two in $u$, specifically the letters three through $m$. There are a total of $a_2$ two's in $u$. When $x<y$ that means that the one in $v$ comes before the $2$ in $v$. Then every letter that comes after the one in $v$ is larger than one, so there are $|v|-x$ letters that are distinct and to the right of this one. Every letter that comes after the two in $v$ will larger than two because the one is found earlier in $v$. Therefore there will be $|v|-y$ letters to the right and bigger than the two in $v$. Therefore $rb=(m-1)a_1+(m-2)a_2+2|v|-x-y$.
            \item Let $x>y$ where $b_1=b_2=1$. Then there are $(m-1)$ letters that are larger and come after every one in $u$ because there is a two in $v$ every one in $u$ will have the letters two through $m$ after it. There are $a_1$ ones in $u$. Now there are $(m-2)$ letters in $w$ that are distinct and are larger than two that come after every two in $u$, specifically the letters three through $m$. There are a total of $a_2$ two's in $u$. When $x>y$ that means that the one in $v$ occurs after the two in $v$. So there will be $|v|-x$ letters to the right and bigger than the one in $v$. There will be $|v|-y-1$ letters to the right and bigger than the two in $v$ because there are $|v|-1$ letters which come after, one of which is a one which is less than two. Therefore $rb=(m-1)a_1+(m-2)a_2+2|v|-x-y-1$.
            
            \end{itemize}
            \end{proof}
            \item{Left and Bigger}
            $$lb=a_1-z+x+y+1$$
            \begin{proof} Let $z$ be the position of the first $2$. Then there are $|u|-z$ letters in $u$ that come after the first two. Of these letters we must remove all other two's to get the number of ones that come after the fist two, so this would be $|u|-z-(a_2-1)=a_1+a_2-z-a_2+1=a_1-z+1$. So $a_1-z+1$ ones come after the first two and each of these will have one distinct letter to the left and bigger than them. The next thing to consider is the one in $v$. Again let $x$ be the position of the one in $v$. Then there will be $x+1$ letters to the left and bigger, the two's present in $u$, the three and all of the letters that occur before one in $v$. And let $y$ be the position of the two in $v$. This two will have three to the left and bigger, it will also have every letter from the beginning of $v$ until the occurrence of the letter larger than it if $x>y$, this would give us $v$ letters to the left and bigger than the two present in $v$. Now consider if $x>y$. Then that two in $v$ is not distinct from the two in $u$ for the $lb$ contribution of the one in $v$, every element in $3v$ that comes before $1$ is larger than it, so this will have an $lb$ contribution of $x$. Now when considering the two in $v$, this comes before the one, so every element in $v$ that comes before this one is also to the left and bigger, and each of these letters in $v$ must be distinct. Therefore $lb=a_1-z+x+y+1$
            \end{proof}
            \item{Right and Smaller}
            \[
 rs=\begin{cases}
 			a_2+x+y-1\\
            a_2+x+y\\
            (l-a_1)+y
         \end{cases}
\]
\begin{proof} 
\begin{itemize}
	\item Let one and two both be present in $v$, and let the position of the one, also referred to as $x$, be less than the position of the two, referred to as $y$. Then all of the two's found in $u$ will have a one to the right of them. There are $a_2$ two's of this type and each will contribute one to the right and smaller statistic. The next possible contribution to the right and smaller statistic is based on the placement of the one in $v$. When one is in the $x^{th}$ position of $v$ then there are $x-1$ letters in $v$ that come before the one and are larger than one. Also there is the three which comes before $v$ and is larger than one. Each of these letters will then contribute one to the right and smaller statistic, so this will be $x-1+1=x$. Finally consider the contribution based on the two in $v$. When $x<y$ there are $y-1$ letters that come before two in $v$, however one of these is a one, which is less than two, so there are $y-2$ letters larger than two that come before. There is also a three which must come before this two. Because each of these cases will contribute one to the $rs$ statistic, the whole statistic becomes $a_2+x+y-1$. No other contributions are possible because these are the last two possible descents in the word, and you must have a descent present after a letter for that letter to possibly contribute to the $rs$ statistic.
    \item Let one and two both be present in $v$ where $y<x$. Then because there is a one in $v$, each two in $u$ has one distinct letter to the right and smaller, contributing $a_2$ to the $rs$ statistic. Next because the two comes before the one in $v$ there will be $y-1$ letters to the left of two that are larger from $v$, and there is also the three, causing a contribution of $y$ to the $rs$ statistic. Finally consider the one, there are $x-1$ letters that come before one and are larger, because two is also larger, and there is the one three found before $v$, so this will contribute $x$ to the $rs$ statistic. Therefore $rs=a_2+x+y$.
    \item Let $b_1=0$ and let $l$ denote the last one present in $u$. There will be $l-a_1$ two's that appear before the last one, each of these must contribute one to the $rs$ statistic. Then when there is a two present in $v$ there are $y-1$ letters that come before it and are larger in $v$, and there is one three that occurs before $v$. Therefore $rs=(l-a_1)+y$. If there is no two present in $v$, then $y=0$ and the $rs$ statistic is still $(l-a_1)+y$.
\end{itemize}
\end{proof}
        \end{itemize}
  
    \wwSeven*
     \par There are again three distinct cases where $\sigma$ avoids the pattern $13/2/4$.
     
    \begin{itemize}
    \item{Case 1: } When no letter in $w(\sigma)$ is larger than two. (Proven in REU Paper). Refer to section 10 for formulas. 
    \item{Case 2: } When $w$ is of the form $uv$ such that $z\leq m-1$.
    	\begin{itemize}
        \item{Left and Bigger Statistic}
        $$lb=(m-z)b_{z}$$
        \begin{proof} Let $w$ be of the form $uv$ such that $z\leq m-1$ and let $b_z$ be the number of $z's$ in $v$. Then there will be a block of $z's$ at the end of the word, however the rest of the word is weakly increasing. Each letter in this block will have $m-z$ letters that come before it in the word and that are larger, no other letter in the word will have a larger letter come before it. Therefore there are $b_z$ such letters that will each have $m-z$ distinct letters to the left and bigger so $lb=(m-z)b_z$.
        \end{proof}
        \item{Left and Smaller Statistic}
        $$ls=(z-1)b_z+\sum\limits^m_{i=2}(i-1)a_i$$
        \begin{proof} Let $w$ be of the form $uv$ such that $z\leq m-1$ and let $b_z$ be the number of $z's$ in $v$. Then there will be a total of $z-1$ letters to the left and smaller than each $z$ found in $v$ and because there are $b_z$ $z's$ this will contribute $(z-1)b_z$ to the left and smaller statistic. Every other letter will have one less than the value of that specific letter to the left and smaller than it, and will have $a_i$ occurrences of the letter, where $i$ is the value of the letter. In this case one will not have any letters to the left and smaller because one is the smallest possible letter in $w$. So $i$ can be between $2$ and $m$. Therefore the left and smaller statistic is $ls=(z-1)b_z+\sum\limits^m_{i=2}(i-1)a_i$.
        \end{proof}
        \item{Right and Bigger Statistic}
        $$rb=\sum\limits^{m-1}_{i=1}(m-i)a_i$$
        \begin{proof}
        Let $w$ be of the form $uv$ where $z\leq m-1$. Because $u$ must be weakly increasing by our word characterization earlier, there must be $m-i$ letters larger than each letter of value $i$ to the right of that letter in the word. In $w$ $i$ can be any value between one and $m-1$ because $m$ will not have any letters larger than it. There will be $a_i$ $i's$, therefore $rb=\sum\limits^{m-1}_{i=1}(m-i)a_i$.
        \end{proof}
        \item{Right and Smaller Statistic}
        $$rs=\sum\limits^m_{i=z+1}a_i$$
        \begin{proof}
        Let $w$ be of the form $uv$ where $z\leq m-1$. Because $u$ must be weakly increasing by the definition of our word characterization earlier, the only descent in the word can be from the block of $z's$ at the end of $v$. Then every element larger than $z$ will have one distinct letter to the right and smaller. This will include every letter between $z+1$ and $m$, so we must count the total occurrences, which gives us $rs=\sum\limits^m_{i=z+1}a_i$. 
        \end{proof}
        \end{itemize}
    \item{Case 3: } When $w$ is of the form $uv$ such that $z=m-1$.
    	\begin{itemize}
        \item{Left and Bigger Statistic}
        $$lb=(m-z)b_z$$
        \begin{proof} Let $w$ be of the form $uv$ where $z=m-1$. Then by the word characterization earlier, we know that $m$ and $z$ can occur in any order in $v$ as long as $m$ occurs first. There will be $m-z$ letters to the left and bigger than every $z$ in $v$ and there are $b_z$ $z's$, so this contributes $(m-z)b_z$ to the left and bigger statistic. 
        \end{proof}
        \item{Left and Smaller Statistic}
        $$ls=(z-1)b_z+\sum\limits^m_{i=2}(i-1)a_i$$
        \begin{proof} Let $w$ be of the form $uv$ such that $z=m-1$ and let $b_z$ be the number of $z's$ in $v$. Then there will be a total of $z-1$ letters to the left and smaller than each $z$ found in $v$ and because there are $b_z$ $z's$ this will contribute $(z-1)b_z$ to the left and smaller statistic. Every other letter will have one less than the value of that specific letter to the left and smaller than it, and will have $a_i$ occurrences of the letter, where $i$ is the value of the letter. In this case one will not have any letters to the left and smaller because one is the smallest possible letter in $w$. So $i$ can be between $2$ and $m$. Therefore the left and smaller statistic is $ls=(z-1)b_z+\sum\limits^m_{i=2}(i-1)a_i$.
        \end{proof}
        \item{Right and Bigger Statistic}
        $$rb=(x-a_m)+\sum\limits^{m-1}_{i=1}(m-i)a_i$$
        \begin{proof} Let $w$ be of the form $uv$ where $z=m-1$ and let $x$ be the position of the last $m$ in $v$. Then there will be $(x-a_m)$ $z's$ that occur before the last $m$ and each of these contributes one to the right and bigger statistic. Now because $u$ must be weakly increasing, every other letter in the word has $m-i$ distinct letters larger than it where $i$ is the value of the letter and there are $a_i$ of each $i^{th}$ letter where $i$ can be anything between one and $m-1$ because $m$ has no letter larger than it in the word. Therefore $rb=(x-a_m)+\sum\limits^{m-1}_{i=1}(m-i)a_i$.
        \end{proof}
        \item{Right and Smaller Statistic}
        $$rs=(y-b_z)$$
        \begin{proof} Let $w$ be of the form $uv$ where $z=m-1$ and let $y$ be the last occurrence of $m-1$ in $v$. Then there will $(y-b_z)$ $m's$ that come before the last $m-1$ and therefore each of these $m's$ will have one distinct letter to the right and smaller. Therefore $rs=(y-b_z)$.
        \end{proof}
        \end{itemize}
    \end{itemize}
\wwEight*
   There are three cases.
   \begin{itemize}
   \item{Case 1: } Same as the $13/2/4$ case 1 where the largest letter is at most two \cite{REU1}.
   \item{Case 2: } When $w$ is weakly increasing except for a block of ones inserted between two distinct letters.
   		\begin{itemize}
        \item{Left and Bigger Statistic}
        $$lb=(x-1)b_1$$
        \begin{proof} Let $w$ be weakly increasing except for a block of ones inserted between two distinct letters. Also let $x$ denote the distinct letter immediately before the first inserted one, and $b_1$ denote the number of ones inserted into $w$. There will be $x-1$ distinct letters that come before this block of ones that are larger than one and there are $b_1$ one's inserted, this is the only possible descent. Therefore $lb=(x-1)b_1$.
        \end{proof}
        \item{Left and Smaller Statistic}
        $$ls=\sum\limits^m_{i=1}(i-1)a_i$$
        \begin{proof} Every letter will have one less than the value of that specific letter to the left and smaller than it, and will have $a_i$ occurrences of the letter, where $i$ is the value of the letter. So $ls=\sum\limits^m_{i=1}(i-1)a_i$.
        \end{proof}
        \item{Right and Bigger Statistic}
        $$rb=\sum\limits^m_{i=1}(m-i)a_i+(m-x)b_1$$
        \begin{proof} Let $w$ be weakly increasing except for a block of ones inserted between two distinct letters. Also let $x$ denote the distinct letter immediately before the first inserted one, and $b_1$ denote the number of ones inserted into $w$. Then there are $m-x$ letters that are larger than one and are found after the block of inserted ones, so each inserted one will have $m-x$ letters to the right and bigger. Because there are $b_1$ such letters this will contribute $(m-x)b_1$ to the $rb$ statistic. For every other letter in the word there will be $m-i$ where $i$ is the value of the letter, distinct letters larger than $i$, and there are $a_i$ occurrences of each letter $i$ in the word, where $i$ can be anything between one and $m-1$ because $m$ cannot have any letter larger than it. Therefore $rb=\sum\limits^m_{i=1}(m-i)a_i+(m-x)b_1$.
        \end{proof}
        \item{Right and Smaller Statistic}
        $$rs=\sum\limits^x_{i=2}a_i$$
        \begin{proof} Let $w$ be weakly increasing except for a block of ones inserted between two distinct letters. Also let $x$ denote the distinct letter immediately before the first inserted one, and $b_1$ denote the number of ones inserted into $w$. Every letter larger than one that occurs before the block of inserted ones will contribute one to the right and smaller statistic. Therefore every occurrence of the letters two through $x$ will contribute one to the right and smaller statistic. So $rs=\sum\limits^x_{i=2}a_i$.
        \end{proof}
        \end{itemize}
    \item{Case 3: } When $w$ contains a prefix $u$ of ones and twos in any order followed by a weakly increasing suffix.
    	\begin{itemize}
        \item{Left and Bigger Statistic}
        $$lb=a_1-c+1$$
        \begin{proof} Let $w$ contain a prefix $u$ of ones and two's in any order followed by a weakly increasing suffix. Then let $c$ be the position of the first two in $w$. Then $|u|=a_1+a_2$. We want the number of ones that occur after the first two because each of these will contribute one to the left and bigger statistic. Nothing else will contribute because every thing after $u$ is weakly increasing and contains letters three or greater. So to find the number of ones occurring after the first two, we will subtract the position of the first two from the length of $u$ and then will subtract $a_2-1$ to remove the rest of the twos present in the word. So $$|u|-c-(a_2-1)=a_1+a_2-a_2-c+1=a_1-c+1$$ Therefore $lb=a_1-c+1$.
        \end{proof}
        \item{Left and Smaller Statistic}
        $$ls=\sum\limits^m_{i=1}(i-1)a_i$$
        See case 1 for proof.
        \item{Right and Bigger Statistic}
        $$rb=(d-a_2)(m-1)+(|u|-d)(m-2)+\sum\limits^{m-1}_{i=2}(m-i)a_i$$
        \begin{proof} Let $w$ contain a prefix $u$ of ones and two's in any order followed by a weakly increasing suffix. Also let $d$ be the position of the last $2$ in $u$. Then there will be $d-a_2$ ones that come before the last two, each of these ones will have $m-1$ letters that are larger and to the right. Then there are $(|u|-d)$ ones that come after that last two. These ones will have $m-2$ letters that are larger and occur to the right. Then every letter between two and $m-1$ will have $m-i$ letters that are  to the right and bigger of each letter $i$. Therefore $rb=(d-a_2)(m-1)+(|u|-d)(m-2)+\sum\limits^{m-1}_{i=2}(m-i)a_i$.
        \end{proof}
        \item{Right and Smaller Statistic}
        $$rs=f-a_1$$
        \begin{proof} Let $w$ contain a prefix $u$ of ones and two's in any order followed by a weakly increasing suffix. Also let $f$ be the position of the last one in $u$. Then there will be $f-a_1$ two's that occur before the last one, meaning that each of these two's will have one distinct letter to the right and smaller. Therefore $rs=f-a_1$.
        \end{proof}

    \end{itemize}
    \end{itemize}

\eqOne*
\begin{proof}Let $\sigma_1$ avoid 13/2/4 and let $\sigma_2$ avoid $1/24/3$. Also let $w$ denote the restricted growth function for $\sigma_1$ and let $y$ denote the restricted growth function for $\sigma_2$. Let $A_1, A_2, A_3\subset R_n(13/2/4)$ such that $A_1$ contains all $w$ where $m\leq 2$ or $w$ weakly increasing, $A_2$ contains all $w=uv$ where $v$ is a block of $m$ followed by a block of $z\leq (m-1)$, and $A_3$ contains all $w=uv$ where $v$ contains $m$ and $(m-1)$ in any order after the first $m$. Also let $B_1, B_2, B_3\subset R_n(1/24/3)$ where $B_1$ contains all $w$ where $m\leq 2$ or $w$ weakly increasing, $B_2$ contains all $w$ with a block of $1$s inserted between two distinct letters in a weakly increasing word, and finally $B_3$ contains all $w=uv$ where $u$ contains $1$s and $2$s in any order allowable by the RGF.
\begin{enumerate}
	\item Let $m\leq 2$ for both $w$ and $y$. Refer to REU paper and the distribution of statistics for the avoidance class based on $1/2/3$.
    \item Let $w=uv$ where $u$ is weakly increasing and $v$ begins with the first $m$ in $w$ and then contains a block of $m$'s followed by a block of $z$'s where $z\leq (m-1)$. When $y$ is a weakly increasing word with a block of ones inserted between two distinct terms.
                \begin{itemize}
                	\item{Left and Smaller of $w$ with Right and Bigger of $y$}
                    $$ls(w)=(z-1)b_z+\sum\limits^m_{i=2}(i-1)a_i$$
                    $$rb(y)=(m-x)b_1+\sum\limits^{m-1}_{i=1}(m-i)a_i$$
                    Let $w\in A_2$, $y\in B_2$. Then the following are equations for the statistics of individual words:
                    $$ls(w)=(z-1)b_z+\sum\limits^m_{i=2}(i-1)a_i$$
                    $$rb(y)=(m-x)b_1+\sum\limits^{m-1}_{i=1}(m-i)a_i$$
                    Now let the function $f:A_2\rightarrow B_2$ where when $w=uv$, $f(w)=y$ such $b_z=b_1$ (where $b_z$ is the number of $z\in v$ and $b_1$ is the size of the block of inserted $1$s in $y$), the distinct letter $x\in y$ before this block of $1$s is $x=m-z+1$ where $z$ is the last letter in $w$, and the weakly increasing base of $y$ is formed by taking the RGF of the complement of the partition in a bijection with $u$. An example would be $f(1234422)=1123114$.
                    \par We will prove that $f$ is injective. Consider $w_1, w_2\in A_2$ where $w_1\neq w_2$. If $a_1\neq a_2$ then $a^c_1\neq a^c_2$ so the weakly increasing word in $f(w_1)$ differs from the weakly increasing word in $f(w_2)$ so $f(w_1)\neq f(w_2)$. If $b_{z_1}\neq b_{z_2}$ then $b_{1_{w_1}}\neq b_{2_{w_2}}$ so $f(w_1)\neq f(w_2)$. Therefore when $f(w_1)=f(w_2)$, $w_1=w_2$, meaning that $f$ is injective.
                    \par Let $f^{-1}$ be the inverse of $f$ such that $f^{-1}:B_2\rightarrow A_2$, where $f(y)=w$ such that $w$ is formed by taking the RGF of the complement of the partition corresponding to the weakly increasing base of $y$. This will be weakly increasing by Lemma 2.6. The function will define $z\in w$ by $z=m-x+1$ and $b_1=b_z$. 
                    \par We will show that $f^{-1}$ is injective. Consider $y_1, y_2\in B_2$ where $y_1\neq y_2$. If the weakly increasing word in $y_1$ differs from the weakly increasing word in $y_2$ then $a_{1_{y_1}}\neq a_{1_{y_2}}$ so $f^{-1}(y_1)\neq f^{-1}(y_2)$. If $b_{1_{y_1}}\neq b_{1_{y_2}}$, then  $f^{-1}(y_1)\neq f^{-1}(y_2)$. Therefore when $f^{-1}(y_1)=f^{-1}(y_2)$, $y_1=y_2$, so $f^{-1}$ is injective. Therefore $f$ is bijective. 
                   \par Now consider 
                    $$rb[f(w)]=(m-x)b_1+\sum\limits^{m-1}_{i=1}(m-i)c_i$$
                    $$=[m-(m-z+1)]b_z+\sum\limits^{m-1}_{i=1}(m-i)a_{m-i+1}$$
                    $$=(z-1)b_z+\sum\limits^{m-1}_{i=1}[(m+1-i)-1]a_{m-i+1}$$
                    $$=(z-1)b_z+\sum\limits^m_{i=2}(i-1)a_i=ls(w)$$
                    Therefore because $rb(f(w))=ls(w)$ and $f$ is a bijective function, $LS(A_2)\sim RB(B_2)$.
                    \item{Right and Bigger of $w$ with Left and Smaller of $y$.}
                    Now consider, (using the equations for statistics from Theorem 4.5 and Theorem 4.6) 
                    $$ls(f(w))=\sum\limits^m_{i=1}(i-1)c_i$$
                    $$=\sum\limits^m_{i=1}(i-1)a_{m-i+1}$$
                    $$=\sum\limits^m_{i=1}[m-(m-i+1)]a_{m-i+1}$$
                    $$=\sum\limits^{m-1}_{i=1}(m-i)a_i=rb(w)$$
                    Therefore because $f$ is a bijection from $A_2$ to $B_2$ and $rb(w)=ls(f(w))$, $RB(A_2)\sim LS(B_2)$.
                    \item{Left and Bigger.}
                     Now consider, $lb(f(w)$. From Theorems 4.5 and 4.6 we have,
                    $$lb(f(w))=(x-1)b_1$$
                    $$=[(m-z+1)-1]b_z$$
                    $$=(m-z)b_z=lb(w)$$
                    Therefore $LB(A_2)\sim LB(B_2)$ because $f$ is a bijective function from $A_2$ to $B_2$ and $lb(f(w))=lb(w)$. 
                    \item{Right and Smaller}
                    $$rs(w)=\sum\limits^m_{i=z+1}b_i$$
                  	$$rs(y)=\sum\limits^x_{i=2}\tilde{a}_i$$
                    Let $w\in A_2$, $y\in B_2$. Now let the function $h:A_2\rightarrow B_2$ where $h$ takes some $w=abc$ of length $n$ and with $m$ distinct letters and $a$ is the weakly increasing prefix that ends with the last occurrence of $z-1$ where $a$ corresponds to $\tilde{c}\in y$ such that $\tilde{c}$ is the complement of $a$, because $a$ is weakly increasing $\tilde{c}$ is weakly increasing by the lemma. The first occurrence of $z$ in $w$ marks the beginning of $b\in w$, where the subword $b$ ends with the last occurrence of $m$ where $b$ corresponds to $\tilde{a}$ which is the prefix of the word $y$ where $z$ becomes one, $z+1$ becomes two and so on such that each element from $b$ becomes $i-z+1$ in $y$ such that $a$ ends with $m-z+1$. The suffix of $w$, $c$, is the block of $z$ that occur after the last $m$ in $w$. The function $h$ will cause the size of $c$ to correspond to the size of $\tilde{b}$ where $\tilde{b}$ is the block of inserted ones after the last occurrence of $m-z+1$ in $y$.
                    \\Now consider $w_1, w_2\in A_2$ where $w_1\neq w_2$. If $a_1\neq a_2$, then $a^c_1\neq a^c_2$ meaning that $\tilde{c}_1\neq \tilde{c}_2$, therefore $h(w_1)\neq h(w_2)$. If $b_1\neq b_2$, then the number of $i-z+1$ will be different between $\tilde{a}_1$ and $\tilde{a}_2$ where $i$ has a different number of occurrences in $b_1$ than in $b_2$, therefore If $h(w_1)\neq h(w_2)$. If $c_1\neq c_2$, then $\tilde{b}_1\neq\tilde{b}_2$ because they will be of different sizes, therefore $h(w_1)\neq h(w_2)$. Therefore if $h(w_1)=h(w_2)$, then $w_1=w_2$. So $h$ is one-one.
                    \\Let $h^{-1}:B_2\rightarrow A_2$ where $h^{-1}$ is the inverse of $h$. Where $h^{-1}(y)=w$ where $y=\tilde{a}\tilde{b}\tilde{c}$ such that $\tilde{a}$, the weakly increasing prefix of $y$ that ends with the last occurrence of $x$ the last letter before a block of inserted ones, corresponds to $b\in w$ where $1$ becomes $z$ and every element $i\in \tilde{a}$ becomes $i+z-1$ in $b$, so the last element in $\tilde{a}$, $m-z+1$ becomes $(m-z+1)+z-1=m$ in $b$. Then $\tilde{b}$, the block of inserted ones, corresponds to $c\in w$ where $c$ is the same length as $\tilde{b}$ however composed of a block of $z$ where $z$ is determined to be $m-x+1$. Finally $\tilde{c}$ the weakly increasing suffix of $y$ corresponds to the weakly increasing prefix of $w$, $a$, where $\tilde{c}^c=a$.
                    \\Now consider $y_1, y_2\in B_2$ where $y_1\neq y_2$. If $\tilde{a}_1\neq \tilde{a}_2$, then for each $i$ where there are a different number of occurrences of $i$ in $\tilde{a}_1$ versus $\tilde{a}_2$ there will be a different number of occurrences of $i+z-1$ in $b_1$ versus $b_2$, therefore $h^{-1}(y_1)\neq h^{-1}(y_2)$. If $\tilde{b}_1\neq \tilde{b}_2$ then the lengths must not be the same, so the length of $c_1$ and $c_2$ will also differ, meaning that $h^{-1}(y_1)\neq h^{-1}(y_2)$. If $\tilde{c}_1\neq\tilde{c}_2$ then $\tilde{c}_1^c\neq\tilde{c}_2^c$ meaning that $a_1\neq a_2$, therefore $h^{-1}(y_1)\neq h^{-1}(y_2)$. So when $h(w_1)=h(w_2)$, then $w_1=w_2$. So $h^{-1}$ is one-one.
                    \\Because $h$ and $h^{-1}$ are both one-one, then $h$ is a bijective function between $A_2$ and $B_2$. 
                    Therefore $x=m-z+1$ and has the same number of occurrences as $a_m$ the number of $m$ in $w$. The word $y$ ends such that the weakly increasing suffix of $y$ is created from the complement of the weakly increasing prefix of $w$ that occurs before the first $z$. (An example of the function would be $h(12234522)=11234115$). Now the same number of elements larger than $z$ will occur before $z$ as the number of elements larger than one that occur before the block of inserted ones. This function $h$ is reversible, and the inverse $h^{-1}:A_2\rightarrow B_2$ takes some word $y\in B_2$ such that $h^{-1}(y)$ will be the word created by $z=m-x+1$, all letters after the block of ones inserted will start the word, such that $m$ in $y$ becomes $1$ in $h^{-1}(y)$ and so on until $m-x$ followed by $z$, $z+1$ and so on until the block of $m$'s followed by a block of $z$'s corresponding in length to the block of ones in $y$. Because $h$ is reversible, $h$ is a bijective function between $A_2$ and $B_2$. Now consider,
                    $$rs(h(w))=\sum\limits^x_{i=2}\tilde{a}_i$$
                    $$=\sum\limits^m_{i=z+1}b_i=rs(w)$$
                    Therefore because $h$ is bijective, and $rs(h(w))=rs(w)$, $RS(A_2)\sim RS(B_2)$.
                \end{itemize}
                \item Let $w$ be of the form $uv$ where $u$ is weakly increasing and $v$ begins with the first $m$ and then contains $m$ and $(m-1)$ in any order and let $y$ be of the form $uv$ where $u$ contains $1$ and $2$ in any reasonable order allowed by the definition of the restricted growth function, followed by $v$ which begins with the first occurrence of $3$ and is a weakly increasing suffix.
                \begin{itemize}
                	\item{Left and Smaller of $w$ with Right and Bigger of $y$}
                    
                    Let $w\in A_3$ where $A_3$ is the set of words in case 3 of the avoidance class for the pattern $13/2/4$. Let $y\in B_3$ where $B_3$ is the set of all words in case 3 of the avoidance class for the pattern $1/24/3$. Let the function $g: A_3\mapsto B_3$. The function $g(w)$ will take a word $w=abc$ where $a$ is weakly increasing from one to $m-2$, $b$ is the subword of the $m-1$'s that occur after $a$, followed by $c$ which begins with the first occurrence of $m$ followed by $m$ and $m-1$ in any order. Then let $g(w)=y$ where $y=\tilde{a}\tilde{b}\tilde{c}$ such that $\tilde{a}$ begins with one followed by one's and two's in any order ending one position before the last occurrence of two, $\tilde{a}$ is found by the function $g$ such that $c\in w$ corresponds to $\tilde{a}\in y$ such that $m\in c$ becomes $1\in \tilde{a}$ and $(m-1)\in c$ becomes $2\in \tilde{a}$. The next subword $\tilde{b}\in y$ is formed from $b\in y$. The first $(m-1)$ in $b$ will become $2$, which is found in the first position of $b'$ and is the last two in $y$. All other $(m-1)\in b$ will become ones $\in \tilde{b}$ which occur after the last two in $y$. Therefore there are $|b|-1$ one's that occur after the last two. Finally the last subword $\tilde{c}\in y$ is formed by taking the complement of the weakly increasing subword $a\in w$, by the lemma the complement of a weakly increasing word is weakly increasing, and by the definition of $a$ where $a$ ends with the last occurrence of $m-2$, the complement of $m-2$ is $m-(m-2)+1=3$ so $\tilde{c}$ is weakly increasing beginning with the first occurrence of $3$. Because $y$ will have a prefix of ones and twos in any order which composed of the two subwords $\tilde{a}$ and $\tilde{b}$ followed by a weakly increasing suffix $\tilde{c}$ which begins with the first occurrence of $3$, $y$ is a word in $B_3$.  
                    \\Now consider two words $w_1, w_2 \in A_3$ where $w_1\neq w_2$. If $a_1\neq a_2$ then $a'_1=\tilde{c_1}$ and $a'_1=\tilde{c_2}$, and $a'_1\neq a'_2$ so $g(w_1)\neq g(w_2)$. If, however, $b_1\neq b_2$ then $|b_1|\neq |b_2|$ and because $|b_1|=|\tilde{b}|$ and $|b_2|=|\tilde{b_2}|$, then $|\tilde{b_1}|\neq |\tilde{b_2}$ so $\tilde{b_1}\neq \tilde{b_2}$, therefore $g(w_1)\neq g(w_2)$. Finally consider if $c_1\neq c_2$. Then by the definition of the function $g$ the ones and twos in $\tilde{a_1}$ will not correspond to the position of the ones and twos in $\tilde{a_2}$, so $g(w_1)\neq g(w_2)$. So when $g(w_1)=g(w_2)$, $w_1=w_2$. Therefore $g$ is one-one.
                    \\Now consider $g^{-1}: B_3\mapsto A_3$. The function $g^{-1}(y)=w$ where $y=\tilde{a}\tilde{b}\tilde{c}$ will result in some $w\in A_3$ where  $w=abc$ such that by taking the complement of the weakly increasing suffix $\tilde{c}\in y$ we create $a\in w$. By the lemma, the complement of a weakly increasing word is weakly increasing, so $a$ is weakly increasing, and because $\tilde{c}$ begins with the first occurrence of $3$, $a$ will end in the last occurrence of $m-3+1=m-2$. The subword $\tilde{b}\in y$ where $\tilde{b}$ begins with the last two and then contains a block of ones ending with the last one in the word will correspond to the subword $b\in w$ such that every element in $\tilde{b}$ corresponds to an $(m-1)$ in $b$ where $|\tilde{b}|=|b|$. Finally the subword $\tilde{a}\in y$ where $\tilde{a}$ begins with one and then contains ones and twos in any order ending before the last occurrence of two which corresponds to the subword $c\in w$ where each one in $\tilde{a}$ corresponds to an $m$ in $c$ and each two in $\tilde{a}$ corresponds to an $(m-1)$ in $c$. Therefore $g^{-1}(y)$ has a weakly increasing prefix composed of $a$ and $b$ followed by $m$ and $(m-1)$ in any order after the first occurence of $m$ which is the subword $c$, meaning that $g^{-1}(y)$ is in $A_3$.
                    \\Now consider two words $y_1, y_2\in B_3$ where $y_1\neq y_2$. If $\tilde{a_1}\neq \tilde{a_2}$ then $c_1\neq c_2$ because of the definition of $g^{-1}$ the $m$'s and $(m-1)$'s in $g^{-1}(y_1)$ and $g^{-1}(y_2)$ will not be found in the same locations and there may be different quantities of each. Therefore $g^{-1}(y_1)\neq g^{-1}(y_2)$. If $\tilde{b_1}\neq \tilde{b_2}$ then $b_1\neq b_2$ because they will not be of the same length. Therefore $g^{-1}(y_1)\neq g^{-1}(y_2)$. Finally consider if $\tilde{c_1}\neq\tilde{c_2}$, then $\tilde{c_1}'\neq\tilde{c_2}'$ meaning $a_1\neq a_2$. Therefore $g^{-1}(y_1)\neq g^{-1}(y_2)$!!! So when $g^{-1}(y_1)=g^{-1}(y_2)$, then $y_1=y_2$. Therefore $g^{-1}$ is one-one. 
                    Because both $g$ and $g^{-1}$ are one-one, $g$ is a bijective function between $A_3$ and $B_3$. 
                    
                    Now consider,
                    $$rb(g(w))=(m-1)(\tilde{a_1})+(m-2)(\tilde{a_2}+|\tilde{b}|)+\sum\limits^{m-1}_{i=3}(m-i)\tilde{c_i}$$
                    Then by the definition of $g$, $$rb(g(w))=(m-1)a_m+(m-2)(a_{m-1}+b_{m-1})+\sum\limits^{m-1}_{i=3}[m-(m-i+1)]\tilde{c}_{m-i+1}$$
                    $$=(m-1)a_m+(m-2)(a_{m-1}+b_{m-1})+\sum\limits^{m}_{i=3}(i-1)a_i=ls(w)$$
                    Because $g$ is a bijective function and $rb(g(w))=ls(w)$, $LS(A_3)\sim RB(B_3)$.
                    \item{Right and Bigger of $w$ with Left and Smaller of $y$.} 
                   
                    Let $w\in A_3$ such that $w=abc$, $a$ is the weakly increasing prefix of $w$ which ends at the last occurrence of $m-2$, $b$ begins with the first $m-1$ and ends with the last occurrence of $m$ in the word, and finally, $c$ is a block of $m-1$ that occur after the last $m$ in the word. Also let $y\in B_3$ such that $y=\tilde{a}\tilde{b}\tilde{c}$ where $\tilde{a}$ is the prefix of one's and two's beginning with one and ending with the last occurrence of two, $\tilde{b}$ is a block of ones after the last two in the word, and $\tilde{c}$ begins with the first three and is a weakly increasing suffix for $w$. Now let $h:A_3\rightarrow B_3$ where $h(w)=y$  such that the complement of $a$ becomes $\tilde{c}$ and by the lemma this will ensure that $\tilde{c}$ is a weakly increasing suffix because $a$ is weakly increasing also $\tilde{c}$ will begin with a three because $m-(m-2)+1=3$, $h$ will cause the $m-1$'s in $c$ to become ones in $\tilde{b}$, and $h$ takes $b$ where $m$ becomes one and each $m-1$ becomes a two and they are found in reverse order in the prefix of $y$, $\tilde{a}$ so $\tilde{a}$ begins with one and ends with the last two in $y$ because $b$ contains the first occurrence of $m-1$ which corresponds to the last two in $y$.
                    \\Now consider two words $w_1, w_2\in A_3$ such that $w_1\neq w_2$. If $a_1\neq a_2$ then $a'_1\neq a'_2$ by the definition of the complement, so $\tilde{c}_1\neq \tilde{c}_1$, meaning that $h(w_1)\neq h(w_2)$. If $b_1\neq b_2$ then the number, order, or both of the one's and two's in $\tilde{a}_1$ will differ from that of $\tilde{a}_2$ because the number, order, or both of $m$'s and $(m-1)$'s differ in $b_1$ and $b_2$, therefore $h(w_1)\neq h(w_2)$. If $c_1\neq c_2$, then $|c_1|\neq |c_2|$, meaning that $|\tilde{b}_1|\neq |\tilde{b}_2|$, therefore $h(w_1)\neq h(w_2)$!!! Therefore when $h(w_1)=h(w_2)$, then $w_1=w_2$. This means that $h$ is a one-one function.
                    \\Let $h^{-1}$ be the inverse function of $h$, such that $h^{-1}:B_2\rightarrow A_2$. This function will cause $\tilde{a}\in y$ to correspond to $b\in w$ where the ones in $\tilde{a}$ correspond to the $m$'s in $b$, and the two's in $\tilde{a}$ correspond to $(m-1)$'s in $b$ such that each element is found in reverse order from that of $\tilde{a}$ in $b$. Then $\tilde{b}$, the block of ones after the last two, corresponds to $c\in w$ the block of $(m-1)$'s after the last $m$. Finally $\tilde{c}$, the weakly increasing suffix of $w$ which begins with a three, corresponds to the weakly increasing prefix $a\in w$ where $\tilde{c}^c=a$.
                    \\Now consider two words $y_1, y_2\in B_3$ where $y_1\neq y_2$. If $\tilde{a}_1\neq\tilde{a}_2$, then the number and/or order of one's and two's differs between these two subwords, because the number and order of $m$ and $m-1$ in $b_1$ and $b_2$ will depend on the order and number of $\tilde{a}_1$ and $\tilde{a}_2$ respectively, $b_1\neq b_2$, therefore $h^{-1}(y_1)\neq h^{-1}(y_2)$. If $\tilde{b}_1\neq\tilde{b}_2$, $|\tilde{b}_1|\neq|\tilde{b}_2|$ so $|c_1|\neq|c_2|$, therefore $h^{-1}(y_1)\neq h^{-1}(y_2)$. If $\tilde{c}_1\neq\tilde{c}_2$, then $\tilde{c}^c_1\neq\tilde{c}^c_2$ so $a_1\neq a_2$, therefore $h^{-1}(y_1)\neq h^{-1}(y_2)$!!! Therefore when $h^{-1}(y_1)=h^{-1}(y_2)$, $y_1=y_2$. So $h^{-1}$ is one-one.
                    Because $h$ and $h^{-1}$ are both one-one, the function $h$ is bijective between $A_3$ and $B_3$.
                    The $m$ terms in $w$ correspond to ones in the prefix of $y$ before the last two, all $z$'s before the last $m$ in $w$ correspond to twos in the word $y$, and finally all $z's$ after $x$ the position of the last $m$ in $w$ correspond to ones after the last two in $y$. All $z$ that occur before the last $m$A will be equal to $x-a_m$ where $x$ is the position of the last $m$ in the suffix of $w$ beginning with the first $z$. This corresponds to the total number of two's in $y$. The weakly increasing suffix of $y$ is the complement of the prefix of $w$. (An example is $h(122345544)=112113445$.) This function is reversible, the inverse, $h^{-1}:B_3\rightarrow A_3$ will take some $y\in B_3$ such that $g^{-1}(y)=w$ where $w$ is the word created by taking the complement of the weakly increasing suffix of $y$ as the prefix in $w$, the one's in the prefix of $y$ before the last two become $m$ in $w$, all two's in $y$ correspond to $z$'s before the last $m$ in $w$, the number of one's after the last two in $y$ corresponds to the $z$'s that occur after the last $m$ in $w$. Now consider,
                    $$ls(h(w))=\tilde{a}_2+\sum\limits^m_{i=3}(i-1)\tilde{c}_i$$
                    $$=b_{m-1}+\sum\limits^{m-2}_{i=1}[(m-i+1)-1]\tilde{c}_{m-i+1}$$
                    $$=b_{m-1}+\sum\limits^{m-2}_{i=1}(m-i)a_i=rb(w)$$
                    Therefore because $h$ is a bijective function between $A_3$ and $B_3$ and $ls(h(w))=rb(w)$, $RB(A_3)\sim LS(B_3)$.
                    \item{Left and Bigger.}

                    Let $w\in A_3$ and $y\in B_3$ where $w=ab$ such that $a$ is the weakly increasing prefix ending with the last occurrence of $m-2$, and $b$ is the suffix of $m-1$ and $m$ in any order beginning with the first occurrence of $m-1$ in $w$ and $y=\tilde{a}\tilde{b}$ such that $\tilde{a}$ is the prefix of ones and twos, followed by $\tilde{b}$ the weakly increasing suffix beginning with the first three in $y$. Let $g: A_3\mapsto B_3$ where $g(w)$ will cause $a\in w$ to correspond to $\tilde{b}$ where $a^c=\tilde{b}$ and by the lemma $\tilde{b}$ will be weakly increasing. The suffix $b$ in $w$ corresponds to $\tilde{a}\in y$ where each $m-1$ in $b$ becomes $1$ found in the same position of $b$ as in $\tilde{a}$, $m\in b$ becomes $2$ in the prefix. So the position of the first two in $\tilde{a}$, $c$, is equal to the position of the first $m$ in the suffix of $w$, $b$. 
                    \\Now consider $w_1, w_2\in A_3$ where $w_1\neq w_2$. If $a_1\neq a_2$ then $a^c_1\neq a^c_2$ so $\tilde{b}_1\neq\tilde{b}_2$, therefore $g(w_1)\neq g(w_2)$. If $b_1\neq b_2$ then the number and order of $m$'s and $(m-1)$'s will differ, and the corresponding $\tilde{a}_1$ and $\tilde{a}_2$ will differ in the number and order of one's and two's, therefore $g(w_1)\neq g(w_2)$. Therefore if $g(w_1)=g(w_2)$, $w_1=w_2$ meaning that $g$ is one-one.
                    \\Let $g^{-1}$ be the inverse function where $g^{-1}:B_3\rightarrow A_3$. The subword $\tilde{a}$ corresponds to $b\in w$ where one's become $m-1$ in $b$, and twos become $m$'s. The subword $\tilde{b}$ corresponds to $a\in w$ where $\tilde{b}^c=a$ which will be weakly increasing by the lemma.
                    \\Now consider $y_1, y_2\in B_3$ where $y_1\neq y_2$. If $\tilde{a}_1\neq\tilde{a}_2$ then the ones and twos in these two subwords differ in position and potentially the quantity, so the position and potentially the quantity of $(m-1)$'s and $m$'s will differ between $b_1$ and $b_2$, therefore $g^{-1}(y_1)\neq g^{-1}(y_2)$. If $\tilde{b}_1\neq\tilde{b}_2$ then $\tilde{b}^c_1\neq\tilde{b}^c_2$ meaning that $a_1\neq a_2$ so $g^{-1}(y_1)\neq g^{-1}(y_2)$. So when $g^{-1}(y_1)=g^{-1}(y_2)$, $y_1=y_2$. Therefore $g^{-1}$ is one-one.
                    Because $g$ and its inverse $g^{-1}$ are one-one, $g$ is a bijection between $A_3$ and $B_3$.
                    Now let $d$ be the position of the first $m$ in $b$, and $c$ the position of the first two in $\tilde{a}$. Then
                    $$lb(g(w))=|\tilde{a}|-c-\tilde{a}_2+1=|b|-d-a_m+1=lb(w)$$
                    Because $g$ is a bijective function between $A_3$ and $B_3$ and $lb(g(w))=lb(w)$, $LB(A_3)\sim LB(B_3)$.
                    \item{Right and Smaller}
                    $$rs(w)=t-b_{m-1}$$
                   	$$rs(y)=f-\tilde{a}_1$$
                    Where $f$ is the position of the last one in $u$.
                    Let $g$ be the function defined in the last case (the one for $LB(A_3)\sim LB(B_3)$). Let $t$ be the position of the last occurrence of $m-1$ in $b\in w$ and let $f$ be the position of the last occurrence of one in $\tilde{a}\in y$
                    $$rs(g(w))=f-a_1$$
                    $$=t-b_{m-1}=rs(w)$$
                    Therefore because $g$ is a bijective function between $A_3$ and $B_3$ and $rs(g(w))=rs(w)$, $RS(A_3)\sim RS(B_3)$.
                \end{itemize}
          
           Because $A_1\cup A_2\cup A_3=R_n(13/2/4)$ and $B_1\cup B_2\cup B_3\cup=R_n(1/24/3)$ and 
           $LS(A_1)\sim RB(B_1)$, $RB(A_1)\sim LS(B_1)$, $RS(A_1)\sim RS(B_1)$, $LB(A_1)\sim LB(B_1)$, $LS(A_2)\sim RB(B_2)$, $RB(A_2)\sim LS(B_2)$, $RS(A_2)\sim RS(B_2)$, $LB(A_2)\sim LB(B_2)$, $LS(A_3)\sim RB(B_3)$, $RB(A_3)\sim LS(B_3)$, $RS(A_3)\sim RS(B_3)$, $LB(A_3)\sim LB(B_3)$,
           $$LS(13/2/4)\sim RB(1/24/3)$$
           $$RB(13/2/4)\sim LS(1/24/3)$$
           $$RS(13/2/4)\sim RS(1/24/3)$$
           $$LB(13/2/4)\sim LB(1/24/3)$$ 
\end{enumerate}
\end{proof}       

\TwoSeven*
\begin{proof} Let $w\in R_n(1/23,13/2/4)$. Then $w\in R_n(1/23)$ so by Sagan in Lemma 2.1, we know that $m\geq 1$ and that $w$ is formed by inserting a $1$ into a word of the form $1^l23...m$ \cite{Sag09}. We also know that $w\in R_n(13/2/4)$ so $w$ can take the following forms by Theorem 3.9:
\begin{itemize}
	\item[a. ] We can have $m\leq 2$. Then either $m=1$ or when $m=2$ by Lemma 2.1 $w=1^l21$ or $w=1^l2$. 
    \item[b. ] Because $w\in R_n(13/2/4)$ we could have $w=ab$ where $a$ is weakly increasing and $b$ begins with the first $m$ and then contains $m$ and $m-1$ in any order. By Theorem 3.9 we know $w$ avoids $13/2/4$ however by Lemma 2.1 we know that $w$ contains $1/23$. Therefore $w$ cannot take this form.
    \item[c. ] Finally let $w=ab$ where $a$ is weakly increasing and $b$ begins with the first $m$ and contains a block of $m$s followed by a block of $z$s where $z/leq(m-1)$. By Theorem 3.9, $w\in R_n(13/2/4)$, however by Lemma 2.1 we know $z=1$ and there can only be one occurrence of $1$ in $b$. Also by Lemma 2.1, $a$ begins with a block of $1$s followed by a strictly increasing subword. 
\end{itemize}
Therefore by Lemma 2.1 and Theorem 3.9, when $w\in R_n(1/23,13/2/4)$, $w=abc$ where $a$ is a prefix of $1$s, $b$ is a strictly increasing subword beginning with $2$ and ending with $m$ where $b$ can be empty, followed by $c$ which is either empty or contains a $1$. 
\end{proof}
$$\#\Pi_n (1/23,13/2/4)=2(n-1)$$
\begin{proof}First let $w=ab$ where $a$ is a prefix of $1$s and $b$ is strictly increasing beginning with the first $2$ where $b$ can be empty. Then $|a|$ can be anything from one to $n$, and because $b$ is strictly increasing, there are a total of $n$ ways to form $w$.
\par Next let $w=abc$ where $a$ is a prefix of $1$s, $b$ is strictly increasing from $2$ to $m$, and $c$ is one occurrence of $1$. Then $|a|$ can vary from one to $n-2$. Because $b$ is strictly increasing and $c$ is a suffix containing exactly one $1$, there is only one way to form $bc$. So there are $n-2$ ways to form $w$ when $w=abc$.

Therefore $\#\Pi_n(1/23,13/2/4)=n+n-2=2(n-1)$
\end{proof}
\FourFour*
\begin{proof} Let $w\in R_n(13/2,123/4)$. Then $w\in R_n(13/2)$ and $w\in R_n(123/4)$. Then because $w\in R_n(13/2)$, $w$ must be weakly increasing by Sagan \cite{Sag09}. Also $w$ must be of the form $ab$ where $a$ is an RGF ending with $m$ where no element is repeated more than twice and $b$ is a block of some element $x$ of any size by the fact that $w\in R_n(123/4)$ from Theorem 3.6. Because $w$ must be weakly increasing, $x=m$. Then $w=ab$ where $a$ is weakly increasing with no element repeated more than twice and $b$ is a block of $m$'s of any length. 
\end{proof}
$$\#\Pi_n (13/2,123/4)=1+\sum^{n-1}_{j=0}\sum^{\lfloor\frac{n-j}{2}\rfloor}_{i=0}{{n-j-i}\choose i}$$
\begin{proof} Consider the enumeration of the avoidance class. We have just shown that $w=ab$. When $m=1$ then there is one way to form $w$ by the Sterling numbers of the second kind. Next consider when $1<m\leq n$. Because $1<m\leq n$, the block of $m$s, beginning with the first occurrence of $m$ can vary between the length of one and $n-1$. Next consider how many of every other element can occur. Every element $1$ through $m-1$ must occur at least once in $a$. There are ${n-a_m-i}\choose{i}$ ways for us to choose $i$ letters that get to be repeated in the word where $a_m$ is the total number of $m\in w$, because when we want to repeat $i$ letters we will have $m-1=n-a_m-i$, so there are $n-a_m-i$ choices for these $i$ letters. Our choice of $i$ can vary from $0$ to half of $n-a_m$ when $n-a_m$ is even because every element from $1$ to $m-1$ could be repeated, or $0$ to half of $n-a_m-1$ when $n-a_m$ is odd, because then every element except one between $1$ and $m-1$ could be repeated in this case. Therefore our choice of $i$ can vary from $0$ to $\lfloor\frac{n-a_m}{2}\rfloor$. Therefore $\#\Pi_n (13/2,123/4)=1+\sum^{n-1}_{j=0}\sum^{\lfloor\frac{n-j}{2}\rfloor}_{i=0}{{n-j-i}\choose i}$ where $j=a_m$.
\end{proof}
\FourThree*
\begin{proof} Let $w\in R_n(13/2,1/234)$. Then $w\in R_n(13/2)$ and $w\in R_n(1/234)$. By Sagan we know that because $w\in R_n(13/2)$ then $w$ is weakly increasing \cite{Sag09}. By Jonathan Bloom $w=ab$ where $a$ is a block of $1$'s of any length, and $b$ begins with $2$ and is in any reasonable order for an RGF and no element in $b$ is repeated more than twice. Because $w\in R_n(13/2,1/234)$ $w$ must be weakly increasing where no element except one can be repeated more than twice. 
\end{proof}
$$\#\Pi_n (13/2,1/234)=1+\sum^{n-1}_{j=0}\sum^{\lfloor\frac{n-j}{2}\rfloor}_{i=0}{{n-j-i}\choose i}$$
\begin{proof}Consider the enumeration of the avoidance class. When $m=1$ there is one way to order $w$. When $1<m\leq n$, $a_1$, or the number of $1$s present in $w$ will vary between one and $n-1$ when $w$ is strictly increasing. Next let $i$ denote the number of letters between $2$ and $m$ that are repeated. Because $w$ is a weakly increasing word, as we find the cardinality order will not be a concern. We know that the length of $w$ following the block of $1$s is $n-a_1$ and that if there are $i$ letters repeated, that there are $n-a_m-i$ choices for those repeated letters. Therefore there are ${n-a_i-i}\choose{i}$ ways to form the weakly increasing suffix of $w$ for every $i$ where $i$ can vary from $0$, when no letters are repeated, to $\frac{n-a_1}{2}$ when $n-a_m$ is even and every letter in the suffix is repeated, or the upper limit for $i$ will be $\frac{n-a_1-1}{2}$ when $n-a_m$ is odd and every letter except one is repeated in the suffix. Therefore $0\leq i\leq \lfloor\frac{n-a_1}{2}\rfloor$. Therefore $\#\Pi_n (13/2,1/234)=1+\sum^{n-1}_{j=0}\sum^{\lfloor\frac{n-j}{2}\rfloor}_{i=0}{{n-j-i}\choose i}$ where $j=a_1$.
\end{proof}

\SevenFive*
\begin{proof} Let $w\in R_n(13/2/4,134/2)$. Then $w\in R_n(13/2/4)$ and $w\in R_n(134/2)$. Then by Theorem 3.9 $w$ must be of one of these forms: 
\begin{itemize}
\item[1.] $m\leq 2$
\item[2.] $w=ab$ where $a$ is weakly increasing ending with the last m in the first block of $m$'s in $w$ and $b$ begins with $m-1$ and contains $m$ and $m-1$ in any order.
\item[3.] $w=ab$ where $a$ is weakly increasing, $b$ is a block of some $x<(m-1)$.
\end{itemize}
By theorem 3.7, because $w\in R_n(134/2)$ then any element $x\in w$ where $x$ repeated more than twice will have a block of $x$'s followed later in the word with a maximum of one more occurrence of $x$. 
Then there are three possibilities for $w\in R_n$.
\begin{itemize}
\item[a)] Let $w$ have $m\leq 2$. Then for $w$ to be in $R_n(134/2)$, $w$ could just contain all $1s$, $w$ could be weakly increasing, or when $w$ is not weakly increasing, then $1$ could be the only block-singleton, or $1$ and $2$ are block-singletons. 
\item[b)] Let $w=ab$ where $b$ begins with $m-1$ and contains $m$ and $m-1$ in any order. Then because $w\in R_n(134/2)$, $|b|=1$ such that $w$ is a weakly increasing word followed by one occurrence of $m-1$, or $(m-1)$ followed by $m$.
\item[c)] Let $w=ab$ where $b$ is a block of some element $x<m-1$. Then by theorem 3.7, $|b|=1$ where the only occurrence in $b$ is the element $x<m-1$.
\end{itemize}
\end{proof}
$\#\Pi_n(13/2/4,134/3)=2n-5+\sum\limits^n_{m=1}{{n-1}\choose{n-m}}+\sum\limits_{|a|=3}^{n-1}\sum\limits_{m=3}^{|a|}{{|a|-1}\choose{|a|-m}}(m-1)+ \sum\limits_{|a|=1}^{n-4}\sum\limits_{m=3}^{|a|+2}{{|a|-1}\choose{|a|-(m-2)}}(n-|a|-2)$.
\begin{proof} Let $w\in R_n(13/2/4,134/2)$. Then $w$ can take the following forms:
\begin{itemize}
	\item[(i)] Let $w$ be weakly increasing and $1\leq m\leq n$ by Lemma 2.2, there are $\sum\limits^n_{m=1}{{n-1}\choose{n-m}}$ possibilities for $w$.
    \item[(ii)] Let $m=2$ and only $1$ be a block-singleton letter. Then there are three assigned positions. We know that the first letter is $1$ and the second letter is $2$ and the last letter is also $1$. Therefore the length of the block of $1$s can vary from one to $n-2$. So there are $n-2$ possible words for $w$.
    \item[(ii)] Let $m=2$ and $1$ and $2$ are block-singleton letters. Then there are $n-3$ possible words of this form.
    \item[(iii)] Let $w=ab$ where $m\geq 3$, $b$ is a block of $z$ where $z$ is a block-singleton, and $a$ is weakly increasing. Then $3\leq m\leq |a|$, and $3\leq |a|\leq (n-1)$. Also note that $1\leq z\leq (m-1)$ so there are $(m-1)$ options for $z$. Then there will be $\sum\limits_{|a|=2}^{n-1}\sum\limits_{m=3}^{|a|}{{|a|-1}\choose{|a|-m}}(m-1)$ possible words of this form.
    \item{(iv)} Finally let $w=ab$ where $b=(m-1)^lm^k(m-1)m$ and $a$ is weakly increasing. Then we know that $1\leq |a|\leq (n-4)$, and so $3\leq m\leq |a|+2$ and there will be $(n-|a|-2)$ ways to form $b$. Therefore there are a total of $\sum\limits_{|a|=1}^{n-4}\sum\limits_{m=3}^{|a|+2}{{|a|-1}\choose{|a|-(m-2)}}(n-|a|-2)$ possible words of this form.
\end{itemize}
Therefore $\#\Pi_n(13/2/4,134/3)=2n-5+\sum\limits^n_{m=1}{{n-1}\choose{n-m}}+\sum\limits_{|a|=3}^{n-1}\sum\limits_{m=3}^{|a|}{{|a|-1}\choose{|a|-m}}(m-1)+ \sum\limits_{|a|=1}^{n-4}\sum\limits_{m=3}^{|a|+2}{{|a|-1}\choose{|a|-(m-2)}}(n-|a|-2)$.
\end{proof}
\NewEight*
\begin{proof}Let $w\in R_n(14/2/3,1/24/3)$. Then by Theorem 3.10 we know that,
\begin{itemize} 
	\item[i)] $w$ could have $m\leq 2$. Because $1/24/3$ has three blocks, and any $w$ of this form will be associated with a partition of only two blocks, $w\in R_n(1/24/3)$.
    \item[ii)] Next consider when $w=ab$ where $a$ is a prefix of $1$s and $2$s in any legal RGF order followed by a weakly increasing suffix $b$ which begins with the first $3$. Then because there are never two or more distinct letters between two of the same letters in $w$, $w\in R_n(14/2/3)$.
    \item[iii)]Finally $w$ could be a weakly increasing word with a block of $1$s inserted between two distinct letters. Then $w\not\in R_n(14/2/3)$ because the the subword of the first $1$, $2$, and $3$, and first $1$ of the second block of $1$s will be associated with the standardized subpartition $14/2/3$. Therefore $w$ cannot take this form.
\end{itemize}

\end{proof}
$$\#\Pi_n (14/2/3,1/24/3)=2^{n-1}+\sum\limits_{|b|=1}^{n-2}\sum\limits^{|b|}_{i=1}{{|b|-1}\choose{|b|-i}}(2^{n-|b|}-1)$$
\begin{proof} Let $w\in R_n(14/2/3,1/24/3)$. Then by the word characterization we have:
\begin{itemize}
	\item[i)]When $w$ has $m\leq 2$ we know that there are $2^{n-1}$ possible words of this form.
    \item[ii)] When $w=ab$ where $a$ is a prefix of $1$s and $2$s in any order and $b$ is weakly increasing beginning with the first occurrence of $3$, we know that $w$ could be weakly increasing, and that there are a total of $2^{n-|b|}-1$ possible forms for $a$. Then by Lemma 2.2 we know that there are a total of $\sum\limits_{|b|=1}^{n-2}\sum\limits_{i=1}^{|b|}{{|b|-1}\choose{|b|-i}}(2^{n-|b|}-1)$ possible $w$ where $i=m-1$.
\end{itemize}
Therefore $\#\Pi_n (14/2/3,1/24/3)=2^{n-1}+\sum\limits_{|b|=1}^{n-2}\sum\limits^{|b|}_{i=1}{{|b|-1}\choose{|b|-i}}(2^{n-|b|}-1)$.
\end{proof}

\NewSeven*
\begin{proof} Let $w\in R_n(14/2/4,13/2/4)$. Then by Theorem 3.9 we know the following,
\begin{itemize}
	\item[i)] The first form $w$ could take is when $m\leq 2$, then from part (i) of the previous proof $w\in R_n(14/2/3)$.
    \item[ii)] Next consider when $w=ab$ where $a$ is weakly increasing and $b$ contains a block of $z<(m-1)$. This $w\not\in R_n(14/2/3)$ because if we consider the subword associated with the first $z$, the first $(m-1)$, and the first $m$ along with the first occurrence of $z\in b$, this is associated with the standardized subpartition $14/2/3$.
    \item[iii)] Finally consider when $w=ab$ where $a$ is weakly increasing and $b$ contains $m$ and $m-1$ in any legal RGF order. Then there are never two or more distinct letters between two of the same letter so $w\in R_n(14/2/3)$.
\end{itemize}
\end{proof}
$$\#\Pi_n (14/2/3,13/2/4)=2^{n-1}+\sum\limits_{|a|=1}^{n-2}\sum\limits^{|a|}_{i=1}{{|a|-1}\choose{|a|-i}}(2^{n-|a|}-1)$$
\begin{proof} Let $w\in R_n(14/2/3,13/2/4)$. Then by the word characterization we have:
\begin{itemize}
	\item[i)]When $w$ has $m\leq 2$ we know that there are $2^{n-1}$ possible words of this form.
    \item[ii)] When $w=ab$ where $a$ is a weakly increasing prefix and $b$ contains $(m-1)$ and $m$ in any order, we know that $w$ could be weakly increasing, and that there are a total of $2^{n-|a|}-1$ possible forms for $b$. Then by Lemma 2.2 we know that there are a total of $\sum\limits_{|a|=1}^{n-2}\sum\limits_{i=1}^{|a|}{{|a|-1}\choose{|a|-i}}(2^{n-|a|}-1)$ possible $w$ where $i=m-1$.
\end{itemize}
Therefore $\#\Pi_n (14/2/3,13/2/4)=2^{n-1}+\sum\limits_{|a|=1}^{n-2}\sum\limits^{|a|}_{i=1}{{|a|-1}\choose{|a|-i}}(2^{n-|a|}-1)$.
\end{proof}

\EightFive*
\begin{proof} By Theorem 3.7 and Theorem 3.10.
\end{proof}
$$\#\Pi_n(1/24/3,134/2)=2+\sum_{m=2}^{n-1}({n-1\choose n-m}+{n-2\choose n-m-1}(m-1))$$
\begin{proof} Let us consider the cardinality of the avoidance class. As we have shown there are two cases we need to consider.
\begin{itemize}
	\item[(i)] The first case is when $w$ is weakly increasing. Then as shown in previous proofs (find for reference here), there will be ${n-1\choose n-m}$ different words for each possible $m$ where we know that $1\leq m\leq n$. Therefore when $w$ is weakly increasing there are $\sum_{m=1}^n{n-1\choose n-m}=2+\sum_{m=2}^{n-1}{n-1\choose n-m}$.
    \item[(ii)] The second, and only other option, is that $w=ab$ where $a$ is weakly increasing, followed by $b$, where $|b|=1$ and $b$ contains only the letter $z\leq(m-1)$. First consider the different possibilities for the weakly increasing prefix. We know that $|a|=n-1$ because $|b|=1$, and we know that when we have a weakly increasing word there are ${n-2\choose n-m-1}$ for some fixed $m$ where $2\leq m\leq n-1$. Then there are also $(m-1)$ options for $z$ because $z\leq(m-1)$. Therefore there are $\sum_{m=2}^{n-1}{n-2\choose n-m-1}(m-1)$ possible words of the form $w=ab$. 
   
\end{itemize}
Therefore $\#\Pi_n(1/24/3,134/2)=2+\sum_{m=2}^{n-1}({n-1\choose n-m}+{n-2\choose n-m-1}(m-1))$.
\end{proof}

\EightSix*
\begin{proof} Let $w\in R_n(1/24/3)$, then by Theorem 3.10 we know that 
\begin{itemize} 
	\item[i)] We have $m\leq 2$, then $w\in R_n(1/23/4)$ also because $w$ is associated with a partition of only two blocks.
    \item[ii)] Next $w=ab$ where $a$ is a prefix of $1$s and $2$s in any order followed by the weakly increasing suffix $b$. However for $w\in R_n(1/23/4)$ we must have $b$ strictly increasing except for $m$, and there must only be one occurrence of $2\in a$, and either one $1$ after $2$ or the entire word is increasing.
    \item[iii)]Finally $w$ is a weakly increasing word with a block of $1$s inserted. This block of $1$s must be of length one unless the block occurs at the end of the word for $w\in R_n(1/23/4)$ and the only other letter than can be repeated is $1$, or if the block of inserted $1$s occurs somewhere else, $m$ can also be repeated.
\end{itemize}
\end{proof}
$$\#\Pi_n(1/24/3,124/3)=2+\sum_{m=2}^{n-1}({n-1\choose n-m}+{n-2\choose n-m-1}(m-1))$$
\begin{proof} Let $w\in R_n(1/24/3,124/3)$. Then there are several possibilities.
\begin{itemize}
	\item[i)] We have $w$ where $w$ is weakly increasing. Then by Lemma 2.2 there are $\sum\limits_{m=1}^{n}{{n-1}\choose{n-m}}=2+\sum\limits_{m=2}^{n-1}{{n-1}\choose{n-m}}$.
    \item[ii)] Otherwise $w=ab$ where $b$ contains a block of $z$ where $z$ a singleton-block in $w$ and $a$ is weakly increasing. Then $1\leq z\leq (m-1)$ and there are $m$ blocks to assign $n-m-1$ letters. So by multisets we have that there are $\sum\limits_{m=2}^{n-1}{{n-2}\choose{n-m-1}}(m-1)$ ways to order $w$.
\end{itemize}
Therefore $\#\Pi_n(1/24/3,124/3)=2+\sum_{m=2}^{n-1}({n-1\choose n-m}+{n-2\choose n-m-1}(m-1))$.
\end{proof}
\SevenSag*
\begin{proof} Let $w\in R_n(13/2/4)$. Then by Theorem 3.9,
\begin{itemize}
	\item[i)] First we can consider when $m\leq 2$ which also indicates that $w\in R_n(13/2/4)$.
    \item[ii)] Next we can consider when $w=ab$ where $a$ is weakly increasing and $b$ is some letter $z$ that $z\leq m$, then for $w\in R_n(1/23/4)$ we also know that only $1$ and $z$ can be repeated.
    \item[iii)] Finally consider when $w=ab$ such that $a$ is weakly increasing and $b$ contains $m$ and $m-1$ in any order (such that $b$ is not the same as $b$ from part (ii)). Then $w\not\in R_n(1/23/4)$ because we can take the subword of the first $1$, the first $m-1$, the next $m-1$ and the $m$ that must occur after these two (m-1) terms. The standardized subpartition associated with this subword is $1/23/4$.
\end{itemize}
\end{proof}
$$\#\Pi_n (13/2/4,1/23/4)=1+2^{n-1}+\sum\limits_{m=3}^{n-1}(nm-m^2)$$
\begin{proof} Let $w\in R_n(13/2/4,1/23/4)$ where $m\leq 2$. There are $2^{n-1}$ possible $w$ of this form.
\par Next let $w=abc$ where $a$ is a block of $1$s, $b$ a strictly increasing word from $2$ to $m$ and $c$ a potentially empty block of some letter $z\leq m$. Then when $z=m$ we know that $3\leq m\leq n$ and that there are $n-m+1$ options for the length of the block of $1$s which also defines the length of the block of $m$s. So when $z=m$ there are $\sum\limits_{m=3}^{n}(n-m+1)=1+\sum\limits_{m=3}^{n-1}(n-m+1)$ possible $w$ of this form. Next consider when $z\leq (m-1)$. Then $3\leq m\leq (n-1)$, there are $m-1$ possible $z$, and the block of $1$s can vary in length from one to $n-m$. Therefore there are $\sum\limits_{m=3}^{n-1}(n-m)(m-1)$ possible $w$ of this form. 
\par Therefore $\#\Pi_n(13/2/4,1/23/4)=1+2^{n-1}+\sum\limits_{m=3}^{n-1}(n-m+1+nm-m^2-n+m)=1+2^{n-1}+\sum\limits_{m=3}^{n-1}(nm-m^2)$
\end{proof}
\EightFive*
\begin{proof} First let $w\in R_n(1/24/3,134/2)$. Then $w\in R_n(1/24/3)$ and $w\in R_n(134/2)$. By Theorem 3.7 and Theorem 3.10, $w$ can either be weakly increasing, or if there is a block-singleton present it must be $1$.

\par Next let us consider the cardinality of the avoidance class. As we have shown there are two cases we need to consider.
\begin{itemize}
	\item[(i)] The first case is when $w$ is weakly increasing. Then by Lemma 2.2  there are $\sum_{m=1}^n{n-1\choose n-m}=2+\sum_{m=2}^{n-1}{n-1\choose n-m}$ possible forms of $w$.
    \item[(ii)] The second, and only other option, is is weakly increasing except for a block-singleton $1$, however the only difference here is that there are $(m-1)$ options for the distinct letter that occurs before the singleton $1$ and there are only $n-1$ weakly increasing letters. Therefore there are $\sum_{m=2}^{n-1}{n-2\choose n-m-1}(m-1)$ possible words. 
   
\end{itemize}
Therefore $\#\Pi_n(1/24/3,134/2)=2+\sum_{m=2}^{n-1}({n-1\choose n-m}+{n-2\choose n-m-1}(m-1))$.
\end{proof}

\multistat*

\begin{proof} The first two formulas for $ls(w)$ and $rb(w)$ are given by Lemma 2.1, 2.2. Now let $w\in R_n(1/23, 13/2/4)$. Then by Theorem 6.3 there are two possibilities for the form of $w$. The first is when $w=ab$ where $a$ is a block of ones and $b$ is strictly increasing. Then $w$ is weakly increasing so $lb(w)=0$ and $rs(w)=0$. The other possibility is that $w=abc$ where $a$ is a block of $1$'s, $b$ is strictly increasing, and $c$ is a $1$ at the end of the word. Then because $ab$ this portion of the word will not contribute to $lb(w)$. Now $w$ will have a one after $ab$ so there will be $m-1$ elements that are larger and occur before this $1$. Therefore when $w=abc$, $lb(w)=m-1$. Now consider the right and smaller statistic. Every element in $b$ will be larger than the $1$ in $c$, so each will have exactly one contribution to $rs(w)$. Therefore when $w=abc$, $rs(w)=|b|$.
\end{proof}

\end{document}